\documentclass[11pt,leqno]{article}

\usepackage{amsthm,amsfonts,amssymb,amsmath,oldgerm}
\usepackage{fullpage}
\usepackage{graphicx}
\usepackage{mathrsfs}
\usepackage{xcolor}
\numberwithin{equation}{section}
\usepackage{enumitem}
\usepackage{hyperref}
\usepackage{subcaption}
\usepackage[nocompress]{cite}

\makeatletter
\def\namedlabel#1#2{\begingroup
    #2%
    \def\@currentlabel{#2}%
    \phantomsection\label{#1}\endgroup
}
\makeatother

\def\eps {\varepsilon}

\newcommand{\R}{\mathbb R}
\newcommand{\C}{\mathbb C}

\newcommand{\vt}{\widetilde{\mathbf{w}}}
\newcommand{\ri}{\mathrm{i}}
\newcommand{\re}{\mathrm{e}}
\newcommand{\de}{\mathrm{d}}
\newcommand{\ub}{\mathbf{u}}
\newcommand{\vb}{\mathbf{w}}
\newcommand{\gb}{\mathbf{g}}
\newcommand{\zb}{\bar{\zeta}}
\newcommand{\zt}{\mathbf{z}}
\renewcommand{\Re}{\mathrm{Re}}
\renewcommand{\Im}{\mathrm{Im}}
\newtheorem{theorem}{Theorem}[section]
\newtheorem{proposition}[theorem]{Proposition}
\newtheorem{corollary}[theorem]{Corollary}
\newtheorem{lemma}[theorem]{Lemma}
\newtheorem{remark}[theorem]{Remark}
\theoremstyle{definition}

\allowdisplaybreaks[3]

\newcommand{\xx}{{\zeta}}
\newcommand{\xt}{{\bar{\zeta}}}
\newcommand{\NT}{\widetilde{\mathcal{N}}}
\newcommand{\Non}{\mathcal{N}}
\newcommand{\El}{\mathcal{L}}

\usepackage{caption}

\title{Nonlinear stability of periodic wave trains in the FitzHugh-Nagumo system against fully nonlocalized perturbations}
\author{Joannis Alexopoulos$^*$, and Bj\"orn de Rijk\thanks{Department of Mathematics, Karlsruhe Institute of Technology, Englerstra\ss e 2, 76131 Karlsruhe, Germany; \texttt{joannis.alexopoulos@kit.edu}, \texttt{bjoern.de-rijk@kit.edu}}}

\begin{document}
\maketitle

\begin{abstract}
Recently, a nonlinear stability theory has been developed for wave trains in reaction-diffusion systems relying on pure $L^\infty$-estimates. In the absence of localization of perturbations, it exploits diffusive decay caused by smoothing together with spatio-temporal phase modulation. In this paper, we advance this theory beyond the parabolic setting and propose a scheme designed for general dissipative semilinear problems. We present our method in the context of the FitzHugh-Nagumo system. The lack of parabolicity and localization complicates mode filtration in $L^\infty$-spaces using the Floquet-Bloch transform. Instead, we employ the inverse Laplace representation of the semigroup generated by the linearization to uncover high-frequency damping, while leveraging a novel link to the Floquet-Bloch representation for the smoothing low-frequency part. Another challenge arises in controlling regularity in the quasilinear iteration scheme for the modulated perturbation. We address this by extending the method of nonlinear damping estimates to nonlocalized perturbations using uniformly local Sobolev norms.

\paragraph*{Keywords.} Periodic waves; nonlinear stability; fully nonlocalized perturbations; FitzHugh-Nagumo system; inverse Laplace transform; Floquet-Bloch analysis; uniformly local Sobolev spaces;  Cole-Hopf transform\\
\textbf{Mathematics Subject Classification (2020).} 35B10; 35B35; 35K57; 44A10
\end{abstract}

\section{Introduction}

We study the nonlinear stability of traveling periodic waves against bounded, fully nonlocalized perturbations in the FitzHugh-Nagumo (FHN) system
\begin{align}
\label{FHN}
\begin{split}
\partial_t u &= u_{xx} + u(1-u)(u-\mu) - v,\\
\partial_t v &= \varepsilon(u-\gamma v - \mu),
\end{split}
\end{align}
with $x \in \mathbb{R}, t\geq 0$ and parameters $\mu \in \R$ and $\gamma, \varepsilon > 0$. The FHN system was originally proposed
as a simplification of the Hodgkin-Huxley model describing signal propagation in nerve fibers~\cite{FIT,murray,NAG}. Mathematically, system~\eqref{FHN} is a coupling between a scalar bistable reaction-diffusion equation and a linear ordinary differential equation and is thereby one of the simplest\footnote{We note that Sturm-Liouville theory implies that all periodic traveling waves in real scalar reaction-diffusion equations are unstable.} models, which can, and does, exhibit stable spatio-temporal patterns. In fact, exploiting the slow-fast structure of system~\eqref{FHN} arising for $0 < \varepsilon \ll 1$, a large variety of (spectrally) stable patterns and nonlinear waves have been rigorously constructed using tools from geometric singular perturbation theory, such as fast traveling pulses~\cite{HAS,Jones,JOE,Yan}, pulses with oscillatory tails~\cite{CARISA,CART}, periodic wave trains~\cite{CarterScheel,Eszter,SOT} and pattern-forming fronts~\cite{CarterScheel} connecting such pulse trains to the homogeneous rest state $(\mu,0)$.

Due to its remarkably rich dynamics, yet simple structure, the FHN system is widely recognized as a paradigmatic model for far-from-equilibrium patterns in excitable and oscillatory media. It has, in small variations, been employed across various scientific disciplines to explain phenomena such as the onset of turbulence in fluids~\cite{Barkley}, oxidation processes on platinum surfaces~\cite{Oxidation1,Oxidation2}, and heart arrhythmias~\cite{Cardiac}.

The simplest and most fundamental spatio-temporal patterns exhibited by~\eqref{FHN} are periodic traveling waves, or \emph{wave trains}. Writing~\eqref{FHN} as a degenerate reaction-diffusion system
\begin{align} \label{FHNsys}
\partial_t \mathbf{u} = D \mathbf{u}_{xx} + F(\mathbf{u}), \qquad D = \begin{pmatrix}
    1 & 0 \\ 0 & 0
\end{pmatrix}, \qquad F(\mathbf{u}) = \begin{pmatrix}
    u(1-u)(u-\mu) - v \\ \varepsilon(u-\gamma v - \mu)
\end{pmatrix},
\end{align}
in $\mathbf{u} = (u,v)^\top$, wave trains are solutions to~\eqref{FHNsys} of the form $\mathbf{u}_0(x,t) = \phi_0(x-c_0 t)$ with smooth periodic profile function $\phi_0 \colon \R \to \R^2$ and propagation speed $c_0 \in \R$. Upon switching to the co-moving frame $\zeta = x - c_0 t$, in which system~\eqref{FHNsys} reads 
\begin{align} \label{FHN_co}
\partial_t \mathbf{u} = D \mathbf{u}_{\zeta\zeta} + c_0 \mathbf{u}_\zeta + F(\mathbf{u}),
\end{align}
we find that $\phi_0$ is a stationary solution to~\eqref{FHN_co}.

Wave-train solutions to~\eqref{FHNsys} have been constructed in the oscillatory regime with $0 < \mu < \frac12$ and $0 < \varepsilon \ll \gamma \ll 1$, as well as in the excitable regime with $\mu < 0$ and $0 < \varepsilon \ll \gamma \ll 1$, using geometric singular perturbation theory and blow-up techniques, see~\cite{CarterScheel,SOT} and Remark~\ref{scaling}. The associated profile functions consist of steep jumps interspersed with long transient states, where the profile varies slowly. Accordingly, these wave trains correspond to highly nonlinear far-from-equilibrium patterns. It has recently been argued theoretically and demonstrated numerically~\cite{CarterScheel} that some of these wave trains are selected by compactly supported perturbations of the unstable rest state $(\mu,0)$ in the oscillatory regime and, thus, play a pivotal role in pattern formation away from onset.

In this paper, we focus on the dynamical, or nonlinear, stability of wave trains as solutions to~\eqref{FHNsys}. The nonlinear stability theory for wave trains in spatially extended dissipative problems such as~\eqref{FHNsys} has been rapidly developing over the past decades. The general approach is to first linearize the system about the wave train, obtain bounds on the $C_0$-semigroup generated by the linearization and then close a nonlinear argument by iterative estimates on the associated Duhamel formulation. A standard issue is that the linearization is a periodic differential operator acting on an unbounded domain, which possesses continuous spectrum touching the imaginary axis at the origin due to translational invariance. The lack of a spectral gap prevents, in contrast to the case of a finite domain with periodic boundary conditions, exponential convergence of the perturbed solution towards a translate of the original profile. 

To overcome this issue a common strategy is to decompose the semigroup generated by the linearization in a diffusively decaying low-frequency part and an exponentially damped high-frequency part, cf.~\cite{JONZ}. The critical diffusive behavior caused by translational invariance can then be captured by introducing a spatio-temporal phase modulation, whose leading-order behavior is given by a viscous Hamilton-Jacobi equation~\cite{DSSS}. The modulated perturbation obeys a quasilinear equation depending only on \emph{derivatives} of the phase modulation, which thus satisfy a perturbed Burgers' equation. Observing that small, sufficiently localized initial data in a (perturbed) viscous Burgers' equation decay diffusively, cf.~\cite[Theorem~1]{UECN} or~\cite[Theorem~4]{BKL}, suggests that the critical dynamics in a nonlinear iteration scheme, tracking the modulated perturbation variable and derivatives of the phase, can be controlled. This observation has led to a series of nonlinear stability results of wave trains against localized perturbations in general (nondegenerate) reaction-diffusion systems~\cite{SAN3,JONZ,JONZNL,JUN,JUNNL} relying on renormalization group theory~\cite{SAN3}, pointwise estimates~\cite{JUN,JUNNL} or $L^1$-$H^k$-estimates~\cite{JONZ,JONZNL} to close the nonlinear iteration. We note that, since only derivatives of the phase enter in the nonlinear iteration and thus need to be localized, one could allow for a nonlocalized phase modulation, cf.~\cite{SAN3,JONZNL,JUNNL}. With the aid of periodic-coefficient damping estimates to obtain high-frequency resolvent bounds and control regularity in the quasilinear iteration scheme, the method employing $L^1$-$H^k$-estimates could be extended beyond the parabolic setting to general dissipative semilinear problems (and some quasilinear problems) such as the St.~Venant equations~\cite{STVenant1,STVenant2}, the Lugiato-Lefever equation~\cite{LLEperiod,ZUM22} and the FHN system~\cite{BjoernAvery}. 

Recently, a novel approach was developed~\cite{HDRS22,BjoernMod} to establish nonlinear stability of wave trains in (nondegenerate) reaction-diffusion systems, which employs pure $L^\infty$-estimates to close the nonlinear iteration, thereby lifting all localization assumptions on perturbations. In contrast to previous methods, diffusive decay cannot be realized by giving up localization, but emanates from smoothing action of the analytic semigroup generated by the linearization about the wave train. The Cole-Hopf transform is then applied to the equation for the phase to eliminate the critical Burgers'-type nonlinearity, which cannot be readily controlled by diffusive smoothing.

In this paper, we extend the approach developed in~\cite{HDRS22,BjoernMod} beyond the parabolic framework by proving nonlinear stability of wave trains in the FHN system~\eqref{FHNsys} against $C_{\mathrm{ub}}$-perturbations. The incomplete parabolicity of~\eqref{FHNsys} in combination with lack of localization of perturbations presents novel challenges in our analysis. These challenges involve the decomposition of the $C_0$-semigroup and the control of regularity. We explain the main ideas on how to address these challenges in~\S\ref{sec:strategy} after we have stated our main result in~\S\ref{section_main_result}.

\begin{remark} \label{scaling}
{\upshape 
Let $\mu < 0$, $\gamma \geq 0$ and $\varepsilon > 0$, so that we are in the excitable regime. Upon rescaling time, space, the variables $u$ and $v$, and the system parameters $\varepsilon$, $\mu$ and $\gamma$ by setting
\begin{align*}
\tilde{x} = (1-\mu) x, \quad \tilde{t} = (1-\mu)^2 t, \quad \tilde{u} = \frac{u - \mu}{1-\mu}, \quad  \tilde{v} = \frac{v}{(1-\mu)^3},\\\tilde{\varepsilon} = \frac{\varepsilon}{(1-\mu)^4}, \quad \tilde{\gamma} = (1-\mu)^2\gamma, \quad \tilde{\mu} = -\frac{\mu}{1-\mu}, \qquad \qquad
\end{align*}
we arrive at the equivalent formulation
\begin{align} \label{FHNrescaled}
\begin{split}
\partial_{\tilde{t}} \tilde{u} &= \tilde{u}_{\tilde{x}\tilde{x}} + \tilde{u}(1-\tilde{u})(\tilde{u}-\tilde{\mu}) - \tilde{v},\\
\partial_{\tilde{t}} \tilde{v} &= \tilde{\varepsilon}(\tilde{u}-\tilde{\gamma} \tilde{v}),
\end{split}
\end{align}
of the FHN system~\eqref{FHN}. Here, we have $\tilde{\mu} \in (0,1)$, $\tilde{\gamma} \geq 0$ and $\tilde{\varepsilon} > 0$. We note that the formulation~\eqref{FHNrescaled} of the FHN system has been used in the existence and spectral stability analysis of wave trains and traveling pulses in the excitable regime, cf.~\cite{CARISA, Eszter,Jones, JOE, SOT, Yan}.
}    
\end{remark}

\subsection{Assumptions on the wave train and its spectrum}

Here, we formulate the hypotheses for our main result. The first hypothesis concerns the existence of the wave train.
\begin{itemize}
\item[\namedlabel{assH1}{\upshape (H1)}] There exist a speed $c_0 \in \R$ and a period $T > 0$ such that~\eqref{FHNsys} admits a wave-train solution $\ub_0(x,t)=\phi_0(x - c_0 t)$, where the profile function $\phi_0 \colon \R \to \R^2$ is nonconstant, smooth and $T$-periodic.
\end{itemize}
We note that wave-train solutions have been shown to exist, i.e.,~\ref{assH1} has been verified, in the excitable regime with $\mu < 0 \leq \gamma \ll 1$ and $0 < \varepsilon \ll 1$, cf.~\cite{SOT}, and in the oscillatory regime with $0 < \mu < \frac12$ and $0 < \varepsilon \ll \gamma \ll 1$, cf.~\cite{CarterScheel}.

Next, we specify our spectral assumptions on the wave train $\ub_0$. Linearizing~\eqref{FHN_co} about its stationary solution $\phi_0$ yields the $T$-periodic differential operator $\El_0 \colon D(\El_0) \subset C_{\mathrm{ub}}(\R) \to C_{\mathrm{ub}}(\R)$ given by
\begin{align} \label{deflinearization}
\El_0 \vb = D \vb_{\zeta\zeta} + c_0 \vb_\zeta + F'(\phi_0) \vb
\end{align}
with domain $D(\El_0) = C_{\mathrm{ub}}^2(\R) \times C_{\mathrm{ub}}^1(\R)$, where $C_{\mathrm{ub}}^m(\R)$ denotes for $m \in \mathbb{N}_0$ the space of bounded and uniformly continuous functions, which are $m$ times differentiable and whose $m$ derivatives are also bounded and uniformly continuous. We endow $C_{\mathrm{ub}}^m(\R)$ with the standard $\smash{W^{m,\infty}}$-norm, so that it is a Banach space.

The spectrum of $\El_0$ is determined by the family of Bloch operators
\begin{align*}
\El(\xi) \vb = D\left(\partial_\zeta + \ri \xi\right)^2 \vb + c_0 \left(\partial_\zeta + \ri \xi\right) \vb + F'(\phi_0)\vb, \qquad \xi \in \C
\end{align*}
posed on $L_{\mathrm{per}}^2(0,T)$ with domain $D(\El(\xi)) = H_{\mathrm{per}}^2(0,T) \times H_{\mathrm{per}}^1(0,T)$. Since $\El(\xi)$ has compact resolvent, its spectrum consists of isolated eigenvalues of finite multiplicity. The spectrum of $\El_0$ can then be characterized as
\begin{align} \label{Blochspecdecomp}
\sigma(\El_0) = \bigcup_{\xi \in \left[-\tfrac{\pi}{T},\tfrac{\pi}{T}\right)} \sigma(\El(\xi)),
\end{align}
cf.~\cite{GARD}. We require that the following standard \emph{diffusive spectral stability} assumptions, cf.~\cite{BjoernMod,JONZ,SAN3,SCHN97}, are satisfied.
\begin{itemize}
\setlength\itemsep{0em}
\item[\namedlabel{assD1}{\upshape (D1)}] We have $\sigma(\El_0)\subset\{\lambda\in\C:\Re(\lambda)<0\}\cup\{0\}$;
\item[\namedlabel{assD2}{\upshape (D2)}] There exists a constant $\theta>0$ such that for any $\xi\in[-\frac{\pi}{T},\frac{\pi}{T})$ we have $\Re\,\sigma(\El(\xi))\leq-\theta \xi^2$;
\item[\namedlabel{assD3}{\upshape (D3)}] $0$ is a simple eigenvalue of $\El(0)$.
\end{itemize}
The main result of~\cite{SpectralFHN} establishes diffusive spectral stability of wave trains in~\eqref{FHN} in the oscillatory regime $(3 - \sqrt{5})/6 < \mu < \frac12$ and $0 < \varepsilon \ll \gamma \ll 1$. On the other hand, a spectral analysis of wave trains in the excitable regime with $\mu < 0$, $\gamma = 0$ and $0 < \varepsilon \ll 1$ can be found in~\cite{Eszter}.\footnote{Although the spectral assumptions~\ref{assD1} and~\ref{assD3} are verified in~\cite{Eszter}, we emphasize that the fact that $\gamma = 0$ yields a lack of damping in the second component of~\eqref{FHN}, causing the spectrum of the linearization to asymptote to $\ri \R$ at infinity. In particular, the spectrum is not bounded away from the imaginary axis away from $0$ and the assumption~\ref{assD2} does not hold, prohibiting diffusive spectral stability.}

It is a consequence of translational invariance that $0$ is an eigenvalue of the Bloch operator $\El(0)$ with associated eigenfunction $\phi_0'$. Assumption~\ref{assD3} then states that the kernel of $\El(0)$ is spanned by $\phi_0'$. In this case $0$ is also a simple eigenvalue of the adjoint operator $\El(0)^*$. We denote by $\smash{\widetilde{\Phi}_0} \in H^2_{\mathrm{per}}(0,T) \times H_{\mathrm{per}}^1(0,T)$ the corresponding eigenfunction satisfying 
\begin{align*} 
\big\langle \widetilde{\Phi}_0,\phi_0'\big\rangle_{L^2(0,T)} = 1.\end{align*}

An application of the implicit function theorem in combination with Assumption~\ref{assD3} readily yields that the wave train can be continued with respect to the wavenumber, cf.~\cite[Section~4.2]{DSSS}. 

\begin{proposition} \label{prop:family}
Assume~\ref{assH1} and~\ref{assD3}. Then, there exists a constant $r_0 \in (0,1)$ and smooth functions $\phi \colon \R \times [1-r_0,1+r_0] \to \R^2$ and $\omega \colon [1-r_0,1+r_0] \to \R$ with $\phi(\cdot;1) = \phi_0$ and $\omega(1) = c_0$ such that $\phi(\cdot;k)$ is $T$-periodic and
$$\ub_k(x,t) = \phi(k x - \omega(k) t;k)$$
is a solution to~\eqref{FHNsys} for each wavenumber $k \in [1-r_0,1+r_0]$. By shifting the wave train if necessary, we can arrange for
\begin{align*} 
\big\langle \widetilde{\Phi}_0,\partial_k \phi(\cdot;1)\big\rangle_{L^2(0,T)} = 0.\end{align*}
\end{proposition}

The curve $\omega \colon [1-r_0,1+r_0] \to \R$ from Proposition~\ref{prop:family} describes the relationship between the temporal frequency $\omega(k)$ and the wavenumber $k$ of the $T/k$-periodic wave train $\ub_k$ and is called the \emph{nonlinear dispersion relation}. 

Because the Bloch operators $\El(\xi)$ depend analytically on the Floquet exponent $\xi$ and $0$ is a simple eigenvalue of $\El(0)$ by Hypothesis~\ref{assD3}, it follows by standard analytic perturbation theory, see e.g.~\cite{Kato}, that the $0$-eigenvalue can be continued to a simple eigenvalue $\lambda_c(\xi)$ of $\El(\xi)$ for $\xi$ close to $0$. The curve $\lambda_c(\xi)$ is analytic and necessarily touches the imaginary axis in a quadratic tangency by Hypothesis~\ref{assD2}. Using Lyapunov-Schmidt reduction, the eigenvalue $\lambda_c(\xi)$, as well as the associated eigenfunction, can be expanded in $\xi$ about $\xi = 0$, cf.~\cite[Section~4.2]{DSSS} or~\cite[Section~2]{johnson_whitham}.\footnote{For the purpose of our current analysis, it suffices to expand the eigenvalue $\lambda_c(\xi)$ up to second order and the associated eigenvector up to first order. We refer to Remark~\ref{rem:expansion} for further details.} We record these facts in the following result.

\begin{proposition} \label{prop:speccons}
Assume~\ref{assH1} and~\ref{assD1}-\ref{assD3}. There exist a constant $C>0$, open balls $V_1,V_2 \subset \C$ centered at $0$ and an analytic function $\lambda_c \colon V_1 \to \C$ such that the following assertions hold. 
\begin{itemize}
\setlength\itemsep{0em}
\item[(i)] $\lambda_c(\xi)$ is a simple eigenvalue of $\El(\xi)$ for each $\xi \in V_1$. An associated eigenfunction $\Phi_\xi$ of $\El(\xi)$ lies in $H_{\mathrm{per}}^m(0,T)$ for any $m \in \mathbb{N}_0$, is analytic in $\xi$, satisfies $\Phi_0 = \phi_0'$ and fulfills
\begin{align*}
 \big\langle \widetilde{\Phi}_0,\Phi_\xi\big\rangle_{L^2(0,T)} = 1
\end{align*}
for $\xi \in V_1$. 
\item[(ii)] It holds $\sigma(\El_0) \cap V_2 = \{\lambda_c(\xi) : \xi \in V_1 \cap \R\} \cap V_2$.
\item[(iii)] The complex conjugate $\overline{\lambda_c(\xi)}$ is a simple eigenvalue of the adjoint $\El(\xi)^*$ for any $\xi \in V_1$. An associated eigenfunction $\widetilde{\Phi}_\xi$ lies in $H_{\mathrm{per}}^m(0,T)$ for any $m \in \mathbb{N}_0$, is smooth in $\xi$ and satisfies
\begin{align*}
 \big\langle \widetilde{\Phi}_\xi,\Phi_\xi\big\rangle_{L^2(0,T)} = 1
\end{align*}
for $\xi \in V_1$. 
\item[(iv)] We have 
\begin{align*}
\lambda_c'(\xi) = 2 \ri \big\langle \widetilde{\Phi}_\xi, D\left(\partial_\zeta + \ri \xi\right) \Phi_\xi\big\rangle_{L^2(0,T)} + \ri c_0 
\end{align*}
and the expansions
\begin{align} \label{e:eig_exp}
\left|\lambda_c(\xi) + \ri c_g\xi + d \xi^2\right| \leq C |\xi|^3, \qquad \left\|\Phi_\xi - \phi_0' - \ri \xi \partial_k \phi(\cdot;1)\right\|_{H^m(0,T)} \leq C |\xi|^2,
\end{align}
hold for $\xi \in V_1$ with coefficients 
\begin{align} \label{coefflindisp}
\begin{split}
c_g &= - 2 \big\langle \widetilde{\Phi}_0,D \phi_{0}''\big\rangle_{L^2(0,T)} - c_0 = \omega'(1) - c_0 \in \R, \\ d &= \big\langle \widetilde{\Phi}_0, D\phi_{0}' + 2 D\partial_{\zeta k} \phi(\cdot;1)\big\rangle_{L^2(0,T)} > 0.
\end{split}
\end{align}
\end{itemize}
\end{proposition}

The function $\lambda_c$ in Proposition~\ref{prop:speccons} is called the \emph{linear dispersion relation}. The coefficient $c_g$ in~\eqref{e:eig_exp} is the \emph{group velocity} of the wave train and provides the speed at which perturbations are transported along the wave train (in the frame moving with the speed $c_0$), cf.~\cite{DSSS}. We make the generic assumption that the wave train has nonzero group velocity. By reversing space $x \to -x$ in~\eqref{FHNsys} we may then without loss of generality assume that the group velocity is negative. 
\begin{itemize}
\item[\namedlabel{assH2}{\upshape (H2)}] Assuming, in accordance with Hypothesis~\ref{assD3}, that $0$ is a simple eigenvalue of $\El(0)$, the group velocity $c_g$, defined in~\eqref{coefflindisp}, is negative.
\end{itemize}

On the linear level, the interpretation of Assumptions~\ref{assD1}-\ref{assD3} and~\ref{assH2} is that perturbations decay diffusively and are transported to the left along the wave train, i.e.,~there is an \emph{outgoing diffusive mode} at the origin, cf.~\cite[Section~2.1]{BjoernAvery}. In~\cite{CarterScheel}, it was shown that the group velocity of the wave trains is negative in the oscillatory regime with $0 < \mu < \frac12$ and $0 < \varepsilon \ll \gamma \ll 1$. 

Another important consequence of Assumption~\ref{assH2} is that the linear dispersion relation $\lambda_c$ is invertible in the point $\xi = 0$. Hence, for $|\lambda|$ sufficiently small, the periodic eigenvalue problem $(\El_0 - \lambda) \vb = 0$ has a single Floquet exponent converging to $0$ as $\lambda \to 0$. In our stability analysis we exploit this fact to relate the inverse Laplace representation of the low-frequency part of the semigroup generated by $\El_0$ with the Floquet-Bloch representation, see~\S\ref{sec:FloquetBloch}.

\subsection{Main result}
\label{section_main_result}

We are now ready to present our main result, which establishes Lyapunov stability of diffusively spectrally stable wave trains in the FHN system against $C_{\mathrm{ub}}$-perturbations. Furthermore, it yields convergence of the perturbed solution towards a modulated wave train, where the phase modulation can be approximated by a solution of a viscous Hamilton-Jacobi equation.

\begin{theorem} \label{t:mainresult}
Assume~\ref{assH1},~\ref{assH2} and~\ref{assD1}-\ref{assD3}. Fix a constant $K > 0$. Then, there exist constants $\alpha,\epsilon_0, M > 0$ such that, whenever $\vb_0 \in C_{\mathrm{ub}}^3(\R) \times C_{\mathrm{ub}}^2(\R)$ satisfies
\[
E_0 := \left\|\vb_0\right\|_{L^\infty} < \epsilon_0, \qquad \left\|\vb_0\right\|_{C_{\mathrm{ub}}^3 \times C_{\mathrm{ub}}^2} < K,
\]
there exist a smooth function $\psi \in C^\infty\big([0,\infty) 
\times \R,\R\big)$ with $\psi(0) = 0$ and $\psi(t) \in C_{\mathrm{ub}}^m(\R)$ for each $m \in \mathbb{N}_0$ and $t \geq 0$ and a unique classical global solution 
\begin{align} \label{e:regu}
\ub \in C\big([0,\infty),C_{\mathrm{ub}}^3(\R) \times C_{\mathrm{ub}}^2(\R)\big) \cap C^1\big([0,\infty),C_{\mathrm{ub}}^1(\R)\big)
\end{align}
to~\eqref{FHN_co} with initial condition $\ub(0) =\phi_0 + \vb_0$, which obey the estimates
\begin{align} 
\label{e:mtest10}
\left\|\ub(t)-\phi_0\right\|_{L^\infty} &\leq ME_0, \\
\label{e:mtest11}
\left\|\ub(t)-\phi_0(\cdot+\psi(\cdot,t))\right\|_{L^\infty} &\leq \frac{ME_0}{\sqrt{1+t}},\\
\label{e:mtest12}
\left\|\ub(t)-\phi_0\left(\cdot + \psi(\cdot,t)\left(1+\psi_\xx(\cdot,t)\right);1+\psi_\xx(\cdot,t)\right)\right\|_{L^\infty} &\leq M E_0 \frac{\log(2+t)}{1+t}
\end{align}
and
\begin{align}  \label{e:mtest2}
\begin{split}
\|\psi(t)\|_{L^\infty} \leq ME_0, \qquad \left\|\psi_\xx(t)\right\|_{L^\infty}&, \|\partial_t \psi(t)\|_{L^\infty}
\leq \frac{ME_0}{\sqrt{1+t}}, \\ 
\|\psi_{\xx\xx}(t)\|_{C_{\mathrm{ub}}^4}, \|\partial_t \psi_{\xx}(t)\|_{C_{\mathrm{ub}}^3} &\leq ME_0 \frac{\log(2+t)}{1+t}
\end{split}
\end{align}
for all $t \geq 0$. Moreover, there exists a unique classical global solution $\breve{\psi} \in C\big([0,\infty),C_{\mathrm{ub}}^2(\R)\big) \cap \smash{C^1\big([0,\infty),C_{\mathrm{ub}}(\R)\big)}$ with initial condition $\smash{\breve{\psi}(0)} = \smash{\widetilde{\Phi}_0^*\vb_0}$ of the viscous Hamilton-Jacobi equation
\begin{align}
\partial_t \breve\psi = d\breve\psi_{\xx\xx} - c_g \breve\psi_{\xx} + \nu \breve\psi_\zeta^2 \label{e:HamJac}
\end{align}
with coefficients~\eqref{coefflindisp} and
\begin{align} \label{e:nuexpress}
\begin{split}
\nu =  -\tfrac{1}{2}\omega''(1) &= \big\langle \widetilde{\Phi}_0, D\left(\phi_0'' + 2 \partial_{\zeta kk} \phi(\cdot;1)\right) + \tfrac12 F''(\phi_0)\big(\partial_k \phi(\cdot;1), \partial_k \phi(\cdot;1)\big)\big\rangle_{L^2(0,T)}\\
&\qquad - 2 \big\langle \widetilde{\Phi}_0,D \phi_{0}''\big\rangle_{L^2(0,T)} \big\langle 
\widetilde{\Phi}_0, \partial_{\zeta k} \phi(\cdot;1)\big\rangle_{L^2(0,T)},
\end{split}
\end{align}
satisfying
\begin{align}
\label{e:mtest3}
t^{\frac{j}{2}}\left\|\partial_\xx^j \left(\psi(t) - \breve{\psi}(t)\right)\right\|_{L^\infty} \leq M E_0\left(E_0^\alpha + \frac{\log(2+t)}{\sqrt{1+t}}\right)
\end{align}
for $j = 0,1$ and $t \geq 0$.
\end{theorem}

We compare Theorem~\ref{t:mainresult} with earlier nonlinear stability results~\cite{BjoernMod,HDRS22} of wave trains in nondegenerate reaction-diffusion systems against $C_{\mathrm{ub}}$-perturbations. First of all, we retrieve the same diffusive decay rates as in the reaction-diffusion case. It is argued in~\cite[Section~6.1]{BjoernMod} that these decay rates are sharp (up to possibly a logarithm). Second, we do require more regular initial data than in~\cite{BjoernMod}, where initial conditions $\vb_0$ in $C_{\mathrm{ub}}(\R)$ are considered. The reason is as follows. The lack of parabolic smoothing naturally leads one to consider initial data $\vb_0$ from the domain $C_{\mathrm{ub}}^2(\R) \times C_{\mathrm{ub}}^1(\R)$ of the diffusion-advection operator $\El_0$, so that the perturbed solution $\ub(t)$ of the semilinear evolution problem~\eqref{FHN_co} with initial condition $\ub(0) = \phi_0 + \vb_0$ is classical. Moreover, we lose one additional degree of regularity due to the embedding of uniformly local Sobolev spaces in $C_{\mathrm{ub}}$-spaces, cf.~\cite[Section~8.3.1]{SU17book}, which are used to obtain a nonlinear damping estimate to control regularity in the scheme, see~\S\ref{sec:strategy} below for more details. We emphasize that we only require our initial data to be \emph{bounded} in $(C_{\mathrm{ub}}^3 \times C_{\mathrm{ub}}^2)$-norm and, similar as in~\cite{BjoernMod}, to be small in $L^\infty$-norm. This contrasts with earlier nonlinear stability results~\cite{LLEperiod,STVenant1,STVenant2} of wave trains in semilinear (nonparabolic) problems and is due to the use of Gagliardo-Nirenberg interpolation in the nonlinear damping estimate, see Remark~\ref{rem:gagl} for more details.

The approximation of the phase modulation $\psi(t)$ by a solution to the viscous Hamilton-Jacobi equation~\eqref{e:HamJac} was also found in the reaction-diffusion case in~\cite{BjoernMod}. Thus, independent of the precise structure and smoothing properties of the underlying system, the viscous Hamilton-Jacobi equation arises as governing equation for the phase modulation, whose coefficients are fully determined by the first and second-order terms in the expansion of the linear and nonlinear dispersion relations. We refer to~\cite{DSSS} for further details. Important to note is that once the diffusive spectral stability assumptions are violated, e.g.~due to the presence of additional conservation laws, the governing equation of the phase modulation can change, cf.~\cite{JNRZ14}.

\subsection{Strategy of proof and main challenges} \label{sec:strategy}

We prove Theorem~\ref{t:mainresult} by extending the $L^\infty$-theory, which was recently developed in~\cite{HDRS22,BjoernMod} and applied to establish nonlinear stability of wave trains in reaction-diffusion systems against $C_{\mathrm{ub}}$-perturbations, beyond the parabolic setting. Here, we outline the strategy of proof and explain how we address the novel challenges arising due to incomplete parabolicity.

To prove Theorem~\ref{t:mainresult}, we wish to control the perturbation $\vt(t) = \ub(t) - \phi_0$ over time, which obeys the semilinear equation
\begin{align}
\left(\partial_t - \El_0\right)\vt = \NT(\vt), \label{e:umodpert}
\end{align}
where $\El_0$ is the linearization of~\eqref{FHN_co} about $\phi_0$ given by~\eqref{deflinearization} and $\NT(\vt)$ is the nonlinear residual given by
\begin{align*}
\NT(\vt) &= F(\phi_0+\vt) - F(\phi_0) - F'(\phi_0) \vt.
\end{align*}
We will establish that $\El_0$ generates a $C_0$-semigroup $\re^{\El_0t}$, which, due to the fact that $\El_0$ has spectrum up to the imaginary axis $\ri \R$, does not exhibit decay as an operator on $C_{\mathrm{ub}}(\R)$, thus obstructing a standard nonlinear stability argument. 

In earlier works~\cite{JONZ,JUN,SAN3}, considering the nonlinear stability of wave trains in reaction-diffusion systems against localized perturbations, this issue was addressed by employing its Floquet-Bloch representation to decompose the semigroup generated by the linearization and introducing a spatio-temporal phase modulation to capture the critical diffusive behavior. More precisely, one considers the \emph{inverse-modulated perturbation}
\begin{align} \vb(\zeta,t) = \ub(\xx - \psi(\xx,t),t) - \phi_0(\xx), \label{e:defv}\end{align}
where the spatio-temporal phase modulation $\psi(\xx,t)$ is determined a posteriori. The inverse-modulated perturbation satisfies a \emph{quasilinear} equation of the form
\begin{align} \label{e:pertbeq}
(\partial_t - \El_0)\left(\vb + \phi_0'\psi - \psi_\xx \vb\right) = N\left(\vb, \vb_\xx,\vb_{\xx\xx},\psi_\xx, \partial_t \psi,\psi_{\xx\xx}, \psi_{\xx\xx\xx}\right),
\end{align}
where $N$ is nonlinear in its variables. One decomposes the semigroup $\re^{\El_0 t}$ into a principal part of the form $\phi_0'S_p(t)$, where $S_p(t)$ decays diffusively, and a residual part exhibiting higher order temporal decay. Finally, one chooses the phase modulation $\psi(t)$ in~\eqref{e:defv} in such a way that it captures the most critical contributions in the Duhamel formulation of~\eqref{e:pertbeq}, allowing one to close a nonlinear iteration argument in $\psi_\xx,\psi_t$ and $\vb$. The leading-order dynamics of the phase modulation $\psi$ is then given by a viscous Hamilton-Jacobi equation, cf.~\cite{DSSS} and Remark~\ref{rem:expansion}.

The above approach has successfully been extended to the nonlinear stability analysis of periodic traveling waves against $L^2$-localized perturbations in nonparabolic dissipative problems such as the St.~Venant equations~\cite{STVenant1,STVenant2} and the Lugiato-Lefever equation~\cite{LLEperiod} using resolvent estimates and the Gearhart-Pr\"uss theorem to render exponential decay of the high-frequency part of the $C_0$-semigroup. 

In the nonlinear stability analysis of wave trains in reaction-diffusion systems against $C_{\mathrm{ub}}$-perturbations in~\cite{BjoernMod}, the decomposition was carried out on the level of the temporal Green's function, which is $C^2$ and exponentially localized, thus circumventing an application of the Floquet-Bloch transform to nonlocalized functions, which is only defined in the sense of tempered distributions. This leads to an explicit representation of the low-frequency part of the semigroup as in~\cite{JONZ} and control on the high-frequency part by pointwise Green's function estimates established in~\cite{JUN}. 

For nonelliptic operators, such as $\El_0$, the temporal Green's function is typically a distribution, complicating a potential decomposition via the Floquet-Bloch transform. We address this challenge by taking inspiration from~\cite{BjoernAvery} and employing its inverse Laplace representation, given by the complex inversion formula
\begin{align} \label{laplace}
\re^{\El_0 t} \vb = \lim_{R \to \infty} \frac{1}{2 \pi \ri} \int_{\eta - \ri R}^{\eta + \ri R} \re^{\lambda t} (\lambda - \El_0)^{-1} \vb \, \de \lambda
\end{align}
with $\eta, t > 0$ and $\vb \in D(\El_0)$, to decompose the semigroup. By partitioning and deforming the integration contour in~\eqref{laplace}, we write the semigroup as the sum of a high- and low-frequency part. Here, we associate the high-frequency part of the semigroup with pieces of the deformed contour integral where $|\Im(\lambda)|\gg  1$, i.e., where~$\re^{\lambda t}$ rapidly oscillates, and the low-frequency part of the semigroup with pieces of the deformed contour integral where $|\lambda| \ll  1$. 

As the space of perturbations $C_{\mathrm{ub}}(\R)$ does not admit any Hilbert-space structure, we cannot rely on the Gearhart-Pr\"uss theorem (or leverage the sectoriality of the linearization) to establish a spectral mapping property. Therefore, we instead use the expansion of the resolvent as a Neumann series for $\lambda \in \C$ with $|\Im(\lambda)| \gg 1$, which was established in~\cite{BjoernAvery}, to control the high-frequency part of the semigroup. The leading-order terms in the Neumann series expansion of resolvent are not absolutely integrable over the high-frequency parts of the contour in~\eqref{laplace} and, thus, the question of how to control these terms is not straightforward. Here, we cannot resort to the arguments in~\cite{BjoernAvery} which rely on test functions, since these are not dense in $C_{\mathrm{ub}}(\R)$. Instead, we identify the critical terms in the Neumann series expansion of $(\lambda - \El_0)^{-1}$ as products of resolvents of simple diffusion and advection operators. The corresponding terms in the inverse Laplace formula then correspond to \emph{convolutions} of the $C_0$-semigroups generated by these diffusion and advection operators. As far as the authors are aware, the observation that the complex inversion formula holds for convolutions of $C_0$-semigroups is novel and is therefore of its own interest, cf.~\cite{HAA}. All in all, we obtain that the high-frequency part of the semigroup is exponentially decaying on $C_{\mathrm{ub}}(\R)$. 

To render decay of the low-frequency part of the semigroup one must rely on diffusive smoothing in the case of nonlocalized perturbations. The diffusive decay rates of the low-frequency part are not strong enough to control the critical nonlinear term $\nu (\psi_\xx)^2$ in the perturbed viscous Hamilton-Jacobi equation satisfied by $\psi$. In~\cite{BjoernMod}, this difficulty has been addressed by further decomposing the low-frequency part of the semigroup via its Floquet-Bloch representation and relating its principal part to the convective heat semigroup $\smash{\re^{(d\partial_\xx^2 - c_g\partial_\xx)t}}$, which allows to apply the Cole-Hopf transform to eliminate the critical $(\psi_\xx)^2$-term.

Here, we link the inverse Laplace representation of the low-frequency part with the Floquet-Bloch representation from~\cite{BjoernMod} modulo exponentially decaying terms, while exploiting the nonzero group velocity of the wave train, cf.~Assumption~\ref{assH2}. This allows us to harness the decomposition of and estimates on the low-frequency part of the semigroup from~\cite{BjoernMod}. We emphasize that, to the authors' knowledge, such a link has not been established before and is interesting in its own right. 

After applying the Cole-Hopf transform to the equation of the phase modulation $\psi$ to eliminate the critical nonlinear term, the decay of all remaining linear and nonlinear terms is strong enough to close a nonlinear iteration argument in $\psi_\xx,\psi_t$ and $\vb$. Yet, the equation for the inverse-modulated perturbation is quasilinear and an apparent loss of derivatives must be addressed to control regularity in the nonlinear argument. This is a standard issue in the nonlinear stability of wave trains and it has been recognized that, as long as the underlying equation is semilinear, such a loss of derivatives can be addressed by considering the unmodulated perturbation or to the so-called \emph{forward-modulated perturbation}
\begin{align} \label{e:defforwregular}
\mathring{\vb}(\zeta,t) = \ub(\xx,t) - \phi_0(\xx + \psi(\xx,t)), 
\end{align}
which measures the deviation from the modulated wave train, cf.~\cite{ZUM22}. Both the unmodulated perturbation $\vt(t)$ and the forward-modulated perturbation $\mathring{\vb}(t)$ obey a semilinear equation in which no derivatives are lost, yet where decay is too slow to close an independent iteration scheme. However, by relating $\vt(t)$ (or $\mathring{\vb}(t)$) to the inverse-modulated perturbation $\vb(t)$ regularity can be controlled in the nonlinear iteration scheme. Regularity control can then be obtained by showing that $\vt(t)$ (or $\mathring{\vb}(t)$) obeys a so-called \emph{nonlinear damping estimate}~\cite{JNRZ14,ZUM22}, which is an energy estimate bounding the $H^m$-norm of the solution for some $m \in \mathbb{N}$ in terms of the $H^m$-norm of its initial condition and the $L^2$-norm of the solution. A nonlinear damping estimate for the forward-modulated perturbation has been derived in the setting of the FHN system in~\cite[Proposition~8.6]{BjoernAvery}.

A second option is to control regularity by deriving tame estimates on derivatives of $\vt(t)$ (or $\mathring{\vb}(t)$) via its Duhamel formulation~\cite{LLEperiod,BjoernMod,RS21}. In the absence of parabolic smoothing, the advantage of using nonlinear damping estimates is that they yield sharp bounds on derivatives and typically require less regular initial data, as can for instance be seen by comparing~\cite[Theorem~6.2]{ZUM22} with~\cite[Theorem~1.3]{LLEperiod}. In the case of nonlocalized perturbations, one has so far been compelled to the second approach using tame estimates, cf.~\cite{BjoernMod,RS21}, since the lack of localization prohibits the use of $L^2$-energy estimates. Motivated by the possibility to accommodate less regular initial data, we control regularity in this work by extending the method of nonlinear damping estimates to uniformly local Sobolev norms, see~\cite[Section~8.3.1]{SU17book}, which allow for nonlocalized perturbations. On top of that, we work with a slightly modified version of the forward-modulated perturbation given by
\begin{align}
\begin{split}
 \mathring{\zt}(\zeta,t) &:= \ub(\xx,t) - \phi(\xx + \psi(\xx,t)(1+\psi_\xx(\xx,t));1+\psi_\xx(\xx,t))\\
 &= \mathring{\vb}(\xx,t) + \phi_0(\xx + \psi(\xx,t)) - \phi(\xx + \psi(\xx,t)(1+\psi_\xx(\xx,t));1+\psi_\xx(\xx,t))\\
 &= \widetilde{\vb}(\xx,t) + \phi_0(\xx) - \phi(\xx + \psi(\xx,t)(1+\psi_\xx(\xx,t));1+\psi_\xx(\xx,t)), 
 \end{split}
 \label{e:defvforw}
\end{align}
which again satisfies a semilinear equation in which no derivatives are lost and is well-defined as long as $\|\psi_\xx(t)\|_{L^\infty}$ is sufficiently small, cf.~Proposition~\ref{prop:family}. The reason is that $\mathring{\zt}(t)$ and its derivatives exhibit stronger decay than $\mathring{\vb}(t)$, cf.~\cite[Corollary~1.4]{BjoernMod}. Having sharper bounds on derivatives, it is no longer necessary to move derivatives in the Duhamel formulation from the nonlinearity to the slowly decaying principal low-frequency part $S_p(t)$ of the semigroup as in~\cite{BjoernMod}. This provides a significant simplification with respect to~\cite{BjoernMod} as the computation and estimation of commutators between the operators $S_p(t)$ and $\partial_\xx^m, m \in \mathbb{N}$, is no longer necessary. 

Thus, using uniformly local Sobolev norms\footnote{We note that uniformly local Sobolev norms have also been used in other works, e.g.~\cite{GALSLI}, to make energy estimate methods available in $L^\infty$-spaces.}, we obtain a nonlinear damping estimate for the modified forward-modulated perturbation $\mathring{\zt}(t)$ and our nonlinear iteration scheme can also be closed from the perspective of regularity. This then leads to the proof of Theorem~\ref{t:mainresult}.

\begin{remark} \label{rem:expansion}
It was already observed in~\cite{DSSS} that the coefficients of the viscous Hamilton-Jacobi equation~\eqref{e:HamJac}, governing the leading-order phase dynamics, can be expressed in terms of the coefficients of the second-order expansion of the linear and nonlinear dispersion relations $\lambda_c(\xi)$ and $\omega(k)$, cf.~Propositions~\ref{prop:family} and~\ref{prop:speccons} and identity~\eqref{e:nuexpress}. In the current setting of fully nonlocalized perturbations~\cite{BjoernMod}, it is important to identify the leading-order Hamilton-Jacobi dynamics of the phase modulation as this allows for an application of the Cole-Hopf transform to eliminate the most critical nonlinear term. In contrast, in the nonlinear stability analyses~\cite{JONZ,JUN,LLEperiod,STVenant1,STVenant2} of wave trains against localized perturbations, it is not necessary to determine the leading-order phase dynamics explicitly. The derivation of the viscous Hamilton-Jacobi equation in the current setting can be found in~\S\ref{sec:derihamjac} and exploits the characterization of the first-order term in the expansion of the eigenfunction $\Phi_\xi$ 
as the derivative of the family of wave trains $\phi(\cdot;k)$, established in Proposition~\ref{prop:family}, with respect to the wavenumber $k$, cf.~Proposition~\ref{prop:speccons}. 
\end{remark}

\begin{remark} 
The nonlinear damping estimate, used in the proof of Theorem~\ref{t:mainresult}, leads to estimates on \emph{derivatives} of the (modulated) perturbation. Specifically, we can replace the $L^\infty$-norms in estimates~\eqref{e:mtest10}-\eqref{e:mtest12} by $(C_{\mathrm{ub}}^2 \times C_{\mathrm{ub}}^1)$-norms upon substituting $E_0$ by its fractional power $\smash{E_0^{\frac15}}$.\footnote{In fact, we can also take $\alpha = \frac15$ in~\eqref{e:mtest3}.}

Here, the occurrence of the fractional power is a consequence of the use of Gagliardo-Nirenberg interpolation in the nonlinear damping estimate, see Remark~\ref{rem:gagl}. In addition, we note that, although our initial perturbation $\vb_0$ lies in $C_{\mathrm{ub}}^3(\R) \times C_{\mathrm{ub}}^2(\R)$, we do not control the associated norm in our nonlinear stability analysis, since we lose one degree of regularity by embedding of uniformly local Sobolev spaces in $C_{\mathrm{ub}}$-spaces. Nevertheless, by considering more regular initial data in Theorem~\ref{t:mainresult}, it is possible to track higher-order derivatives in the nonlinear argument. More precisely, taking $m \in \mathbb{N}$ and $\vb_0 \in C_{\mathrm{ub}}^{m+3}(\R) \times C_{\mathrm{ub}}^{m+2}(\R)$ with $\|\vb_0\|_{C_{\mathrm{ub}}^{m+3} \times C_{\mathrm{ub}}^{m+2}} < K$ in Theorem~\ref{t:mainresult}, we find
\begin{align*}
\ub \in C\big([0,\infty),C_{\mathrm{ub}}^{m+3}(\R) \times C_{\mathrm{ub}}^{m+2}(\R)\big) \cap C^1\big([0,\infty),C_{\mathrm{ub}}^{m+1}(\R)\big).
\end{align*}
and the estimates~\eqref{e:mtest10}-\eqref{e:mtest2} can be upgraded to
\begin{align*} 
\left\|\ub(t)-\phi_0\right\|_{C_{\mathrm{ub}}^{m+2} \times C_{\mathrm{ub}}^{m+1}} &\leq ME_0^{\alpha_m}, \\
\left\|\ub(t)-\phi_0(\cdot+\psi(\cdot,t))\right\|_{C_{\mathrm{ub}}^{m+2} \times C_{\mathrm{ub}}^{m+1}} &\leq \frac{ME_0^{\alpha_m}}{\sqrt{1+t}},\\
\left\|\ub(t)-\phi_0\left(\cdot + \psi(\cdot,t)\left(1+\psi_\xx(\cdot,t)\right);1+\psi_\xx(\cdot,t)\right)\right\|_{C_{\mathrm{ub}}^{m+2} \times C_{\mathrm{ub}}^{m+1}} &\leq M E_0^{\alpha_m} \frac{ \log(2+t)}{1+t},
\end{align*}
where $\alpha_m > 0$ depends on $m$ only, and
\begin{align*}
\|\partial_t \psi_\xx(t)\|_{C_{\mathrm{ub}}^{3+m}}
\leq \frac{ME_0}{\sqrt{1+t}}, \qquad \ \left\|\psi_{\xx\xx}(t)\right\|_{C_{\mathrm{ub}}^{4+m}} \leq ME_0 \frac{\log(2+t)}{1+t}
\end{align*}
for all $t \geq 0$. For the sake of clarity of exposition and in order to reduce the amount of technicalities, we have chosen to only consider $(C_{\mathrm{ub}}^3 \times C_{\mathrm{ub}}^2)$-regular initial data only in our nonlinear stability analysis. 
\end{remark}

\subsection{Outline}

This paper is organized as follows. In~\S\ref{sec:resol}, we analyze the resolvent associated with the linearization $\El_0$ of~\eqref{FHN_co} about the wave train. In~\S\ref{sec_lin} we decompose the $C_0$-semigroup $\re^{\El_0 t}$ and derive associated estimates with the aid of the inverse Laplace representation and establish a Floquet-Bloch representation for its critical low-frequency part. In~\S\ref{sec:iteration}, we set up our nonlinear iteration scheme and derive a nonlinear damping estimate. We close the nonlinear argument and prove our main result, Theorem~\ref{t:mainresult}, in~\S\ref{sec:nonlstab}. We conclude in~\S\ref{sec:disc} by discussing the wider applicability of our method to achieve nonlinear stability of wave trains against fully nonlocalized perturbations in semilinear dissipative problems. Appendix~\ref{sec:laplace} is devoted to background material on the vector-valued Laplace transform. In particular, we prove that its complex inversion formula holds for convolutions of $C_0$-semigroups. Finally, we relegate the derivation of the equation for the modified forward-modulated perturbation to Appendix~\ref{app:derivationforward}. 

\paragraph*{Notation.} Let $S$ be a set, and let $A, B \colon S \to \R$. Throughout the paper, the expression ``$A(x) \lesssim B(x)$ for $x \in S$'', means that there exists a constant $C>0$, independent of $x$, such that $A(x) \leq CB(x)$ holds for all $x \in S$. 

\paragraph*{Acknowledgments.}  This project is funded by the Deutsche Forschungsgemeinschaft (DFG, German Research Foundation) -- Project-ID 491897824.

\section{Resolvent analysis}
\label{sec:resol}

This section is devoted to the study of the resolvent and serves as preparation to derive pure $L^\infty$-estimates on the high- and low-frequency components of the semigroup given by~\eqref{laplace}. That is, we collect and prove properties of $(\lambda-\El_0)^{-1}$ in the regimes $|\Im(\lambda)| \gg 1$ and $|\lambda|\ll 1$. Our refined low-frequency analysis of the resolvent is the starting point to link the inverse Laplace representation to the Floquet-Bloch representation of the low-frequency part of the semigroup. 

\subsection{Low-frequency resolvent analysis and decomposition}

We consider the resolvent problem
\begin{align} \label{e:resolventprob}
(\El_0 - \lambda) \vb = \gb
\end{align}
with $\vb = (u,v)^\top$ and $\gb \in C_{\mathrm{ub}}(\R)$ for $\lambda$ in a small ball $B(0,\delta) \subset \C$ of radius $\delta > 0$ centered at the origin. We proceed as in~\cite{BjoernAvery} and write~\eqref{e:resolventprob} as a first-order system
\begin{align}
\label{e:firstorder}
\psi' = A(\zeta;\lambda) \psi + G
\end{align}
in $\psi = (u,u_\zeta,v)^\top$ with inhomogeneity $G = (0,\gb)^\top$ and coefficient matrix
\begin{align*}
A(\zeta;\lambda) = \begin{pmatrix}
0 & 1 & 0 \\
\lambda - f'(u_0) & -c_0 & 1 \\
-\frac{\varepsilon}{c_0} & 0 & \frac{\varepsilon \gamma + \lambda}{c_0}
\end{pmatrix},
\end{align*}
where $u_0$ is the first-component of the wave train $\phi_0 = (u_0,v_0)^\top$ and $f(u) = u(1-u)(u-\mu)$ is the cubic nonlinearity in the FHN system~\eqref{FHN}. 

The coefficient matrix $A(\cdot;\lambda)$ is $T$-periodic for each $\lambda \in \C$. Thus, we can apply Floquet theory, cf.~\cite[Section~2.1.3]{KapitulaPromislow}, to establish a $T$-periodic change of coordinates, which is locally analytic in $\lambda$, converting the homogeneous problem
\begin{align}
\psi' = A(\zeta;\lambda)\psi \label{e:firstorderhom}
\end{align}
into a constant-coefficient system. 

\begin{proposition} \label{p:Floquet0} Assume~\ref{assH1}. For $\delta > 0$ sufficiently small, there exist maps $Q \colon \R \times B(0,\delta) \to \C^{3 \times 3}$ and $M \colon B(0,\delta) \to \C^{3 \times 3}$ such that the evolution $T(\zeta,\bar{\zeta};\lambda)$ of~\eqref{e:firstorderhom} can be expressed as
\begin{align*}
T(\zeta,\zb;\lambda) = Q(\zeta;\lambda)^{-1} \re^{M(\lambda)(\zeta - \zb)} Q(\zb;\lambda).
\end{align*}
Here, $Q(\cdot;\lambda)$ is smooth and $T$-periodic for each $\lambda \in B(0,\delta)$. Moreover, $M$ and $Q(\zeta;\cdot)$ are analytic for each $\zeta \in \R$.
\end{proposition}

An eigenvalue $\nu(\lambda)$ of the monodromy matrix $M(\lambda)$ is called a \emph{spatial Floquet exponent}. It gives rise to a solution $\psi(\zeta;\lambda) = \re^{\nu(\lambda) \zeta} p(\zeta;\lambda)$ of~\eqref{e:firstorderhom}, where $p(\cdot;\lambda)$ is $T$-periodic. Thus, translating back to the eigenvalue problem $(\El_0 - \lambda) \vb = 0$, one readily observes that for each $\xi \in \C$ a point $\lambda \in B(0,\delta)$ is a (temporal) eigenvalue of the Bloch operator $\El(\xi)$ if and only if $\ri \xi$ is an eigenvalue of $M(\lambda)$. The spectral decomposition~\eqref{Blochspecdecomp} then implies that a point $\lambda \in B(0,\delta)$ lies in $\sigma(\El_0)$ if and only if $M(\lambda)$ has a purely imaginary eigenvalue. 

Proposition~\ref{prop:speccons} yields balls $V_1,V_2 \subset \C$ centered at $0$ and a holomorphic map $\lambda_c \colon V_1 \to \C$ such that $\El(\xi)$ has a simple eigenvalue $\lambda_c(\xi)$ for each $\xi \in V_1$ and it holds $\sigma(\El_0) \cap V_2 = \{\lambda_c(\xi) : \xi \in \R \cap V_1\} \cap V_2$. Since we have $\lambda_c'(0) = - \ri c_g \neq 0$ by Assumption~\ref{assH2}, the implicit function theorem implies, provided $\delta > 0$ is sufficiently small, that for each $\lambda \in B(0,\delta)$ the matrix $M(\lambda)$ possesses precisely one simple eigenvalue $\nu_c(\lambda)$ in $V_1$. 
These observations readily lead to the following proposition.

\begin{proposition} \label{p:Floquet} Assume~\ref{assH1},~\ref{assH2} and~\ref{assD1}-\ref{assD3}. There exist constants $C, \delta > 0$ and a holomorphic map $\nu_c \colon B(0,\delta) \to \C$ satisfying the following assertions.
\begin{itemize}
    \item[(i)] $\nu_c(\lambda)$ is a simple spatial Floquet exponent associated with the $T$-periodic first-order problem~\eqref{e:firstorderhom} for each $\lambda \in B(0,\delta)$. 
    \item[(ii)] A point $\lambda \in B(0,\delta)$ lies in $\sigma(\El_0)$ if and only if $\nu_c(\lambda)$ is purely imaginary.
    \item[(iii)] We have $\nu_c(\lambda_c(\xi)) = \ri \xi$ for each $\xi \in V_1$ such that $\lambda_c(\xi) \in B(0,\delta)$.
    \item[(iv)] The expansion
    \begin{align*}
        \left|\nu_c(\lambda) + \frac{1}{c_g} \lambda\right| \leq C|\lambda|^2
    \end{align*}
    holds for all $\lambda \in B(0,\delta)$. 
    \item[(v)] For $\lambda \in B(0,\delta)$ to the right of $\sigma(\El_0)$ we have $\Re(\nu_c(\lambda)) > 0$.
    \end{itemize}
\end{proposition}

Propositions~\ref{p:Floquet0} and~\ref{p:Floquet} imply that for $\lambda \in B(0,\delta)$ system~\eqref{e:firstorderhom} has an exponential dichotomy on $\R$ if and only if there are no purely imaginary Floquet exponents, which is the case precisely if $\lambda$ lies in the resolvent set $\rho(\El_0)$. Hence, taking $\lambda \in B(0,\delta) \cap \rho(\El_0)$ and letting $P^s(\lambda)$ and $P^u(\lambda)$ be the spectral projections onto the stable and unstable subspaces of $M(\lambda)$, we can express the spatial Green's function associated with~\eqref{e:firstorderhom} as
\begin{align*}
\mathcal{G}(\zeta,\zb;\lambda) = Q(\zeta;\lambda)^{-1} \re^{M(\lambda)(\zeta - \zb)} \left(P^s(\lambda) \mathbf{1}_{(-\infty,\zeta]}(\zb) - P^u(\lambda) \mathbf{1}_{[\zeta,\infty)}(\zb)\right) Q(\zb;\lambda)
\end{align*}
where $\mathbf{1}_{(-\infty,\zeta]}$ and $\mathbf{1}_{[\zeta,\infty)}$ are indicator functions. Introducing the matrices
\begin{align*}
\Pi_2 = \begin{pmatrix}
1 & 0 & 0 \\
0 & 0 & 1
\end{pmatrix}, \qquad \Pi_3 = \begin{pmatrix}
0 & 0 \\
1 & 0 \\
0 & c_0^{-1}
\end{pmatrix}. 
\end{align*}
to translate between the original formulation~\eqref{e:resolventprob} and the first-order formulation~\eqref{e:firstorder} of the resolvent problem, we find that the unique solution of~\eqref{e:resolventprob} is now given by
\begin{align*}
\left((\El_0-\lambda)^{-1} \gb\right)(\zeta) = \vb(\zeta;\lambda) = \int_\R \Pi_2 \mathcal{G}(\zeta,\zb;\lambda)  \Pi_3 \gb(\zb) \,\de \zb.
\end{align*}
By Proposition~\ref{p:Floquet} the spatial Floquet exponent $\nu_c(\lambda)$ is a simple eigenvalue of $M(\lambda)$ and all other spatial Floquet exponents are bounded away from $\ri \R$ for $\lambda \in B(0,\delta)$. Therefore, the spectral projection $P^{cu}(\lambda)$ of $M(\lambda)$ onto the eigenspace associated with $\nu_c(\lambda)$ is defined for all $\lambda \in B(0,\delta)$. For $\lambda \in B(0,\delta)$ to the right of $\sigma(\El_0)$ it holds $\Re(\nu_c(\lambda)) > 0$ and we can decompose $P^u(\lambda) = P^{uu}(\lambda) + P^{cu}(\lambda)$. This then leads to the desired resolvent decomposition for small $\lambda$.

\begin{proposition} \label{p:lowfreqresolvdecomp}
Assume~\ref{assH1},~\ref{assH2} and~\ref{assD1}-\ref{assD3}. There exist constants $C,\delta > 0$ and a holomorphic map $S_e^0 \colon B(0,\delta) \to \mathcal{B}\big(C_{\mathrm{ub}}(\R)\big)$ such that for $\lambda \in B(0,\delta)$, $\gb \in C_{\mathrm{ub}}(\R)$ and $\zeta \in \R$ we have
\begin{align} \label{e:resolventdecomplowfreq}
\begin{split}
\left((\El_0-\lambda)^{-1} \gb\right)(\zeta) &= -\int_\R \Pi_2 Q(\zeta;\lambda)^{-1} \re^{\nu_c(\lambda)(\zeta - \zb)} \mathbf{1}_{[\zeta,\infty)}(\zb) P^{cu}(\lambda) Q(\zb;\lambda) \Pi_3 \gb(\zb) \,\de \zb\\
&\qquad + \, \left(S_e^0(\lambda) \gb\right) (\zeta)
\end{split}
\end{align}
and it holds
\begin{align*}
\left\|S_e^0(\lambda)\gb\right\|_{L^\infty} \leq C \|\gb\|_{L^\infty}.
\end{align*}
\end{proposition}

In order to later relate the inverse Laplace representation of the low-frequency part of the semigroup $\re^{\El_0 t}$ to its Floquet-Bloch representation, we prove the following technical lemma showing that the expression $\Pi_2 Q(\zeta;\lambda)^{-1} P^{cu}(\lambda) Q(\zb;\lambda) \Pi_3$ in~\eqref{e:resolventdecomplowfreq} can be written as a product of solutions of the eigenvalue problem $(\El_0 - \lambda) \vb = 0$ and its adjoint $(\El_0-\lambda)^* \vb = 0$. 

\begin{lemma}
\label{expansion_in_laplace}
Assume~\ref{assH1},~\ref{assH2} and~\ref{assD1}-\ref{assD3}. There exist a constant $\delta > 0$ and functions $\Psi, \tilde{\Psi} \colon \R \times B(0,\delta) \to \C^2$ satisfying
\begin{align} \label{e:spectralprojid}
\Pi_2 Q(\zeta;\lambda)^{-1} P^{cu}(\lambda) Q(\zb;\lambda) \Pi_3 = \Psi(\zeta;\lambda)\tilde{\Psi}(\zb;\lambda)^*
\end{align}
for $\zeta, \zb \in \R$ and $\lambda \in B(0,\delta)$. Moreover, $\Psi(\cdot;\lambda)$ and $\tilde{\Psi}(\cdot;\lambda)$ are smooth and $T$-periodic for each $\lambda \in B(0,\delta)$ and $\Psi(\zeta;\cdot)\tilde{\Psi}(\zb;\cdot)^*$ is analytic for each $\zeta,\zb \in \R$. Finally, we have
\begin{align} \label{e:spectralprojid2}
\Psi(\cdot;\lambda_c(\xi)) = \Phi_\xi, \qquad \lambda_c'(\xi) \tilde{\Psi}(\cdot;\lambda_c(\xi)) = \ri \widetilde{\Phi}_\xi
\end{align}
for $\xi \in V_1$ such that $\lambda_c(\xi) \in B(0,\delta)$, where $\Phi_\xi$ and $\widetilde{\Phi}_\xi$ are defined in Proposition~\ref{prop:speccons}.
\end{lemma}
\begin{proof}
Let $\lambda \in B(0,\delta)$. By Propositions~\ref{p:Floquet0} and~\ref{p:Floquet} the monodromy matrix $M(\lambda)$ has a simple eigenvalue $\nu_c(\lambda)$, provided $\delta > 0$ is sufficiently small. Let $w_1(\lambda)$ be an associated eigenvector. Moreover, let $\widetilde{w}_1(\lambda)$ be an eigenvector associated with the simple eigenvalue $\overline{\nu_c(\lambda)}$ of the adjoint matrix $M(\lambda)^*$. The spectral projection $P^{cu}(\lambda)$ onto the eigenspace of $M(\lambda)$ associated with $\nu_c(\lambda)$ is now given by
$$P^{cu}(\lambda) = \frac{ w_1(\lambda) \widetilde{w}_1(\lambda)^*}{\langle \widetilde{w}_1(\lambda), w_1(\lambda) \rangle}.$$ 
Since $\nu_c(\lambda)$ is simple for each $\lambda \in B(0,\delta)$, the map $P^{cu} \colon B(0,\delta) \to \C^{3 \times 3}$ is holomorphic by standard analytic perturbation theory~\cite[Section~II.1.4]{Kato}.

We define $\Psi, \tilde{\Psi} \colon \R \times B(0,\delta) \to \C^2$ by 
$$\Psi(\zeta;\lambda) = \Pi_2 v_1(\zeta;\lambda), \qquad v_1(\zeta;\lambda) := \frac{Q(\zeta;\lambda)^{-1}w_1(\lambda)}{\langle \widetilde{\Phi}_{-\ri \nu_c(\lambda)}, \Pi_2Q(\cdot;\lambda)^{-1}w_1(\lambda) \rangle_{L^2(0,T)}}$$
and
$$\tilde{\Psi}(\zeta;\lambda) = \Pi_3^* v_2(\zeta;\lambda), \qquad v_2(\zeta;\lambda) := \frac{Q(\zeta;\lambda)^* \widetilde{w}_1(\lambda)}{\langle w_1(\lambda),\widetilde{w}_1(\lambda)\rangle}\langle \Pi_2Q(\cdot;\lambda)^{-1}w_1(\lambda),\widetilde{\Phi}_{-\ri \nu_c(\lambda)}\rangle_{L^2(0,T)}.$$ 
Then, $\Psi(\cdot;\lambda)$ and $\tilde{\Psi}(\cdot;\lambda)$ are smooth and $T$-periodic for each $\lambda \in B(0,\delta)$ by Proposition~\ref{p:Floquet0}. One readily observes that~\eqref{e:spectralprojid} holds for $\zeta, \zb \in \R$ and $\lambda \in B(0,\delta)$. Moreover, since $Q(\zeta;\cdot)$, $Q(\zb;\cdot)$ and $P^{cu}$ are analytic by Proposition~\ref{p:Floquet0}, so is $\Psi(\zeta;\cdot)\tilde{\Psi}(\zb;\cdot)^*$ for each $\zeta,\zb \in \R$. 

Next, we observe that the evolution $T_{\mathrm{ad}}(\zeta,\zb;\lambda)$ of the adjoint problem
\begin{align} \label{e:firstorderadjoint}
\vartheta' = -A(\zeta;\lambda)^*\vartheta,
\end{align}
of~\eqref{e:firstorderhom} is given by $T_{\mathrm{ad}}(\zeta,\zb;\lambda) = T(\zb,\zeta;\lambda)^*$, where $T(\zeta,\zb;\lambda)$ is the evolution of~\eqref{e:firstorderhom}. So, since $\nu_c(\lambda)$ is an eigenvalue of $M(\lambda)$ with associated eigenvector $w_1(\lambda)$ and $-\overline{\nu_c(\lambda)}$ is an eigenvalue of $-M(\lambda)^*$ with associated eigenvector $\smash{\widetilde{w}_1(\lambda)}$, we obtain, by Proposition~\ref{p:Floquet0}, that $\psi(\zeta;\lambda) = \smash{\re^{\nu_c(\lambda) \zeta} v_1(\zeta;\lambda)}$ and $\vartheta(\zeta;\lambda) = \smash{\re^{-\overline{\nu_c(\lambda)} \zeta} v_2(\zeta;\lambda)}$ are solutions of~\eqref{e:firstorderhom} and~\eqref{e:firstorderadjoint}, respectively. Consequently, $\vb(\zeta;\lambda) =  \re^{\nu_c(\lambda) \zeta} \Psi(\zeta;\lambda)$ and $\smash{\widetilde{\vb}}(\zeta;\lambda) = \smash{\re^{-\overline{\nu_c(\lambda)} \zeta} \tilde{\Psi}(\zeta;\lambda)}$ solve the eigenvalue problems $(\El_0 - \lambda) \vb = 0$ and $(\El_0 - \lambda)^* \smash{\widetilde{\vb}} = 0$, respectively. Therefore, $\Psi(\cdot;\lambda), \smash{\tilde{\Psi}(\cdot;\lambda) \in H^2_{\mathrm{per}}(0,T)}$ are nontrivial solutions of the eigenvalue problems $(\El(-\ri \nu_c(\lambda)) -\lambda) \vb = 0$ and $(\El(-\ri \nu_c(\lambda)) -\lambda)^* \widetilde{\vb} = 0$, respectively. Now, let $\xi \in V_1$ be such that $\lambda_c(\xi) \in B(0,\delta)$. Then, we find with the aid of Proposition~\ref{p:Floquet} that $\Psi(\zeta;\lambda_c(\xi))$ and $\tilde{\Psi}(\zeta;\lambda_c(\xi))$ lie in $\ker(\El(\xi) - \lambda_c(\xi))$ and $\ker((\El(\xi) - \lambda_c(\xi))^*)$, respectively, which are spanned by $\Phi_\xi$ and $\widetilde{\Phi}_\xi$, respectively, by Proposition~\ref{prop:speccons}. Hence, on the one hand, the gauge condition $\smash{\langle \widetilde{\Phi}_\xi, \Psi(\zeta;\lambda_c(\xi))\rangle_{L^2(0,T)} = 1 = \langle \widetilde{\Phi}_\xi, \Phi_\xi\rangle_{L^2(0,T)}}$, cf.~Proposition~\ref{prop:speccons}, implies $\Phi_\xi = \Psi(\cdot;\lambda_c(\xi))$. On the other hand, there exists $\kappa_\xi \in \C \setminus \{0\}$ such that $\smash{\tilde{\Psi}(\cdot;\lambda_c(\xi)) = \kappa_\xi \widetilde{\Phi}_\xi}$. So, all that remains to show is that $\kappa_\xi = \ri/\lambda_c'(\xi)$.

First, using that $\psi(\zeta;\lambda) = \smash{\re^{\nu_c(\lambda) \zeta} v_1(\zeta;\lambda)}$ and $\vartheta(\zeta;\lambda) = \smash{\re^{-\overline{\nu_c(\lambda)} \zeta} v_2(\zeta;\lambda)}$ are solutions of~\eqref{e:firstorderhom} and~\eqref{e:firstorderadjoint}, respectively, and we have $\nu_c(\lambda_c(\xi)) = \ri \xi$ by Proposition~\ref{p:Floquet}, we obtain
\begin{align*}
v_1(\zeta;\lambda_c(\xi)) = \begin{pmatrix}
\Phi_{1,\xi} \\
\ri \xi \Phi_{1,\xi} + \Phi_{1,\xi}' \\
\Phi_{2,\xi}
\end{pmatrix}, \qquad 
v_2(\zeta;\lambda_c(\xi)) = \kappa_\xi \begin{pmatrix}
\left(c_0-\ri \xi\right)\widetilde{\Phi}_{1,\xi} - \widetilde{\Phi}_{1,\xi}' \\
\widetilde{\Phi}_{1,\xi} \\
c_0\widetilde{\Phi}_{2,\xi}
\end{pmatrix}.
\end{align*}
Finally, evoking Proposition~\ref{prop:speccons}, integrating by parts and using $1 = \langle \widetilde{\Phi}_\xi, \Phi_\xi\rangle_{L^2(0,T)}$, we arrive at
\begin{align*}
\kappa_\xi^{-1} &= \kappa_\xi^{-1}\big\langle v_2(\cdot;\lambda), v_1(\cdot;\lambda)\big\rangle_{L^2(0,T)}\\ 
&= 
\big\langle (c_0-\ri\xi) \widetilde{\Phi}_{1,\xi} - \widetilde{\Phi}_{1,\xi}', \Phi_{1,\xi}\big\rangle_{L^2(0,T)} + \big\langle \widetilde{\Phi}_{1,\xi}, \ri \xi \Phi_{1,\xi} + \Phi_{1,\xi}'\big\rangle_{L^2(0,T)} + \big\langle c_0 \widetilde{\Phi}_{2,\xi}, \Phi_{2,\xi}\big\rangle_{L^2(0,T)}\\
&= c_0 + 2\big\langle  \widetilde{\Phi}_\xi, D \left(\partial_\zeta + \ri \xi\right) \Phi_\xi\big\rangle_{L^2(0,T)} = -\ri\lambda_c'(\xi),
\end{align*}
which concludes the proof.
\end{proof}

\subsection{High-frequency resolvent analysis} 

We consider the resolvent $(\lambda-\El_0)^{-1}$ in the high-frequency regime. The spectrum of $\El_0$ away from the origin is by Proposition~\ref{prop:speccons} confined to the left-half plane with uniform distance from the imaginary axis, which allows us to deform the high-frequency parts of the integration contour in~\eqref{laplace} into the left-half plane away from the imaginary axis and the spectrum. Specifically, this leads us to consider the contours connecting $b\pm \ri\varpi_0$ with $b \pm \ri R$ for some $b<0$ and $R>\varpi_0>0$. Since these contours are unbounded as $R \rightarrow \infty$, we require a more refined understanding of the resolvent to secure exponential decay on the high-frequency contributions of the corresponding contour integrals.\footnote{Indeed, the naive bound $\|(\lambda - \El_0)^{-1}\| \lesssim \frac{1}{\Re \lambda}$, given by the Hille-Yosida theorem, is not strong enough.}  
The idea from~\cite{BjoernAvery} is to expand the resolvent $(\lambda - \El_0)^{-1}$ as a Neumann series in $\smash{|\Im(\lambda)|^{-\frac12}}$ for $|\Im(\lambda)| \gg 1$. It turns out that it suffices to explicitly identify the first three terms in this expansion, since a remainder of order $\smash{\mathcal{O}(|\Im(\lambda)|^{-\frac32})}$ is integrable. These three leading-order terms can be expressed as products of the resolvents of the simpler operators $\El_1 \colon C_{\mathrm{ub}}(\R) \subset C_{\mathrm{ub}}^2(\R) \to C_{\mathrm{ub}}(\R)$ and $\El_2 \colon C_{\mathrm{ub}}(\R) \subset C_{\mathrm{ub}}^1(\R) \to C_{\mathrm{ub}}(\R)$ given by
\begin{align*}
\El_1 = \partial_{\zeta\zeta}, \qquad \El_2 = c_0 \partial_\zeta - \varepsilon \gamma.
\end{align*}

Before stating the outcome of the expansion procedure in~\cite{BjoernAvery}, we provide the following standard result showing that $\El_1$ and $\El_2$ generate $C_0$-semigroups and providing bounds on their resolvents.

\begin{lemma} \label{lem:resolventsimple}
The operators $\El_1$ and $\El_2$ are closed, densely defined and generate $C_0$-semigroups on $C_{\mathrm{ub}}(\R)$. Morover, there exists a constant $M > 0$ such that for each $t \geq 0$, $\gb \in C_{\mathrm{ub}}(\R)$ and $\lambda \in \C \setminus \{0\}$ with $|\arg(\lambda)| \leq \frac{3\pi}{4}$ we have $\lambda \in \rho(\El_1)$ and
\begin{align*}
\|(\lambda - \El_1)^{-1}\gb\|_{L^\infty} \leq \frac{M}{|\lambda|} \|\gb\|_{L^\infty}, \qquad \|\re^{\El_1 t}\gb\|_{L^\infty} \leq M \|\gb\|_{L^\infty}.
\end{align*}
Finally, for each $t \geq 0$, $\gb \in C_{\mathrm{ub}}(\R)$ and $\lambda \in \C$ with $\Re(\lambda) > -\varepsilon\gamma$ it holds $\lambda \in \rho(\El_2)$ and
\begin{align*}
\|(\lambda - \El_2)^{-1}\gb\|_{L^\infty} \leq \frac{\|\gb\|_{L^\infty}}{\Re(\lambda) + \varepsilon\gamma}, \qquad \|\re^{\El_2 t}\gb\|_{L^\infty} \leq \re^{-\varepsilon\gamma t} \|\gb\|_{L^\infty}.
\end{align*}
\end{lemma}
\begin{proof}
The operator $\partial_\zeta$ generates the strongly continuous translational group on $C_{\mathrm{ub}}(\R)$ by~\cite[Proposition~II.2.10.1]{nagel}. Since translation preserves the $L^\infty$-norm, $\re^{\partial_\zeta t}$ is a group of isometries. Therefore, each $\lambda \in \C$ with $\Re(\lambda) > 0$ lies in $\rho(\partial_\zeta)$ and it holds $\|(\lambda - \partial_\zeta)^{-1}\gb\|_{L^\infty} \leq \Re(\lambda)^{-1}\|\gb\|_{L^\infty}$ for $\gb \in C_{\mathrm{ub}}(\R)$ by~\cite[Corollary~3.7]{nagel}. The bounds on $(\lambda -\El_2)^{-1}$ and $\re^{\El_2 t}$ now readily follow by rescaling space. Moreover, $\El_1$ generates a bounded analytic semigroup $\re^{\El_1 t}$ by~\cite[Corollary~II.4.9]{nagel} being the square of the operator $\partial_\zeta$. The resolvent estimate on $(\lambda - \El_1)^{-1}$ is stated in the proof of~\cite[Corollary~II.4.9]{nagel}. 
\end{proof}

Now, we state the high-frequency expansion of the resolvent $(\lambda - \El_0)^{-1}$ obtained in~\cite{BjoernAvery}.

\begin{proposition} \label{p:highfreqresolvdecomp}
Assume~\ref{assH1},~\ref{assH2} and~\ref{assD1}-\ref{assD3}. Let $b_0 > 0$. Then, there exist constants $C,\varpi_0 > 0$ such that we have $b + \ri \varpi \in \rho(\El_0)$ with
\begin{align*} 
(b + \ri \varpi - \El_0)^{-1} \gb = I^1_{b,\varpi}\gb + I^2_{b,\varpi}\gb + I^3_{b,\varpi}\gb + I^4_{b,\varpi}\gb,
\end{align*}
for all $\gb = (g_1,g_2)^\top \in C_{\mathrm{ub}}(\R)$ and $b, \varpi \in \R$ with $-\frac{3}{4} \varepsilon \gamma \leq b \leq b_0$ and $|\varpi| \geq \varpi_0$, where we denote
\begin{align*}
I^1_{b,\varpi}\gb 
= \begin{pmatrix}
(\ri\varpi - \mathcal{L}_1)^{-1} g_1 \\
(b + \ri\varpi - \mathcal{L}_2)^{-1} g_2
\end{pmatrix}, \qquad 
I^2_{b,\varpi}\gb  = \begin{pmatrix}
(\ri\varpi - \mathcal{L}_1)^{-1}(b + \ri\varpi - \mathcal{L}_2)^{-1} g_2 \\
-\varepsilon(b + \ri\varpi - \mathcal{L}_2)^{-1} (\ri\varpi - \mathcal{L}_1)^{-1} g_1
\end{pmatrix}
\end{align*}
and
\begin{align*}
I^3_{b,\varpi}\gb =\begin{pmatrix}
0 \\
-\varepsilon(b + \ri\varpi - \mathcal{L}_2)^{-1} (\ri\varpi - \mathcal{L}_1)^{-1} (b + \ri\varpi - \mathcal{L}_2)^{-1} g_2
\end{pmatrix},
\end{align*}
and the residual operator $I^4_{b,\varpi} \colon C_{\mathrm{ub}}(\R) \to C_{\mathrm{ub}}(\R)$ obeys the estimate
\begin{align*}
\left\|I^4_{b,\varpi}\gb\right\|_{L^\infty} \leq C |\varpi|^{-\frac32} \|\gb\|_{L^\infty}.
\end{align*}
\end{proposition}
\begin{proof}
This result was proved in~\cite[Lemma~B.4]{BjoernAvery} for $\gb \in C^\infty(\R)$, which immediately yields the statement by density of $C^\infty(\R)$ in $C_{\mathrm{ub}}(\R)$. 
\end{proof}

\section{Semigroup decomposition and linear estimates}
\label{sec_lin}

In this section, we decompose the $C_0$-semigroup generated by the linearization $\El_0$ of~\eqref{FHN_co} about the wave train $\phi_0$ and establish corresponding estimates. To this end, we employ the complex inversion formula~\eqref{laplace} of the $C_0$-semigroup. We first deform and partition the integration contour in~\eqref{laplace}. The high-frequency contribution of the deformed integration contour lies fully in the open left-half plane. Thus, exponential decay of the associated part of the $C_0$-semigroup can be obtained with the aid of the high-frequency resolvent expansion established in Proposition~\ref{p:highfreqresolvdecomp}. 

For low frequencies, we employ the resolvent decomposition obtained in Proposition~\ref{p:lowfreqresolvdecomp} leading to a critical and residual low-frequency contribution of the contour integral. On the one hand, we shift the contour fully into the open left-half plane to render exponential decay of the residual low-frequency contribution. On the other hand, we relate the critical low-frequency contribution to its Floquet-Bloch representation by shifting the integration contour onto the critical spectral curve. This allows us to gather the relevant estimates on this critical part of the semigroup from~\cite{BjoernMod}. 

\subsection{Inverse Laplace representation}

We start by showing that $\El_0$ generates a $C_0$-semigroup on $C_{\mathrm{ub}}(\R)$ and represent its action by the complex inversion formula.

\begin{proposition} \label{p:semigroupgen}
Assume~\ref{assH1}. Let $k \in \mathbb N_0$. The operator $\El_0 \colon D(\El_0) \subset C_{\mathrm{ub}}^k (\R, \C^2) \to C_{\mathrm{ub}}^k (\R, \C^2)$ with domain $D(\El_0) = C_{\mathrm{ub}}^{k+2} (\R, \C) \times C_{\mathrm{ub}}^{k+1} (\R, \C)$ generates a strongly continuous semigroup $\re^{\El_0 t}$ on $C_{\mathrm{ub}}^k(\R, \C^2)$. Moreover, there exists $\eta > 0$ such that the integration contour $\Gamma_0^R$, which is depicted in Figure~\ref{fig:deform} and connects $\eta - \ri R$ to $\eta + \ri R$, lies in the resolvent set $\rho(\El_0)$ and the inverse Laplace representation 
\begin{align} \label{laplace2}
\re^{\El_0 t} \gb = \lim_{R \to \infty} \frac{1}{2 \pi \ri} \int_{\Gamma_0^R} \re^{\lambda t} (\lambda - \El_0)^{-1} \gb \,\de \lambda
\end{align}
holds for any $\gb \in D(\El_0)$ and $t > 0$, where the limit in~\eqref{laplace2} is taken with respect to the $C_{\mathrm{ub}}^k$-norm.  
\end{proposition}
\begin{proof}
The operator $\El_0$ is a bounded perturbation of the diagonal diffusion-advection operator $L_0 = D \partial_{\zeta\zeta} + c_0 \partial_\zeta$ on $C_{\mathrm{ub}}^k (\R, \C^2)$ with dense domain $D(L_0) = C_{\mathrm{ub}}^{k+2} (\R, \C) \times C_{\mathrm{ub}}^{k+1} (\R, \C)$. The first component of $L_0$ is sectorial by~\cite[Corollary~3.1.9]{LUN} and thus generates an analytic semigroup, which is strongly continuous by~\cite[p.~34]{LUN}. On the other hand, the second component of $L_0$ generates the strongly continuous translational semigroup on $C_{\mathrm{ub}}^k(\R)$ by~\cite[Proposition~II.2.10.1]{nagel}. Since $\El_0$ is a bounded perturbation of $L_0$, $\El_0$ also generates a $C_0$-semigroup by~\cite[Theorem~III.1.3]{nagel}. The inverse Laplace representation, given by the complex inversion formula~\eqref{laplace2}, follows from~\cite[Proposition~3.12.1]{arendt}. 
\end{proof}

We note that standard semigroup theory provides sufficient control on the short-time behavior of the semigroup $\re^{\El_0 t}$. To distinguish between short- and long-time behavior, we introduce a smooth temporal cut-off function $\chi \colon [0,\infty) \to \R$ satisfying $\chi(t) = 0$ for $t \in [0,1]$ and $\chi(t) = 1$ for $t \in [2,\infty)$ and obtain the following short-time bound.

\begin{lemma} \label{l:shorttime}
Assume~\ref{assH1}. Consider $\El_0$ as an operator on $C_{\mathrm{ub}}(\R)$. There exist constants $C,\alpha > 0$ such that
\begin{align*}
\|(1-\chi(t))\re^{\El_0 t} \gb\| \leq C\re^{-\alpha t}
\end{align*}
holds for $\gb \in C_{\mathrm{ub}}(\R)$ and $t \geq 0$. 
\end{lemma}
\begin{proof}
This follows immediately from~\cite[Proposition~I.5.5]{nagel}, Proposition~\ref{p:semigroupgen} and the fact that $1-\chi$ vanishes on $[2,\infty)$. 
\end{proof}

Next, we deform the integration contour $\Gamma_0^R$ in~\eqref{laplace2} using Cauchy's integral theorem and analyticity of the resolvent $\lambda \mapsto (\lambda-\El_0)^{-1}$ on $\rho(\El_0)$.

\begin{proposition} \label{deform}
Assume~\ref{assH1} and~\ref{assD1}-\ref{assD2}. Consider $\El_0$ as an operator on $C_{\mathrm{ub}}(\R)$ and let $\eta > 0$ be as in Proposition~\ref{p:semigroupgen}. For each $\varpi_0 > 0$ sufficiently large the integration contours $\Gamma_1^R$ and $\Gamma_3^R$, which are depicted in Figure~\ref{fig:deform} and connect $\ri \varpi_0 - \frac{3}{4} \varepsilon\gamma$ to $\ri R  - \frac{3}{4} \varepsilon\gamma$ and $-\ri R  - \frac{3}{4} \varepsilon\gamma$ to $-\ri \varpi_0  - \frac{3}{4} \varepsilon\gamma$, respectively, as well as the rectangular integration contour $\Gamma_2$, which connects $-\ri \varpi_0  - \frac{3}{4} \varepsilon\gamma$ via $-\ri \varpi_0 + \frac{\eta}{2}$ and $\ri \varpi_0 + \frac{\eta}{2}$ to $\ri \varpi_0 - \frac{3}{4} \varepsilon\gamma$, lie in the resolvent set $\rho(\El_0)$. Moreover, we have
\begin{align} \label{e:Cauchy1}
\begin{split}
\re^{\El_0 t} \gb &= \frac{\chi(t)}{2\pi \ri} \int_{\Gamma_2} \re^{\lambda t} (\lambda - \El_0)^{-1} \gb \,\de\lambda + \lim_{R \to \infty} \frac{\chi(t)}{2 \pi \ri} \int_{\Gamma_1^R \cup \Gamma_3^R} \re^{\lambda t} (\lambda - \El_0)^{-1} \gb \,\de\lambda\\ 
&\qquad + \, (1-\chi(t))\re^{\El_0 t} \gb
\end{split}
\end{align}
for $\gb \in D(\El_0)$ and $t \geq 0$.
\end{proposition}
\begin{proof}
Let $\gb \in D(\El_0)$ and $t > 0$. Let $R > \varpi_0$. Let $\Gamma_0^R$ be as in Proposition~\ref{p:semigroupgen}. Let $\Gamma_4^R$ and $\Gamma_5^R$ be the integration contours depicted in Figure~\ref{fig:deform} connecting $-\ri R + \eta$ to $-\ri R - \frac{3}{4} \varepsilon\gamma$ and $\ri R  - \frac{3}{4} \varepsilon\gamma$ to $\ri R + \eta$, respectively. Let $\Gamma^R$ be the closed contour consisting of $-\Gamma_0^R$, $\Gamma_1^R$, $\Gamma_2$, $\Gamma_3^R$, $\Gamma_4^R$ and $\Gamma_5^R$, so that $\Gamma^R$ is oriented clockwise, cf.~Figure~\ref{fig:deform}. By Assumption~\ref{assD1} and Proposition~\ref{p:highfreqresolvdecomp} $\Gamma^R$, as well as its interior, lies in $\rho(\El_0)$, provided $\varpi_0 > 0$ is large enough. Moreover, the map $\rho(\El_0) \to C_{\mathrm{ub}}(\R)$ given by $\lambda \mapsto \re^{\lambda t} (\lambda - \El_0)^{-1} \gb$ is analytic. Hence, Cauchy's integral theorem yields
\begin{align} \label{e:Cauchy}
0 = \int_{\Gamma^R} \re^{\lambda t} (\lambda - \El_0)^{-1} \gb \,\de \lambda.
\end{align}

We express the contribution of the complex line integral over $\Gamma_4^R \cup \Gamma_5^R$ as
\begin{align} \label{e:Cauchy0}
\int_{\Gamma_4^R \cup \Gamma_5^R} \re^{\lambda t} (\lambda - \El_0)^{-1} \gb \,\de \lambda = \int_{\Gamma_4^R \cup \Gamma_5^R} \frac{\re^{\lambda t}}{\lambda} \left((\lambda - \El_0)^{-1} \El_0 \gb + \gb\right) \, \de \lambda.
\end{align}
Lemma~\ref{lem:resolventsimple} and Proposition~\ref{p:highfreqresolvdecomp} yield an $R$-independent constant $C > 0$ such that we have the bound $\|(\lambda - \El_0)^{-1}\|_{\mathcal{B}(C_{\mathrm{ub}}(\R))} \leq C$ for $\lambda \in \Gamma_4^R \cup \Gamma_5^R$. Since the length of $\Gamma_4^R \cup \Gamma_5^R$ can be bounded by an $R$-independent constant $M > 0$, we find that~\eqref{e:Cauchy0} implies
\begin{align*}
\left\|\lim_{R \to \infty} \int_{\Gamma_4^R \cup \Gamma_5^R} \re^{\lambda t} (\lambda - \El_0)^{-1} \gb \,\de\lambda\right\|_{L^\infty} \leq \lim_{R \to \infty} \re^{\eta t} M \frac{C\|\El_0 \gb\|_{L^\infty} + \|\gb\|_{L^\infty}}{R} = 0.
\end{align*}
Combining the latter with Proposition~\ref{p:semigroupgen} and identity~\eqref{e:Cauchy}, we arrive at~\eqref{e:Cauchy1}, which concludes the proof.
\end{proof}

\begin{figure}[ht]
\centering
\begin{subfigure}[b]{.4\textwidth}
\centering
\includegraphics[width=\linewidth]{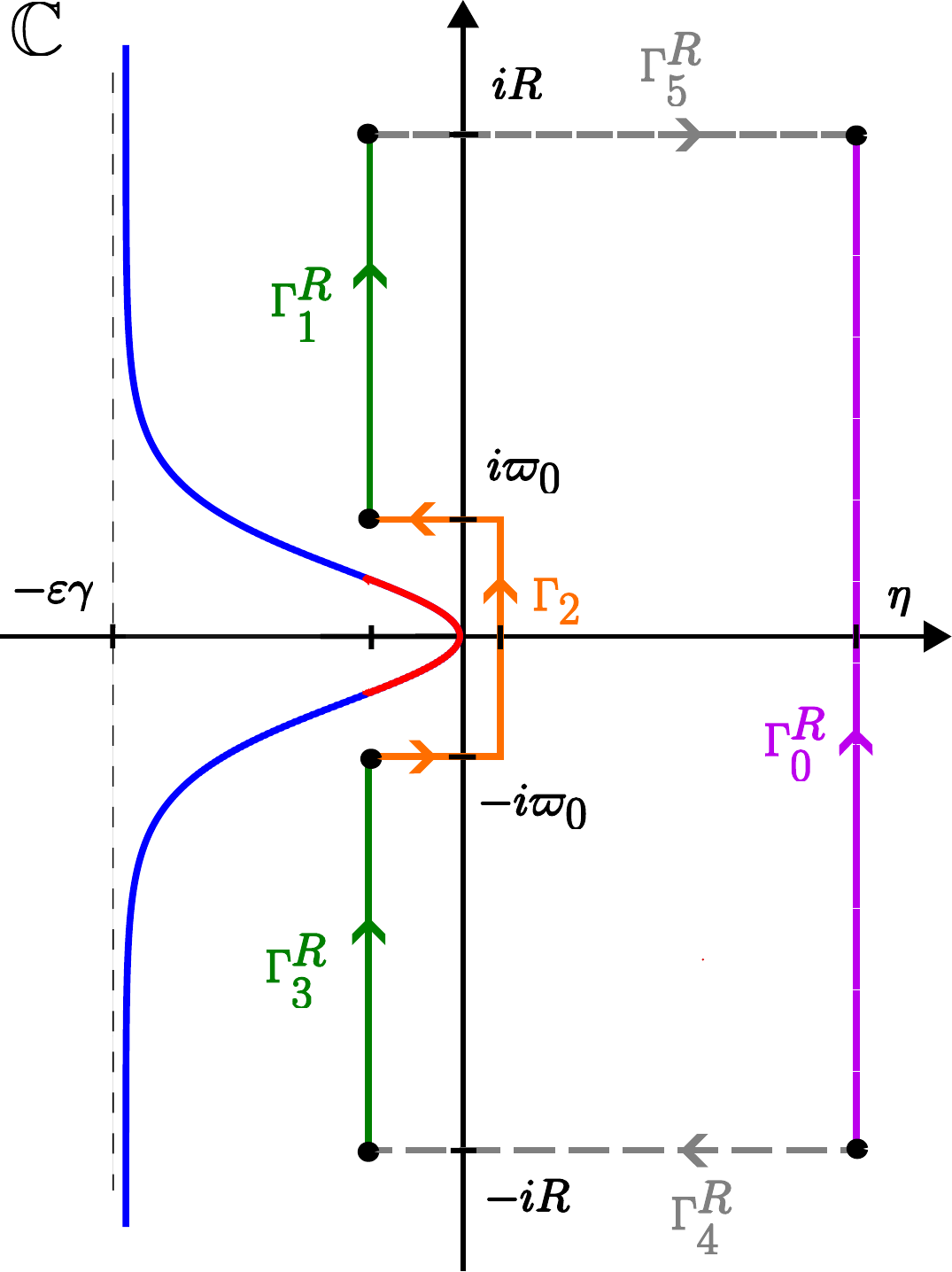}
\end{subfigure} \hspace{0.1\textwidth}
\begin{subfigure}[b]{.47\textwidth}
\centering
\includegraphics[width=\linewidth]{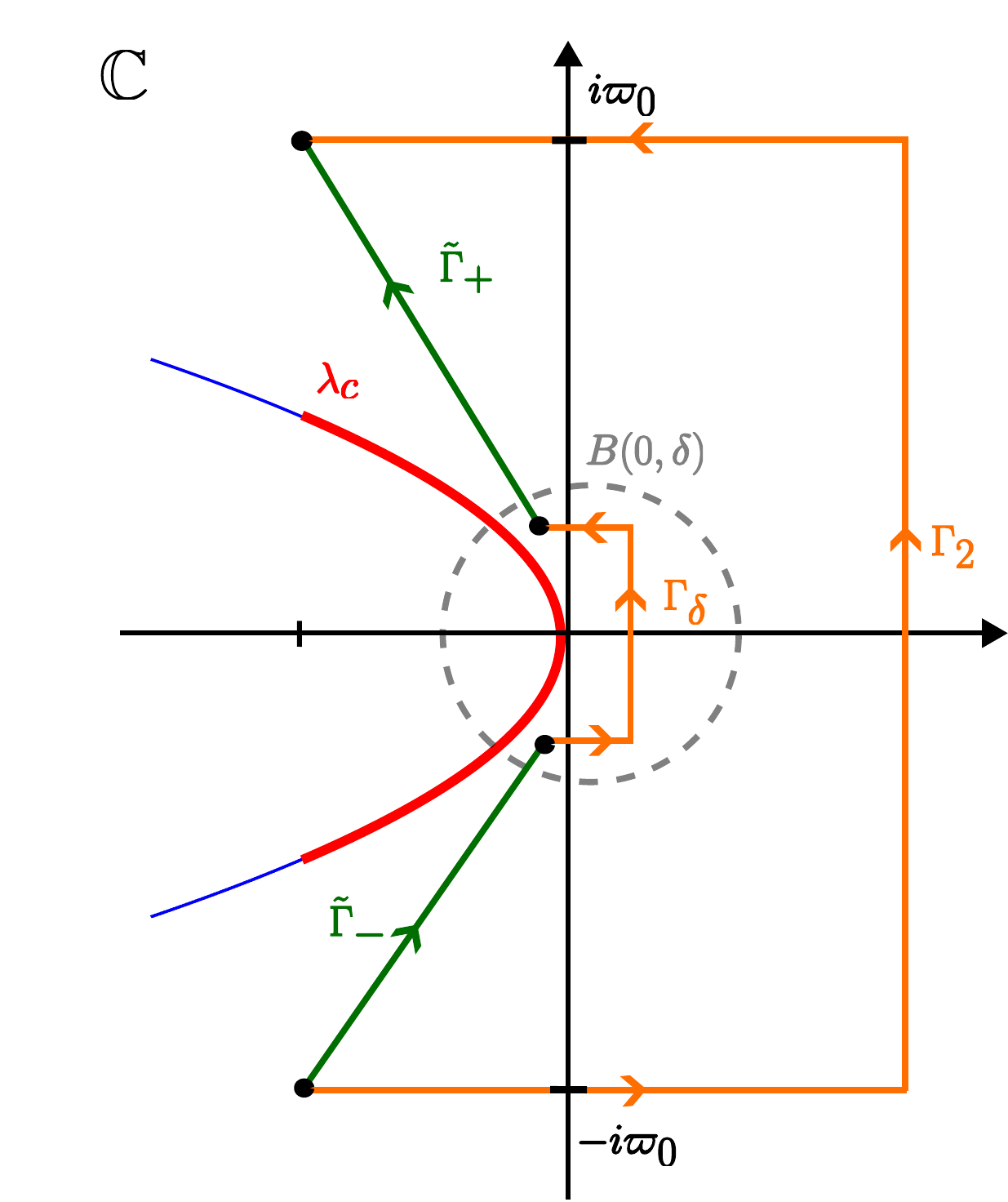}
\end{subfigure}
\caption{The spectrum of the linearization $\El_0$ of system~\eqref{FHN_co} about the wave train $\phi_0$ (depicted in blue and red) touches the origin in a quadratic tangency. It asymptotes to the line $\Re(\lambda) = -\varepsilon\gamma$. The red part of the spectrum is the critical curve $\{\lambda_c(\xi) : \xi \in \R \cap V_1\}$ established in Proposition~\ref{prop:speccons}. Left panel: the original contour $\Gamma_0^R$ used in the inverse Laplace representation~\eqref{laplace2} of the $C_0$-semigroup $\re^{\El_0 t}$, together with the deformed contour $\Gamma_4^R \cup \Gamma_3^R \cup \Gamma_2 \cup \Gamma_1^R \cup \Gamma_5^R$. The contributions of the inverse Laplace integral over $\Gamma_4^R$ and $\Gamma_5^R$ vanish as $R \to \infty$, cf.~Proposition~\ref{deform}. Right panel: a zoom-in on the contour $\Gamma_2$, as well as its deformation $\smash{\widetilde{\Gamma}_- \cup \Gamma_\delta \cup \widetilde{\Gamma}_+}$ used in the proof of Proposition~\ref{p:isolate1}. The rectangular contour $\Gamma_\delta$ lies in the ball $B(0,\delta)$, is reflection symmetric in the real axis and connects points $-\ri\eta_2 - \eta_1$ to $\ri \eta_2 - \eta_1$ with $\eta_{1,2} > 0$.}
\label{fig:deform}
\end{figure}

\subsection{Estimates on the high-frequency part}

We utilize the resolvent expansion obtained in Proposition~\ref{p:highfreqresolvdecomp} to establish exponential decay of the high-frequency part of the semigroup $\re^{\El_0 t}$, which corresponds to the complex line integrals over the contours $\Gamma_1^R$ and $\Gamma_3^R$ in the inverse Laplace representation~\eqref{e:Cauchy1} of the semigroup. 

\begin{proposition} \label{p:highfreqest}
Assume~\ref{assH1} and~\ref{assD1}-\ref{assD2}. Consider $\El_0$ as an operator on $C_{\mathrm{ub}}(\R)$. For each $\varpi_0 > 0$ sufficiently large there exist constants $C, \alpha > 0$ such that the operator $S_e^1(t) \colon C_{\mathrm{ub}}(\R) \to C_{\mathrm{ub}}(\R)$ given by
\begin{align*}
S_e^1(t) \gb = \chi(t) \lim_{R \to \infty} \frac{1}{2 \pi \ri} \int_{\Gamma_1^R \cup \Gamma_3^R} \re^{\lambda t} (\lambda - \El_0)^{-1} \gb \,\de\lambda
\end{align*}
for $\gb \in D(\El_0)$ and $t \geq 0$ obeys the estimate
\begin{align} \label{e:se1estimate}
\|S_e^1(t)\gb\|_{L^\infty} \leq C\re^{-\alpha t}\|\gb\|_{L^\infty}
\end{align}
for $\gb \in C_{\mathrm{ub}}(\R)$ and $t \geq 0$. 
\end{proposition}
\begin{proof}
Let $\gb = (g_1,g_2)^\top \in D(\El_0)$ and $t \geq 0$. We abbreviate $b_1 = -\frac{3}{4} \varepsilon\gamma$. Employing the high-frequency resolvent expansion from Proposition~\ref{p:highfreqresolvdecomp}, we arrive, provided $\varpi_0 > 0$ is sufficiently large, at the decomposition
\begin{align} \label{e:decompse1}
\begin{split}
S_e^1(t) \gb &= \chi(t) \lim_{R \to \infty} \frac{1}{2 \pi} \left(\int_{-R}^{-\varpi_0} + \int_{\varpi_0}^R\right) \re^{\ri \varpi t + b_1 t} \left(b_1 + \ri \varpi - \El_0\right)^{-1} \gb \,\de\varpi\\ 
&= \re^{b_1 t} \left(S_1(t) \gb + S_2(t) \gb + S_3(t) \gb + S_4(t) \gb\right),
\end{split}
\end{align}
where we denote
\begin{align*}
S_j(t)\gb = \chi(t) \lim_{R \to \infty} \frac{1}{2 \pi} \left(\int_{-R}^{-\varpi_0} + \int_{\varpi_0}^R\right) \re^{\ri \varpi t} I^j_{b_1,\varpi}\gb \,\de \varpi, \qquad j = 1,\ldots,4.
\end{align*}
The estimate on $I^4_{b_1,\varpi} \gb$ in Proposition~\ref{p:highfreqresolvdecomp} readily provides $\gb$- and $t$-independent constants $C_{1,2}>0$ such that
\begin{align} \label{e:Cauchy4}
\|S_4(t)\gb\|_{L^\infty} \leq C_1 \int_{\varpi_0}^\infty \varpi^{-\frac32} \|\gb\|_{L^\infty} \,\de \varpi \leq C_2\|\gb\|_{L^\infty}.
\end{align}

We relate the leading-order contributions $S_1(t),S_2(t),S_3(t)$ to (convolutions of) the $C_0$-semi- groups $T_1(t) := \re^{\El_1 t}$ and $T_2(t) := \re^{(\El_2 - b_1) t}$ using~\cite[Proposition~3.12.1]{arendt} and Corollary~\ref{convolution_semigroup}. To this end, we define an $R$-independent contour $\check{\Gamma}_2$, which connects $-\ri\varpi_0$ to $\ri \varpi_0$ and lies in $\Sigma := \{\lambda \in \C \setminus \{0\} : -\frac{1}{4}\varepsilon\gamma < \Re(\lambda) < \frac{1}{8}\varepsilon\gamma, |\arg(\lambda)| < \frac{3\pi}{4}\}$. Moreover, let $\check{\Gamma}_4^R$ and $\check{\Gamma}_5^R$ be the lines connecting $-\ri R + \frac{1}{4}\varepsilon\gamma$ with $-\ri R$ and connecting $\ri R$ with $\ri R + \frac{1}{4}\varepsilon\gamma$, respectively. Using that the maps $\Sigma \to C_{\mathrm{ub}}(\R)$ given by $\lambda \mapsto (\lambda - \El_1)^{-1}\gb$ and $\lambda \mapsto (\lambda + b_1 - \El_2)^{-1}$ are holomorphic by Lemma~\ref{lem:resolventsimple}, Cauchy's integral theorem yields
\begin{align} \label{e:Cauchy3}
\begin{split}
\frac{\chi(t)}{2\pi} \left(\int_{-R}^{-\varpi_0} + \int_{\varpi_0}^R\right) \re^{\ri \varpi t} I^j_{b_1,\varpi}\gb \,\de \varpi &= \frac{\chi(t)}{2 \pi\ri} \int_{\frac{1}{4}\varepsilon \gamma - \ri R}^{\frac{1}{4}\varepsilon \gamma + \ri R} \re^{\lambda t} I^j_{b_1,\lambda} \gb \,\de \lambda\\ 
&\qquad - \, \frac{\chi(t)}{2 \pi\ri} \int_{\check{\Gamma}_2 \cup \check{\Gamma}_4^R \cup \check{\Gamma}_5^R} \re^{\lambda t} I^j_{b_1,\lambda} \gb \,\de \lambda.
\end{split}
\end{align}

We note that the length of the contours $\check{\Gamma}_2, \check{\Gamma}_4^R,\check{\Gamma}_5^R \subset \Sigma \subset \rho(\El_1) \cap \rho(\El_2 - b_1)$ can be bounded by an $R$-independent constant. So, using the resolvent estimates from Lemma~\ref{lem:resolventsimple}, we establish a $t$-, $R$- and $\gb$-independent constant $C_3 > 0$ such that
\begin{align} \label{e:Cauchy2}
\left\|\chi(t)\int_{\check{\Gamma}_2 \cup \check{\Gamma}_4^R \cup \check{\Gamma}_5^R} \re^{\lambda t} I^j_{b_1,\lambda} \gb \,\de \lambda\right\|_{L^\infty} \leq C_3 \re^{\frac14 \varepsilon\gamma t} \|\gb\|_{L^\infty} 
\end{align}
for $j = 1,2,3$. 

Lemma~\ref{lem:resolventsimple} implies that $g_1 \in D(\El_1)$, $g_2 \in D(\El_2 - b_1)$ and the semigroups $T_1(t)$ and $T_2(t)$ are strongly continuous and exponentially bounded with growth bounds $\varpi_0(T_1) \leq 0$ and $\varpi_0(T_2) \leq -\frac{1}{4} \varepsilon\gamma$. Hence, an application of~\cite[Proposition~3.12.1]{arendt} and Corollary~\ref{convolution_semigroup} yields
\begin{align*}
\frac{\chi(t)}{2 \pi\ri} \lim_{R \to \infty} \int_{\frac{1}{4}\varepsilon \gamma - \ri R}^{\frac{1}{4}\varepsilon \gamma + \ri R} \re^{\lambda t} I^1_{b_1,\lambda} \gb \,\de \lambda &= \chi(t) \begin{pmatrix} T_1(t) g_1 \\ T_2(t) g_2 \end{pmatrix}, \\
\frac{\chi(t)}{2 \pi\ri} \lim_{R \to \infty} \int_{\frac{1}{4}\varepsilon \gamma - \ri R}^{\frac{1}{4}\varepsilon \gamma + \ri R} \re^{\lambda t} I^2_{b_1,\lambda} \gb \,\de \lambda &= \chi(t) \begin{pmatrix} \left(T_1 \ast T_2 \right)(t) g_2 \\ -\varepsilon\left(T_2 \ast T_1\right)(t) g_1 \end{pmatrix},
\end{align*}
and
\begin{align*}
\frac{\chi(t)}{2 \pi\ri} \lim_{R \to \infty} \int_{\frac{1}{4}\varepsilon \gamma - \ri R}^{\frac{1}{4}\varepsilon \gamma + \ri R} \re^{\lambda t} I^3_{b_1,\lambda} \gb \,\de \lambda &= \chi(t) \begin{pmatrix} 0 \\ -\varepsilon\left(T_2 \ast T_1 \ast T_2 \right)(t) g_2 \end{pmatrix}.
\end{align*}
By~\cite[Theorem~C.17]{nagel}, the convolutions $T_1 \ast T_2$, $T_2 \ast T_1$ and $T_2 \ast T_1 \ast T_2$ are strongly continuous and exponentially bounded with growth bounds being at most $\max\{\varpi_0(T_1),\varpi_0(T_2)\} \leq 0$. Therefore, we find a $t$- and $\gb$-independent constant $C_4 > 0$ such that
\begin{align*}
\left\|\frac{\chi(t)}{2 \pi\ri} \lim_{R \to \infty} \int_{\frac{1}{4}\varepsilon \gamma - \ri R}^{\frac{1}{4}\varepsilon \gamma + \ri R} \re^{\lambda t} I^j_{b_1,\lambda} \gb \,\de \lambda\right\| \leq C_4 \|\gb\|_{L^\infty}
\end{align*}
for $j = 1,2,3$. Combining the latter with the decompositions~\eqref{e:decompse1} and~\eqref{e:Cauchy3} and the estimates~\eqref{e:Cauchy4} and~\eqref{e:Cauchy2}, we arrive at~\eqref{e:se1estimate} with $\alpha = \frac{1}{2} \varepsilon\gamma > 0$ by density of $D(\El_0)$ in $C_{\mathrm{ub}}(\R)$. 
\end{proof}

\begin{remark}
Comparing the proof of Proposition~\ref{p:highfreqest} with the high-frequency analysis of the semigroup in~\cite[Appendix~B.2]{BjoernAvery}, we find that the identification of the critical high-frequency part of the semigroup as convolutions of the the heat and translation semigroups simplifies the analysis significantly. In particular, it is no longer necessary to compute the inverse Laplace transform of the leading-order terms of the Neumann-series expansion of the resolvent explicitly for a test function $\gb$.
\end{remark}

\subsection{Isolating the critical low-frequency part}

We wish to employ the decomposition of the resolvent $(\lambda- \El_0)^{-1}$ for $|\lambda|$ sufficiently small established in Proposition~\ref{p:lowfreqresolvdecomp} to isolate the critical low-frequency part of the semigroup. To this end, we deform the contour $\Gamma_2$ in the inverse Laplace representation~\eqref{e:Cauchy1} of the semigroup $\re^{\El_0 t}$, so that its part in the right-half plane is contained in the ball $B(0,\delta)$, where Proposition~\ref{p:lowfreqresolvdecomp} applies, cf.~Figure~\ref{fig:deform}. The remainder of the deformed contour lies in the open left-half plane, away from the spectrum of $\El_0$ and, thus, the associated complex line integrals are exponentially decaying. 

\begin{proposition} \label{p:isolate1}
Assume~\ref{assH1} and~\ref{assD1}-\ref{assD3}. Consider $\El_0$ as an operator on $C_{\mathrm{ub}}(\R)$. Let $\varpi_0 > 0$. For each $\delta > 0$ sufficiently small there exist constants $C,\alpha > 0$, a linear operator $S_e^2(t) \colon C_{\mathrm{ub}}(\R) \to C_{\mathrm{ub}}(\R)$ and a rectangular contour $\Gamma_\delta$, which is reflection symmetric in the real axis, lies in $B(0,\delta)$ strictly to the right of $\sigma(\El_0)$ and connects points $-\ri \eta_2 - \eta_1$ and $\ri \eta_2 - \eta_1$ with $\eta_{1,2} > 0$, such that we have the decomposition 
\begin{align} \label{e:Cauchy5}
\frac{\chi(t)}{2 \pi \ri} \int_{\Gamma_2} \re^{\lambda t} (\lambda - \El_0)^{-1} \gb \,\de\lambda = \frac{\chi(t)}{2 \pi \ri} \int_{\Gamma_\delta} \re^{\lambda t} (\lambda - \El_0)^{-1} \gb\,\de\lambda + S_e^2(t)\gb,
\end{align}
for each $\gb \in D(\El_0)$ and $t \geq 0$ and the estimate
\begin{align} \label{e:se2estimate}
\|S_e^2(t)\gb\|_{L^\infty} \leq C\re^{-\alpha t}\|\gb\|_{L^\infty}
\end{align}
holds for $\gb \in C_{\mathrm{ub}}(\R)$ and $t \geq 0$.
\end{proposition}
\begin{proof}
Let $\gb \in D(\El_0)$ and $t \geq 0$. By Proposition~\ref{prop:speccons}, there exist constants $a \in \R$ and $b, \delta_0 > 0$ such that the spectrum of $\El_0$ in the ball $B(0,\delta_0)$ lies on or to the left of the parabola $\{\ri a \xi - b\xi^2 : \xi \in \R\}$. Take $\delta \in (0,\delta_0)$. By Assumption~\ref{assD1}, there exists a constant $\varrho > 0$ such that the spectrum of $\El_0$ in the compact set $K_0 = \{\lambda \in \C : |\Im(\lambda)| \leq 2\varpi_0, |\Re(\lambda)| \leq \gamma\} \setminus B(0,\delta)$ lies to the left of the line $\Re(\lambda) = -\varrho$. Furthermore, the contour $\Gamma_2$ lies in the resolvent set of $\El_0$ by Proposition~\ref{deform}. We conclude that there exist points $-\eta_1 \pm \ri \eta_2$ with $\eta_{1,2} > 0$ lying in $B(0,\delta)$ strictly to the right of $\sigma(\El_0)$, as well as contours $\smash{\widetilde{\Gamma}_-}$, connecting the lower end point $-\ri \varpi_0 - \frac34 \varepsilon \gamma$ of $\Gamma_2$ to the point $- \ri \eta_2 - \eta_1$, and $\smash{\widetilde{\Gamma}_+}$, connecting $\ri \eta_2 - \eta_1$ to the upper end point $\ri \varpi_0 - \frac34 \varepsilon\gamma$ of $\Gamma_2$, such that $\smash{\widetilde{\Gamma}_-}$ and $\smash{\widetilde{\Gamma}_+}$ are both contained in the resolvent set $\rho(\El_0)$ and in the open left-half plane. Hence, there exists a rectangular contour $\Gamma_\delta$, which connects $-\ri\eta_2 - \eta_1$ to $\ri\eta_2 -\eta_1$, is reflection symmetric in the real axis and lies in $B(0,\delta)$, strictly to the right of $\sigma(\El_0)$. Since the map $\rho(\El_0) \to C_{\mathrm{ub}}(\R)$ given by $\lambda \mapsto \re^{\lambda t} (\lambda - \El_0)^{-1} \gb$ is analytic, Cauchy's integral theorem yields~\eqref{e:Cauchy5} with
\begin{align*}
 S_e^2(t)\gb = \frac{\chi(t)}{2 \pi \ri} \left( \int_{\widetilde{\Gamma}_-} + \int_{\widetilde{\Gamma}_+}\right) \re^{\lambda t} (\lambda - \El_0)^{-1} \gb \,\de\lambda.
\end{align*}
The analytic map $\rho(\El_0) \to C_{\mathrm{ub}}(\R), \lambda \mapsto (\lambda - \El_0)^{-1}$ is bounded on the compact sets $\widetilde{\Gamma}_\pm \subset \rho(\El_0)$, which lie in the open left-half plane. Thus, the estimate~\eqref{e:se2estimate} follows by density of $D(\El_0)$ in $C_{\mathrm{ub}}(\R)$.
\end{proof}

We can now identify the critical part of the remaining complex line integral in~\eqref{e:Cauchy5} by employing the low-frequency decomposition of the resolvent obtained in Proposition~\ref{p:lowfreqresolvdecomp} and using the identity~\eqref{e:spectralprojid} derived in Lemma~\ref{expansion_in_laplace}.

\begin{proposition} \label{p:isolate2}
Assume~\ref{assH1},~\ref{assH2} and~\ref{assD1}-\ref{assD3}. Consider $\El_0$ as an operator on $C_{\mathrm{ub}}(\R)$. For each $\delta > 0$ sufficiently small there exist constants $C,\alpha > 0$ and a linear operator $S_e^3(t) \colon C_{\mathrm{ub}}(\R) \to C_{\mathrm{ub}}(\R)$ such that for each $\gb \in D(\El_0)$, $\zeta \in \R$ and $t \geq 0$ we have the decomposition 
\begin{align} \label{e:Cauchy6}
\begin{split}
\frac{\chi(t)}{2 \pi \ri} \int_{\Gamma_\delta} \re^{\lambda t} (\lambda - \El_0)^{-1} \gb \,\de\lambda\left[\zeta\right] &= \frac{\chi(t)}{2 \pi \ri} \int_\zeta^\infty \int_{\Gamma_\delta} \re^{\lambda t + \nu_c(\lambda)(\zeta - \zb)} \Psi(\zeta,\lambda) \tilde{\Psi}(\zb,\lambda)^* \,\de \lambda \, \gb(\zb) \,\de \zb\\ 
&\qquad + \, \left(S_e^3(t)\gb\right)[\zeta].
\end{split}
\end{align}
Moreover, the estimate
\begin{align} \label{e:se3estimate}
\|S_e^3(t)\gb\|_{L^\infty} \leq C\re^{-\alpha t}\|\gb\|_{L^\infty}
\end{align}
holds for $\gb \in C_{\mathrm{ub}}(\R)$ and $t \geq 0$.
\end{proposition}
\begin{proof}
Provided $\delta > 0$ is sufficiently small, identity~\eqref{e:Cauchy6} follows readily from Fubini's theorem, Proposition~\ref{p:lowfreqresolvdecomp} and Lemma~\ref{expansion_in_laplace} by setting
\begin{align*}
 S_e^3(t) \gb = \frac{\chi(t)}{2 \pi \ri} \int_{\Gamma_\delta} \re^{\lambda t} S_e^0(\lambda) \gb \,\de \lambda
\end{align*}
for $t \geq 0$ and $\gb \in D(\El_0)$, where $S_e^0 \colon B(0,\delta) \to \mathcal{B}(C_{\mathrm{ub}}(\R))$ is the analytic map from Proposition~\ref{p:lowfreqresolvdecomp}, obeying the estimate 
\begin{align} \label{e:expresolvest}
 \|S_e^0(\lambda) \gb\|_{L^\infty} \leq C_0 \|\gb\|_{L^\infty} 
\end{align}
for some $\gb$- and $\lambda$-independent constant $C_0 > 0$. Now let $\widetilde{\Gamma}_\delta$ be the straight line connecting the end points $\pm \ri \eta_2 - \eta_1$ of $\Gamma_\delta$. Then, $\widetilde{\Gamma}_\delta$ lies both in $B(0,\delta)$ and in the open left-half plane. By Cauchy's integral theorem and analyticity of $S_e^0$, we infer
\begin{align*}
  S_e^3(t) \gb = \frac{\chi(t)}{2 \pi \ri} \int_{\widetilde{\Gamma}_\delta} \re^{\lambda t} S_e^0(\lambda) \gb \,\de \lambda
\end{align*}
for $\gb \in D(\El_0)$ and $t \geq 0$. Taking norms in the latter, using that the compact contour $\widetilde{\Gamma}_\delta$ lies in the open left-half plane and applying the bound~\eqref{e:expresolvest} readily yields the estimate~\eqref{e:se3estimate} by density of $D(\El_0)$ in $C_{\mathrm{ub}}(\R)$.
\end{proof}

\subsection{Floquet-Bloch representation of the critical low-frequency part} \label{sec:FloquetBloch}

Except for the integral appearing on the right-hand side of~\eqref{e:Cauchy6} representing its critical low-frequency part, the semigroup $\re^{\El_0 t}$ is exponentially decaying by Propositions~\ref{deform},~\ref{p:highfreqest},~\ref{p:isolate1} and~\ref{p:isolate2} and Lemma~\ref{l:shorttime}. The following result recovers, up to some exponentially decaying terms, the same Floquet-Bloch representation for the critical low-frequency part of the semigroup as in~\cite{BjoernMod}. 

The main idea is to exploit that the integral
\begin{align*}
\int_{\Gamma_\delta} \re^{\lambda t + \nu_c(\lambda)(\zeta - \zb)} \Psi(\zeta,\lambda) \tilde{\Psi}(\zb,\lambda)^* \,\de \lambda
\end{align*}
possesses an integrand, which is analytic in $\lambda$ on $B(0,\delta)$ for each $\zeta,\zb \in \R$ and $t \geq 0$, cf.~Proposition~\ref{prop:speccons} and Lemma~\ref{expansion_in_laplace}. This \emph{pointwise} analyticity\footnote{See~\cite[Section~5.1]{BjoernAvery} for further discussion on pointwise and $L^p$-analyticity.} allows us to shift (part of) the integration contour $\Gamma_\delta$  onto the critical spectral curve $\lambda_c(\xi)$, see Figure~\ref{ident_figure}. Via the identities $\nu_c(\lambda_c(\xi)) = \ri \xi$ and~\eqref{e:spectralprojid2}, obtained in Proposition~\ref{p:Floquet} and Lemma~\ref{expansion_in_laplace}, respectively, we then arrive at the desired Floquet-Bloch representation from~\cite{BjoernMod}. We show that the remainder terms are exponentially decaying by using pointwise estimates obtained through integration by parts, essentially following the same strategy as in~\cite[Lemma~A.1]{HDRS22}. 

\begin{proposition} \label{p:FloquetBlochRep}
Assume~\ref{assH1},~\ref{assH2} and~\ref{assD1}-\ref{assD3}. For each $\delta > 0$ sufficiently small there exist constants $\xi_0, C,\alpha > 0$, a linear operator 
$S_e^4(t) \colon C_{\mathrm{ub}}(\R) \to C_{\mathrm{ub}}(\R)$ and a smooth cut-off function $\rho \colon \R \to \R$ such that for each $\gb \in C_{\mathrm{ub}}(\R)$, $\zeta \in \R$ and $t \geq 0$ we have
\begin{align*} 
\begin{split}
\frac{\chi(t)}{2 \pi \ri} \int_\zeta^\infty \int_{\Gamma_\delta} &\re^{\lambda t + \nu_c(\lambda)(\zeta - \zb)} \Psi(\zeta,\lambda) \tilde{\Psi}(\zb,\lambda)^* \,\de \lambda \, \gb(\zb) \,\de \zb\\ &= \frac{\chi(t)}{2 \pi} \int_\R \int_\R \rho(\xi)\re^{\lambda_c(\xi) t + \ri \xi(\xx - \xt)} \Phi_\xi(\xx)\widetilde{\Phi}_\xi(\zb)^* \,\de\xi  \,\gb(\xt) \,\de\xt + \left(S_e^4(t)\gb\right)[\zeta].
\end{split}
\end{align*}
Moreover, $\rho$ is supported on the interval $(-\xi_0,\xi_0) \subset V_1 \cap \R$ and satisfies $\rho(\xi) = 1$ for $\xi \in \smash{[-\frac12 \xi_0, \frac12 \xi_0]}$. Finally, for each $\gb \in C_{\mathrm{ub}}(\R)$ and $t \geq 0$ it holds
\begin{align*} 
\|S_e^4(t)\gb\|_{L^\infty} \leq C\re^{-\alpha t}\|\gb\|_{L^\infty}.
\end{align*}
\end{proposition}
\begin{proof}
First, we note that Propositions~\ref{prop:speccons} and~\ref{p:Floquet} imply $\nu_c'(0) = -c_g^{-1} \neq 0$ and $\lambda_c'(0) = -\ri c_g \neq 0$. So, using Proposition~\ref{prop:speccons}, we can take $\delta > 0$ so small that Proposition~\ref{p:Floquet} and Lemma~\ref{expansion_in_laplace} apply, it holds $\nu_c'(\lambda) \neq 0$ for all $\lambda \in \overline{B(0,\delta)}$, and each point in $\sigma(\El_0) \cap B(0,\delta)$ lies on the curve $\{\lambda_c(\xi) : \xi \in V_1 \cap \R\}$. In addition, there exists, again by Proposition~\ref{prop:speccons}, $\xi_0 > 0$ such that we have $[-\xi_0,\xi_0] \subset V_1 \cap \R$, it holds $\mathrm{sgn}(\Im(\lambda_c(\pm \xi_0))) = \pm 1$, each point on the curve $\lambda_c([-\xi_0,\xi_0])$ lies in the ball $B(0,\delta)$ and on the rightmost boundary $\{z \in \sigma(\El_0) : z + w \in \rho(\El_0) \text{ for all } w > 0\}$ of the spectrum of $\El_0$, and $\lambda_c'(\xi)$ is nonzero for each $\xi \in [-\xi_0,\xi_0]$. We let $\rho \colon \R \to \R$ be a smooth cut-off function, which is supported on $(-\xi_0,\xi_0)$ and satisfies $\rho(\xi) = 1$ for $\xi \in \smash{[-\frac12 \xi_0, \frac12 \xi_0]}$.

Our approach is to deform the contour $\Gamma_\delta$ into a new contour consisting of a smooth curve $\Gamma_- \subset B(0,\delta) \cap \{z \in \C : \Re(\lambda) < 0\}$ which connects the lower endpoint $-\eta_1 - \ri \eta_2$ of $\Gamma_\delta$ to $\lambda_c(-\xi_0)$ and satisfies $\Gamma_- \setminus \{\lambda_c(-\xi_0)\} \subset \rho(\El_0)$, the smooth curve $\Gamma_c \subset B(0,\delta)$ which connects $\lambda_c(-\xi_0)$ to $\lambda_c(\xi_0)$ and is parameterized by $\lambda_c$, and a smooth curve $\Gamma_+ \subset B(0,\delta) \cap \{z \in \C : \Re(\lambda) < 0\}$ which connects the point $\lambda_c(\xi_0)$ to the upper endpoint $-\eta_1 + \ri \eta_2$ of $\Gamma_\delta$ and satisfies $\Gamma_+ \setminus \{\lambda_c(\xi_0)\} \subset \rho(\El_0)$, see Figure~\ref{ident_figure}. We note that the contours $\Gamma_\pm$ exist, because the points $-\eta_1 \pm \ri \eta_2$ lie in the open left-half plane strictly to the right of $\sigma_0(\El_0) \cap B(0,\delta)$, it holds $\mathrm{sgn}(\Im(\lambda_c(\pm \xi_0))) = \pm 1$, and each point on the curve $\lambda_c([-\xi_0,\xi_0])$ lies in the ball $B(0,\delta)$ and on the rightmost boundary $\{z \in \sigma(\El_0) : z + w \in \rho(\El_0) \text{ for all } w > 0\}$ of the spectrum of $\El_0$, which lies in $\{z \in \C : \Re(z) < 0\} \cup \{0\}$ by assumption~\ref{assD1}. 

We choose parameterizations $\lambda_\pm \colon [0,1] \to \C$ of the curves $\Gamma_\pm$ satisfying $\lambda_\pm'(\xi) \neq 0$ for $\xi \in [0,1]$. Since $\nu_c$ and $\smash{\Psi(\zeta,\cdot) \tilde{\Psi}(\zb,\cdot)^*}$ are analytic and it holds $\ri \smash{\Phi_\xi(\zeta)\widetilde{\Phi}_\xi(\zb)^* = \Psi(\zeta,\lambda_c(\xi))\tilde{\Psi}(\zb,\lambda_c(\xi))^* \lambda_c'(\xi)}$ for each $\zeta,\zb \in \R$ and $\xi \in (-\xi_0,\xi_0)$ by Proposition~\ref{p:Floquet} and Lemma~\ref{expansion_in_laplace}, Cauchy's integral theorem implies
\begin{align} \label{e:Cauchy8}
\begin{split}
\int_{\Gamma_\delta} \re^{\lambda t + \nu_c(\lambda)(\zeta - \zb)} \Psi(\zeta,\lambda) \tilde{\Psi}(\zb,\lambda)^* \,\de \lambda &= \left(\int_{\Gamma_-} + \int_{\Gamma_c} + \int_{\Gamma_+}\right) \re^{\lambda t + \nu_c(\lambda)(\zeta - \zb)} \Psi(\zeta,\lambda) \tilde{\Psi}(\zb,\lambda)^* \,\de \lambda\\
&= \ri \int_\R \rho(\xi)\re^{\lambda_c(\xi) t + \ri \xi(\xx - \xt)} \Phi_\xi(\zeta)\widetilde{\Phi}_\xi(\zb)^* \,\de\xi + I_+ + I_- + I_c 
\end{split}
\end{align}
where we denote
\begin{align*}
I_\pm &= \int_{\Gamma_\pm} \re^{\lambda t + \nu_c(\lambda)(\zeta - \zb)} \Psi(\zeta,\lambda) \tilde{\Psi}(\zb,\lambda)^* \,\de \lambda, \qquad I_c = \ri \int_{-\xi_0}^{\xi_0} (1-\rho(\xi))\re^{\lambda_c(\xi) t + \ri \xi(\xx - \xt)} \Phi_\xi(\zeta)\widetilde{\Phi}_\xi(\zb)^* \,\de\xi
\end{align*}
for $\xx,\xt \in \R$ and $t \geq 0$. On the other hand, using again $\smash{\ri \Phi_\xi(\zeta)\widetilde{\Phi}_\xi(\zb)^* = \Psi(\zeta,\lambda_c(\xi))\tilde{\Psi}(\zb,\lambda_c(\xi))^* \lambda_c'(\xi)}$, we infer
\begin{align} \label{e:Cauchy9}
\ri \int_\R \rho(\xi)\re^{\lambda_c(\xi) t + \ri \xi(\xx - \xt)} \Phi_\xi(\zeta)\widetilde{\Phi}_\xi(\zb)^* \,\de\xi = I_0 - I_c,
\end{align}
where we denote
\begin{align*}
I_0 = \int_{\Gamma_c} \re^{\lambda t + \nu(\lambda)(\xx-\xt)}\Psi(\zeta,\lambda)\tilde{\Psi}(\zb,\lambda)^* \,\de\lambda
\end{align*}
for $\xx,\xt \in \R$ and $t \geq 0$. All in all,~\eqref{e:Cauchy8} and~\eqref{e:Cauchy9} yield the decomposition
\begin{align}
\label{e:Cauchy_intermediate}
\begin{split}
    \ri \int_\R \rho(\xi)\re^{\lambda_c(\xi) t + \ri \xi(\xx - \xt)} \Phi_\xi(\zeta)\widetilde{\Phi}_\xi(\xt)^* \,\de\xi 
    &= \mathbf{1}_{[\zeta,\infty)}(\zb)\int_{\Gamma_\delta} \re^{\lambda t + \nu(\lambda)(\xx-\xt)}\Psi(\zeta,\lambda)\tilde{\Psi}(\xt,\lambda)^* \,\de\lambda\\ 
    &\qquad - \, \mathbf{1}_{[\zeta,\infty)}(\zb)\left(I_+ + I_- + I_c \right) + \mathbf{1}_{(-\infty,\zeta]}(\zb) \left(I_0-I_c\right)
    \end{split}
\end{align}
for $\xx,\xt \in \R$ and $t\geq 0$. We will  use integration by parts to establish pointwise approximations of $I_\pm$, $I_0$ and $I_c$, which yield integrability in space and exponential decay in time of $\mathbf{1}_{[\zeta,\infty)}(\zb)\left(I_+ + I_- + I_c \right)$ and of $\mathbf{1}_{(-\infty,\zeta]}(\zb) \left(I_0-I_c\right)$. This then readily leads to the desired result.

\paragraph*{Pointwise approximations of $I_\pm$ for $\xx \leq \xt$.} We wish to factor out the space-integrable quotient $(1+(\xx-\xt)^2)^{-1}$ by establishing pointwise approximations of $I_+$ and $(\xx-\xt)^2 I_+$. Recalling $\nu_c'(\lambda) \neq 0$ for all $\lambda \in B(0,\delta)$, abbreviating $\Psi_1(\xx,\xt,\lambda) = \Psi(\zeta,\lambda)\tilde{\Psi}(\zb,\lambda)^*/\nu_c'(\lambda)$ and using integration by parts and Proposition~\ref{p:Floquet}, we rewrite
\begin{align*}
(\xx-\xt)^2 I_+ &= \int_0^1  (\xx-\xt)\re^{\lambda_+(\xi) t}  \Psi_1(\xx,\xt,\lambda_+(\xi)) \partial_\xi \left(\re^{\nu(\lambda_+(\xi))(\xx-\xt)}\right) \,\de\xi\\
&= \left[(\xx-\xt)\re^{\lambda_+(\xi) t + \nu(\lambda_+(\xi))(\xx-\xt)}  \Psi_1(\xx,\xt,\lambda_+(\xi))\right]_{\xi = 0}^1\\ 
&\qquad \qquad - \, \int_0^1 \partial_\xi\left((\xx-\xt)\re^{\lambda_+(\xi) t}  \Psi_1(\xx,\xt,\lambda_+(\xi)) \right)\re^{\nu(\lambda_+(\xi))(\xx-\xt)} \,\de\xi\\
&= (\xx - \xt)\re^{(-\eta_1 + \ri\eta_2) t + \nu(-\eta_1 + \ri\eta_2)(\xx-\xt)}  \Psi_1(\xx,\xt,-\eta_1 + \ri\eta_2)\\ 
&\qquad - \, \int_0^1 \partial_\xi\left((\xx-\xt)\re^{\lambda_+(\xi) t}  \Psi_1(\xx,\xt,\lambda_+(\xi)) \right)\re^{\nu(\lambda_+(\xi))(\xx-\xt)} \,\de\xi\\
&\qquad - \, (\xx - \xt)\re^{\lambda_c(\xi_0) t + \ri\xi_0(\xx-\xt)}  \Psi_1(\xx,\xt,\lambda_c(\xi_0))\\
&=: II_+ + III_+ - (\xx - \xt)\re^{\lambda_c(\xi_0) t + \ri\xi_0(\xx-\xt)}  \Psi_1(\xx,\xt,\lambda_c(\xi_0)).
\end{align*}
Abbreviating $\Psi_2(\xx,\xt,\lambda) = \Psi_1(\xx,\xt,\lambda)/\nu_c'(\lambda)$ and $\Psi_3(\xx,\xt,\lambda) = \partial_\lambda \Psi_1(\xx,\xt,\lambda)/\nu_c'(\lambda)$ and integrating by parts once again, we arrive at
\begin{align*}
 III_+ &= -\int_0^1 \re^{\lambda_+(\xi) t} \left(t\Psi_2(\xx,\xt,\lambda_+(\xi)) + \Psi_3(\xx,\xt,\lambda_+(\xi))\right) \partial_\xi\left(\re^{\nu(\lambda_+(\xi))(\xx-\xt)}\right) \,\de\xi\\
&= \int_0^1  \partial_\xi \left(\re^{\lambda_+(\xi) t} \left(t\Psi_2(\xx,\xt,\lambda_+(\xi)) + \Psi_3(\xx,\xt,\lambda_+(\xi))\right)\right) \re^{\nu(\lambda_+(\xi))(\xx-\xt)} \,\de\xi\\
 &\qquad - \, \left[\re^{\lambda_+(\xi) t + \nu(\lambda_+(\xi))(\xx-\xt)} \left(t\Psi_2(\xx,\xt,\lambda_+(\xi)) + \Psi_3(\xx,\xt,\lambda_+(\xi))\right)\right]_{\xi = 0}^1\\
& = \int_{\Gamma_+} \re^{\lambda t + \nu(\lambda)(\xx-\xt)}\left(t^2 \Psi_2(\xx,\xt,\lambda) + t\left(\partial_\lambda \Psi_2(\xx,\xt,\lambda) + \Psi_3(\xx,\xt,\lambda)\right) + \partial_\lambda \Psi_3(\xx,\xt,\lambda)\right) \,\de \lambda\\
 &\qquad - \, \left[\re^{\lambda_+(\xi) t + \nu(\lambda_+(\xi))(\xx-\xt)} \left(t\Psi_2(\xx,\xt,\lambda_+(\xi)) + \Psi_3(\xx,\xt,\lambda_+(\xi))\right)\right]_{\xi = 0}^1.
\end{align*}
We establish pointwise estimates on the contributions $I_+, II_+$ and $III_+$. Here, we use the following facts which follow with the aid of Proposition~\ref{p:Floquet} and Lemma~\ref{expansion_in_laplace}. First, since the curves $\lambda_\pm$ lie  in the open left-half plane and the points $-\eta_1 \pm \ri \eta_2$ lie strictly to the right of $\sigma(\El_0)$, there exists a constant $\eta_0 > 0$ such that $\Re(\nu(-\eta_1 \pm \ri \eta_2)) \geq \eta_0$ and $\Re(\lambda_\pm(\xi)) \leq -\eta_0$ for all $\xi \in [0,1]$. Second, since the curves $\lambda_\pm$ lie to the right of $\sigma(\El_0)$, it holds $\Re(\nu(\lambda_\pm(\xi))) \geq 0$ for all $\xi \in [0,1]$. Third, the functions $\Psi_i(\xx,\xt,\lambda)$ as well as their derivatives with respect to $\lambda$ are bounded on $\R \times \R \times B(0,\delta)$ for $i = 1,2,3$. Thus, we establish the following pointwise bounds
\begin{align*}
 |II_+| &\lesssim |\xx-\xt| \re^{-\eta_1 t + \eta_0 (\xx - \xt)}, \qquad |I_+|, |III_+| \lesssim (1+t+t^2)\re^{-\eta_0 t}
\end{align*}
for $t \geq 0$ and $\xx, \xt \in \R$ with $\xx - \xt \leq 0$. All in all, we conclude
\begin{align}
\label{e:I+est}
\begin{split}
&\left|I_+ + \frac{(\xx - \xt)\re^{\lambda_c(\xi_0) t + \ri\xi_0(\xx-\xt)}}{1+(\xx-\xt)^2} \Psi_1(\xx,\xt,\lambda_c(\xi_0))\right| \lesssim \frac{(1+t+t^2)\re^{-\eta_0 t} + |\xx-\xt| \re^{-\eta_1 t + \eta_0 (\xx - \xt)}}{1 + \left(\xx - \xt\right)^2}
\end{split}
\end{align}
for $t \geq 0$ and $\xx, \xt \in \R$ with $\xx - \xt \leq 0$

Analogously, one finds
\begin{align}
\label{e:I-est}
\begin{split}
&\left|I_- - \frac{(\xx - \xt)\re^{\lambda_c(-\xi_0) t - \ri\xi_0(\xx-\xt)}}{1+(\xx-\xt)^2} \Psi_1(\xx,\xt,\lambda_c(-\xi_0))\right| \lesssim \frac{(1+t+t^2)\re^{-\eta_0 t} + |\xx-\xt| \re^{-\eta_1 t + \eta_0 (\xx - \xt)}}{1 + \left(\xx - \xt\right)^2},
\end{split}
\end{align}
for $t \geq 0$ and $\xx, \xt \in \R$ with $\xx - \xt \leq 0$. 

\paragraph*{Pointwise approximation of $I_0$ for $\xx \geq \xt$.}  Recalling that the integrand of $I_0$ is analytic in $\lambda$ on $B(0,\delta)$, we can apply Cauchy's integral theorem to deform the contour $\Gamma_c$ to a line $\smash{\widetilde{\Gamma}_c}$ connecting the point $\lambda_c(-\xi_0)$ to $\lambda_c(\xi_0)$. We parameterize the line by a curve $\lambda_0 \colon [0,1] \to \C$ satisfying $\lambda_0'(\xi) \neq 0$ for all $\xi \in [0,1]$, see Figure~\ref{ident_figure}. We proceed similarly as before and factor out the quotient $(1+(\xx-\xt)^2)^{-1}$, which is integrable in space. Thus, using integration by parts and Proposition~\ref{p:Floquet}, we rewrite
\begin{align*}
(\xx-\xt)^2 I_0 &= \int_0^1  (\xx-\xt)\re^{\lambda_0(\xi) t}  \Psi_1(\xx,\xt,\lambda_0(\xi)) \partial_\xi \left(\re^{\nu(\lambda_0(\xi))(\xx-\xt)}\right) \,\de\xi\\
&= (\xx - \xt)\re^{\lambda_c(\xi_0) t + \ri\xi_0(\xx-\xt)}  \Psi_1(\xx,\xt,\lambda_c(\xi_0)) - (\xx - \xt)\re^{\lambda_c(-\xi_0) t - \ri\xi_0(\xx-\xt)}  \Psi_1(\xx,\xt,\lambda_c(-\xi_0))\\
& \qquad - \, \int_0^1 \partial_\xi\left((\xx-\xt)\re^{\lambda_0(\xi) t}  \Psi_1(\xx,\xt,\lambda_0(\xi)) \right)\re^{\nu(\lambda_0(\xi))(\xx-\xt)} \,\de\xi.
\end{align*}
Using integration by parts once again, we establish
\begin{align*}
II_0 &:= -\int_0^1 \partial_\xi\left((\xx-\xt)\re^{\lambda_0(\xi) t}  \Psi_1(\xx,\xt,\lambda_0(\xi)) \right)\re^{\nu(\lambda_0(\xi))(\xx-\xt)} \,\de\xi \\
& = -\int_0^1 \re^{\lambda_0(\xi) t} \left(t\Psi_2(\xx,\xt,\lambda_0(\xi)) + \Psi_3(\xx,\xt,\lambda_0(\xi))\right) \partial_\xi\left(\re^{\nu(\lambda_0(\xi))(\xx-\xt)}\right) \,\de\xi\\
& = \int_{\widetilde\Gamma_c} \re^{\lambda t + \nu(\lambda)(\xx-\xt)}\left(t^2 \Psi_2(\xx,\xt,\lambda) + t\left(\partial_\lambda \Psi_2(\xx,\xt,\lambda) + \Psi_3(\xx,\xt,\lambda)\right) + \partial_\lambda \Psi_3(\xx,\xt,\lambda)\right) \,\de \lambda\\
 &\qquad - \, \left[\re^{\lambda_0(\xi) t + \nu(\lambda_0(\xi))(\xx-\xt)} \left(t\Psi_2(\xx,\xt,\lambda_0(\xi)) + \Psi_3(\xx,\xt,\lambda_0(\xi))\right)\right]_{\xi = 0}^1.
\end{align*}
Since $\widetilde{\Gamma}_c$ is a straight line in $B(0,\delta)$ lying to the left of $\sigma(\El_0)$, it holds $\Re(\lambda_0(\xi)) \leq \Re(\lambda_c(\pm \xi_0)) \leq -\eta_0$ and $\Re(\nu(\lambda_0(\xi))) \leq 0$ for all $\xi \in [0,1]$ by Proposition~\ref{p:Floquet}. Hence, we obtain the following pointwise bounds
\begin{align*}
 |I_0|, |II_0| &\lesssim (1+t+t^2)\re^{-\eta_0 t},
\end{align*}
for $t \geq 0$ and $\xx, \xt \in \R$ with $\xx - \xt \geq 0$. We conclude that
\begin{align}
\label{e:I0est}
\begin{split}
&\left|I_0 - \frac{\xx - \xt}{1+(\xx-\xt)^2}\left(\re^{\lambda_c(\xi_0) t + \ri\xi_0(\xx-\xt)}  \Psi_1(\xx,\xt,\lambda_c(\xi_0)) - \re^{\lambda_c(-\xi_0) t - \ri\xi_0(\xx-\xt)}  \Psi_1(\xx,\xt,\lambda_c(-\xi_0))\right)\right|\\ 
&\qquad \lesssim \frac{(1+t+t^2)\re^{-\eta_0 t}}{1 + \left(\xx - \xt\right)^2},
\end{split}
\end{align}
for $t \geq 0$ and $\xx, \xt \in \R$ with $\xx - \xt \geq 0$.

\paragraph*{Pointwise approximation of $I_c$.} Again our approach is to factor out the quotient $(1+(\xx-\xt)^2)^{-1}$. Recalling $\ri \Phi_\xi(\zeta)\widetilde{\Phi}_\xi(\zb)^* = \Psi(\zeta,\lambda_c(\xi))\tilde{\Psi}(\zb,\lambda_c(\xi))^*\lambda_c'(\xi)$ for $\xi \in [-\xi_0,\xi_0]$, using integration by parts and applying Proposition~\ref{p:Floquet}, we rewrite
\begin{align*}
(\xx-\xt)^2 I_c &= \int_{-\xi_0}^{\xi_0}  \left(1-\rho(\xi)\right) (\xx-\xt)\re^{\lambda_c(\xi) t}  \Psi_1(\xx,\xt,\lambda_c(\xi)) \partial_\xi \left(\re^{\nu(\lambda_c(\xi))(\xx-\xt)}\right) \,\de\xi\\
&= (\xx - \xt)\re^{\lambda_c(\xi_0) t + \ri\xi_0(\xx-\xt)}  \Psi_1(\xx,\xt,\lambda_c(\xi_0)) - (\xx - \xt)\re^{\lambda_c(-\xi_0) t - \ri\xi_0(\xx-\xt)}  \Psi_1(\xx,\xt,\lambda_c(-\xi_0))\\
&\qquad - \, \int_{-\xi_0}^{\xi_0} \partial_\xi\left( \left(1-\rho(\xi)\right) (\xx-\xt)\re^{\lambda_c(\xi) t}  \Psi_1(\xx,\xt,\lambda_c(\xi)) \right)\re^{\ri\xi(\xx-\xt)} \,\de\xi.
\end{align*}
Abbreviating $\tilde{\Psi}_2(\xx,\xt,\xi) = \left(1-\rho(\xi)\right) \Psi_1(\xx,\xt,\lambda_c(\xi))$ and integrating by parts once again, we establish
\begin{align*}
II_c &:= -\int_{-\xi_0}^{\xi_0} \partial_\xi\left((\xx-\xt)\re^{\lambda_c(\xi) t}  \tilde{\Psi}_2(\xx,\xt,\xi) \right)\re^{\ri\xi(\xx-\xt)} \,\de\xi \\
& = \ri\int_{-\xi_0}^{\xi_0} \re^{\lambda_c(\xi) t} \left(\lambda_c'(\xi) t \tilde{\Psi}_2(\xx,\xt,\xi) + \partial_\xi \tilde{\Psi}_2(\xx,\xt,\xi)\right)\partial_\xi\left(\re^{\ri \xi(\xx-\xt)}\right) \,\de\xi\\
& = \int_{-\xi_0}^{\xi_0} \frac{\re^{\lambda_c(\xi) t + \ri \xi(\xx-\xt)}}{\ri} \left(\left(\left(\lambda_c'(\xi)t\right)^2\! + \lambda_c''(\xi)t\right) \tilde{\Psi}_2(\xx,\xt,\xi) + 2\lambda_c'(\xi) t \partial_\xi \tilde{\Psi}_2(\xx,\xt,\xi) +  \partial_\xi^2 \tilde{\Psi}_2(\xx,\xt,\xi)\right)\!\de \xi\\
&\qquad + \, \ri\left[\re^{\lambda_c(\xi) t + \ri\xi(\xx-\xt)} \left(\lambda_c'(\xi) t \tilde{\Psi}_2(\xx,\xt,\xi) + \partial_\xi \tilde{\Psi}_2(\xx,\xt,\xi)\right)\right]_{\xi = -\xi_0}^{\xi_0}.
\end{align*}
In order to obtain pointwise estimates on $I_c$ and $II_c$, we note that there exists $\eta_c > 0$ such that $\Re(\lambda_c(\pm \xi)) \leq -\eta_c$ for all $\xi \in [\tfrac12 \xi_0,\xi_0]$ by Proposition~\ref{prop:speccons}. Therefore, recalling that $1-\rho(\xi)$ vanishes on $[-\tfrac12 \xi_0,\tfrac12 \xi_0]$, we obtain
\begin{align*}
 |I_c|, |II_c| &\lesssim (1+t+t^2)\re^{-\eta_c t},
\end{align*}
for $t \geq 0$ and $\xx, \xt \in \R$. We conclude
\begin{align}
 \label{e:Icest}
\begin{split}
&\left|I_c - \frac{\xx - \xt}{1+(\xx-\xt)^2}\left(\re^{\lambda_0(\xi_0) t + \ri\xi_0(\xx-\xt)}  \Psi_1(\xx,\xt,\lambda_0(\xi_0)) - \re^{\lambda_0(-\xi_0) t - \ri\xi_0(\xx-\xt)}  \Psi_1(\xx,\xt,\lambda_0(-\xi_0))\right)\right|\\ 
&\qquad \lesssim \frac{(1+t+t^2)\re^{-\eta_c t}}{1 + \left(\xx - \xt\right)^2},
\end{split}
\end{align}
for $t \geq 0$ and $\xx, \xt \in \R$. 

\paragraph*{Conclusion.} Denote $\widetilde{\eta} := \min\{\eta_0/2,\eta_c/2,\eta_1\} > 0$. Recalling the decomposition~\eqref{e:Cauchy_intermediate} and applying the estimates~\eqref{e:I+est},~\eqref{e:I-est},~\eqref{e:I0est} and~\eqref{e:Icest}, we find the desired bound
\begin{align*}
&\left|\ri \int_\R \int_\R \rho(\xi)\re^{\lambda_c(\xi) t + \ri \xi(\xx - \xt)} \Phi_\xi(\zeta)\widetilde{\Phi}_\xi(\xt)^* \,\de\xi \,\gb(\xt) \,\de\xt - \int_\xx^\infty \int_{\Gamma_\delta} \re^{\lambda t + \nu(\lambda)(\xx-\xt)}\Psi(\zeta,\lambda)\tilde{\Psi}(\xt,\lambda)^* \,\de\lambda \,\gb(\xt)\,\de \xt\right|\\ 
&\qquad \lesssim \|\gb\|_{L^\infty} \left(\int_\R \frac{(1+t+t^2)\re^{-2\widetilde{\eta} t}}{1 + \xt^2} \,\de \xt + \int_{-\infty}^0 \frac{|\xt|\re^{\eta_0 \xt-\eta_1 t}}{1 + \xt^2} \,\de \xt \right)\lesssim \|\gb\|_{L^\infty} \re^{-\widetilde{\eta} t},
\end{align*}
for $\gb \in C_{\mathrm{ub}}(\R)$, $\xx \in \R$ and $t \geq 0$.
\end{proof}

\begin{figure}[ht]
\centering
\begin{subfigure}[b]{.34\textwidth}
\centering
\includegraphics[width=\linewidth]{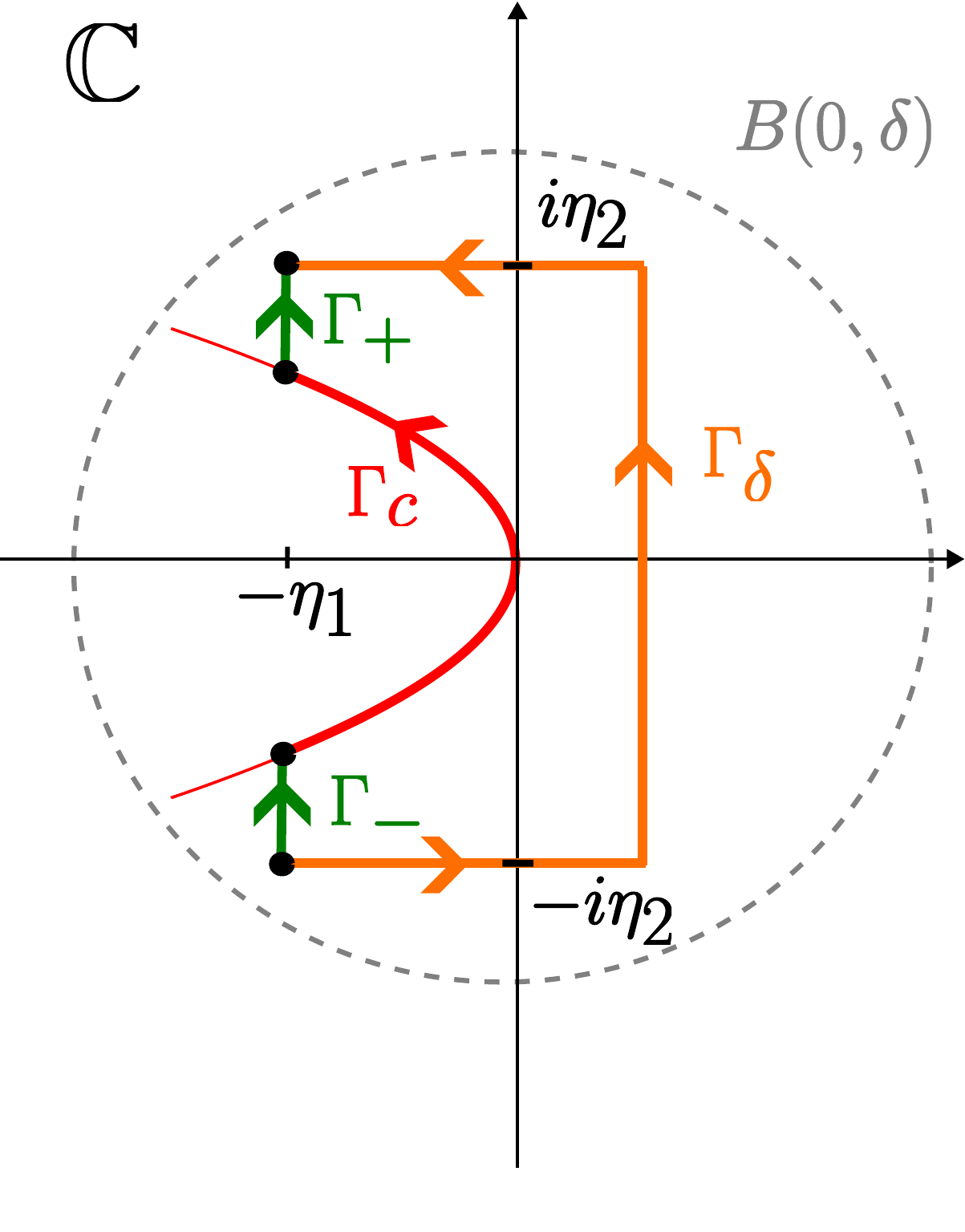}
\end{subfigure} \hspace{0.2\textwidth}
\begin{subfigure}[b]{.34\textwidth}
\centering
\includegraphics[width=\linewidth]{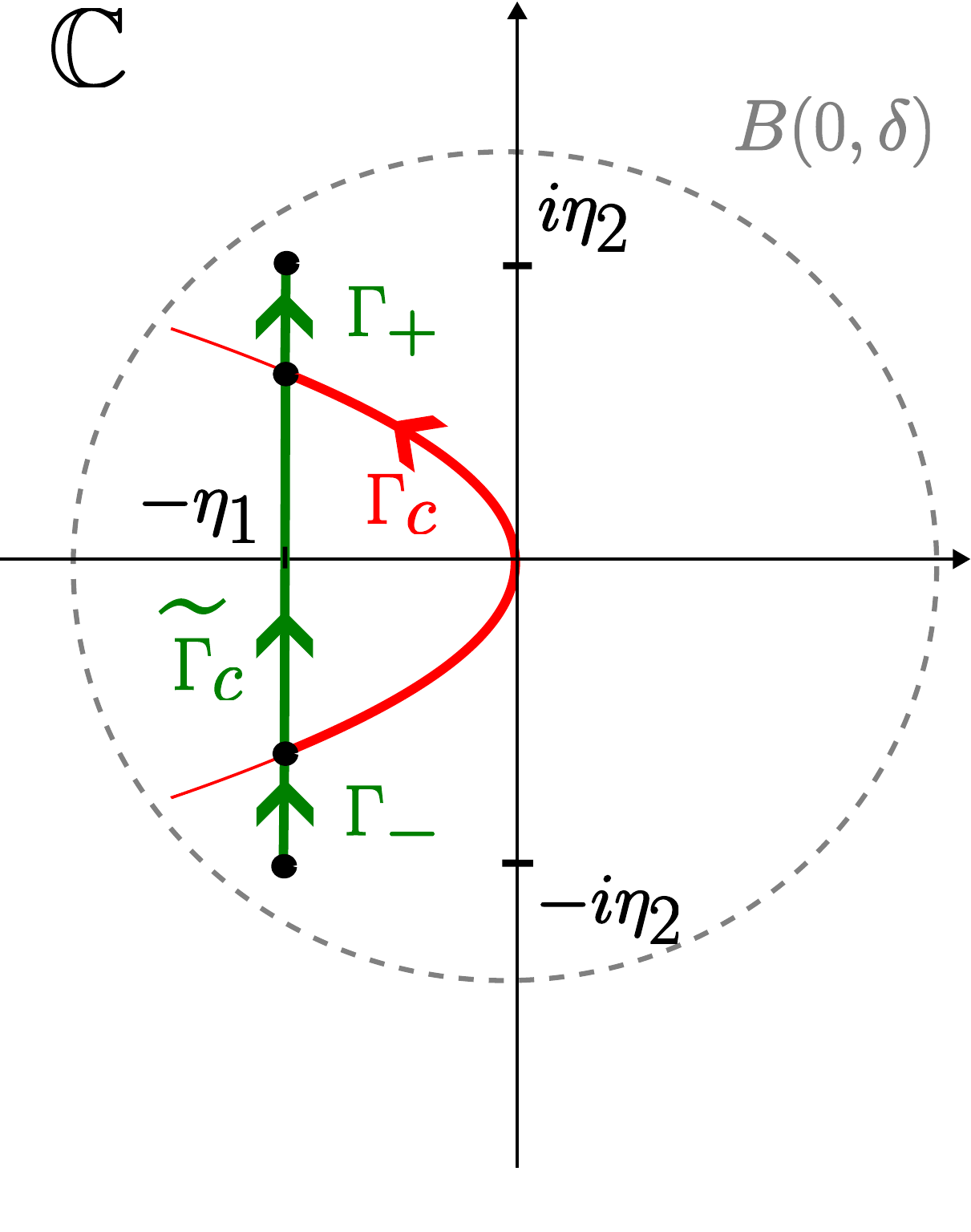}
\end{subfigure}
\caption{In the proof of Proposition~\ref{p:FloquetBlochRep} we relate the Floquet-Bloch representation of the critical part of the semigroup, corresponding to an inverse Laplace integral over $\Gamma_c$, with the aid of Cauchy's integral theorem to complex line integrals over $\Gamma_\delta,\Gamma_-$ and $\Gamma_+$ for $\xx \leq \xt$ (left panel) and over $\smash{\widetilde{\Gamma}_c}$ for $\xx \geq \xt$ (right panel). Here, $\Gamma_c$ lies on the critical spectral curve $\{\lambda_c(\xi) : \xi \in \R \cap V_1\}$ established in Proposition~\ref{prop:speccons}.}
\label{ident_figure}
\end{figure}

\subsection{Linear estimates}

By Propositions~\ref{deform},~\ref{p:highfreqest},~\ref{p:isolate1},~\ref{p:isolate2} and~\ref{p:FloquetBlochRep}, the semigroup $\re^{\El_0 t}$ decomposes for $t \geq 0$ as
\begin{align*}
 \re^{\El_0 t} = S_c(t) + S_e(t),
\end{align*}
where the operator $S_c(t) \colon C_{\mathrm{ub}}(\R) \to C_{\mathrm{ub}}(\R)$ given by
\begin{align} \label{e:Scdef}
\left(S_c(t)\gb\right)[\zeta] = \frac{\chi(t)}{2 \pi} \int_\R \int_\R \rho(\xi)\re^{\lambda_c(\xi) t + \ri \xi(\xx - \xt)} \Phi_\xi(\xx)\widetilde{\Phi}_\xi(\zb)^* \,\de\xi \, \gb(\xt) \,\de\xt
\end{align}
corresponds to the critical low-frequency part of the semigroup and
\begin{align} \label{e:Sedef}
 S_e(t) = (1-\chi(t)) \re^{\El_0 t} + S_e^1(t) + S_e^2(t) + S_e^3(t) + S_e^4(t)
\end{align}
is the exponentially decaying residual. The Floquet-Bloch representation~\eqref{e:Scdef} of the critical part of the semigroup is identical to the one obtained in the stability analysis~\cite{BjoernMod} of wave trains in reaction-diffusion systems against $C_{\mathrm{ub}}$-perturbations. Thus, the further decomposition of $S_c(t)$, as well as the proofs of the associated $L^\infty$-estimates, can be taken verbatim from~\cite{BjoernMod}. On the other hand, estimates on the terms comprising $S_e(t)$ were obtained in Lemma~\ref{l:shorttime} and Propositions~\ref{p:highfreqest},~\ref{p:isolate1},~\ref{p:isolate2} and~\ref{p:FloquetBlochRep}. In the final result of this section, we collect these results and state the decomposition of the semigroup and associated estimates needed for our nonlinear stability analysis.

\begin{theorem} \label{t:linear}
Assume~\ref{assH1},~\ref{assH2} and~\ref{assD1}-\ref{assD3}. Let $j,l \in \mathbb{N}_0$. There exist constants $C,\alpha > 0$ such that the semigroup $\re^{\El_0 t}$ decomposes as
\begin{align} \label{e:decomp1}
 \re^{\El_0 t} = \left(\phi_0' + \partial_k \phi(\cdot,1)\partial_\xx\right) S_p(t) + S_r(t) + S_e(t),
\end{align}
where the operators $S_e(t),S_r(t) \colon C_{\mathrm{ub}}(\R) \to C_{\mathrm{ub}}(\R)$ obey the estimates
\begin{align} \label{e:linest1}
\|S_e(t)\gb\|_{L^\infty} &\leq C \re^{-\alpha t}\|\gb\|_{L^\infty},\qquad
\|S_r(t)\gb\|_{L^\infty} \leq C \frac{\|\gb\|_{L^\infty}}{1+t}
\end{align}
for $t \geq 0$ and $\gb \in C_{\mathrm{ub}}(\R)$. In addition, $S_p(t) \colon C_{\mathrm{ub}}(\R) \to C_{\mathrm{ub}}(\R)$ satisfies $S_p(t) = 0$ for $t \in [0,1]$ and the map $t \mapsto S_p(t)\gb$ lies in $C^i\big([0,\infty),C^k_{\mathrm{ub}}(\R)\big)$ for any $i,k \in \mathbb{N}_0$ with 
\begin{align} \label{e:linest2}
\big\|(\partial_t + c_g \partial_\xx)^j \partial_\xx^l S_p(t)\gb\big\|_{L^\infty} &\leq C \frac{\|\gb\|_{L^\infty}}{(1+t)^{j + \frac{l}{2}}}
\end{align}
for $t \geq 0$ and $\gb \in C_{\mathrm{ub}}(\R)$. We have the further decomposition
\begin{align} \label{e:decomp2}
 \partial_\xx^m S_p(t)\gb = \partial_\xx^m \re^{\left(d\partial_\xx^2 - c_g\partial_\xx\right) t} \left(\widetilde{\Phi}_0^*\gb\right) + \partial_\xx^m \widetilde{S}_r(t)\gb,
\end{align}
where the operator $\partial_\xx^m \widetilde{S}_r(t) \colon  C_{\mathrm{ub}}(\R) \to C_{\mathrm{ub}}(\R)$ obeys the estimate
\begin{align}\label{e:linest3}
\big\|\partial_\xx^m \widetilde{S}_r(t)\gb\big\|_{L^\infty} \leq C (1+t)^{-\frac12} t^{-\frac{m}{2}}\|\gb\|_{L^\infty}
\end{align}
for $m = 0,1$, $t > 0$ and $\gb \in C_{\mathrm{ub}}(\R)$. Finally, there exist a bounded operator $A_h \colon L^2_{\mathrm{per}}\big((0,T),\R^2\big) \to C(\R,\R)$ such that it holds
\begin{align} \label{e:shid}
\re^{\left(d\partial_\xx^2 - c_g\partial_\xx\right) t} \left(v \widetilde{\Phi}_0^* \gb \right) = \re^{\left(d\partial_\xx^2 - c_g\partial_\xx\right) t}\left(\langle \widetilde{\Phi}_0, \gb\rangle_{L^2(0,T)} v - A_h(\gb) \partial_\xx v\right) +  \partial_\xx \re^{\left(d\partial_\xx^2 - c_g\partial_\xx\right) t}\left(A_h(\gb) v\right),
\end{align}
for $\gb \in L^2_{\mathrm{per}}((0,T),\R^2)$, $v \in C_{\mathrm{ub}}^1(\R,\R)$ and $t > 0$.
\end{theorem}
\begin{proof}
The decomposition $\re^{\El_0 t} = S_c(t) + S_e(t)$, where $S_e(t)$ is given by~\eqref{e:Sedef} and $S_c(t)$ is given by~\eqref{e:Scdef}, follows from Propositions~\ref{deform},~\ref{p:highfreqest},~\ref{p:isolate1},~\ref{p:isolate2} and~\ref{p:FloquetBlochRep}. The desired bound~\eqref{e:linest1} on $S_e(t)$ can be derived by combining Lemma~\ref{l:shorttime} and Propositions~\ref{p:highfreqest},~\ref{p:isolate1},~\ref{p:isolate2} and~\ref{p:FloquetBlochRep}. Moreover, it has been shown in~\cite[Section~3.3]{BjoernMod} that $S_c(t)$ decomposes as $S_c(t) = \left(\phi_0' + \partial_k \phi(\cdot,1)\partial_\xx\right) S_p(t) + S_r(t)$, where $S_p(t), S_r(t) \colon C_{\mathrm{ub}}(\R) \to C_{\mathrm{ub}}(\R)$ are operators obeying the estimates~\eqref{e:linest1} and~\eqref{e:linest2}. Moreover, $S_p(t)$ satisfies $S_p(t) = 0$ for $t \in [0,1]$ and the map $t \mapsto S_p(t)\gb$ lies in $C^i\big([0,\infty),C^k_{\mathrm{ub}}(\R)\big)$ for any $i,k \in \mathbb{N}_0$. Finally, the decomposition~\eqref{e:decomp2}, the estimates~\eqref{e:linest3} and the identity~\eqref{e:shid} can be found in~\cite[Section~3.5]{BjoernMod}. 
\end{proof}

\section{Nonlinear iteration scheme and nonlinear estimates} \label{sec:iteration}

In this section, we set up the nonlinear iteration scheme and state associated nonlinear estimates, which will be employed in the upcoming section to prove our nonlinear stability result, Theorem~\ref{t:mainresult}. To this end, we consider a diffusively spectrally stable wave-train solution $\ub_0(x,t) = \phi_0(x-c_0t)$ to~\eqref{FHNsys}, i.e., we assume that Hypotheses~\ref{assH1},~\ref{assH2} and~\ref{assD1}-\ref{assD3} are satisfied, and an initial perturbation $\vb_0 \in C_{\mathrm{ub}}^3(\R) \times C_{\mathrm{ub}}^2(\R)$. We wish to control the perturbation $\vt(t) = \ub(t) - \phi_0$ over time, where $\ub(t)$ is the solution to~\eqref{FHN_co} with initial condition $\ub(0) = \phi_0 + \vb_0$. The perturbation $\vt(t)$ satisfies equation~\eqref{e:umodpert}. Theorem~\ref{t:linear} shows that the bounds on full semigroup $\re^{\El_0 t}$ are too weak to close a nonlinear iteration argument using the Duhamel formulation of~\eqref{e:umodpert}. 

As explained in~\S\ref{sec:strategy}, this leads us to consider the inverse-modulated perturbation $\vb(t)$ given by~\eqref{e:defv}. We derive a quasilinear equation for $\vb(t)$, establish $L^\infty$-bounds on the nonlinearity and define a suitable phase modulation $\psi(t)$ compensating for the most critical terms in the Duhamel formulation of $\vb(t)$. We then infer, as in~\cite{BjoernMod}, that $\psi(t)$ satisfies a perturbed viscous Hamilton-Jacobi equation, whose most critical nonlinear term cannot be controled through $L^\infty$-estimates, but can be eliminated with the aid of the Cole-Hopf transform. We formulate an equation for the Cole-Hopf variable and state $L^\infty$-bounds on the nonlinearity. 

Lastly, we control regularity in the quasilinear iteration scheme by relying on forward-modulated damping estimates. We obtain an equation for the modified forward-modulated perturbation $\mathring{\zt}(t)$ given by~\eqref{e:defvforw}, establish norm equivalences between $\mathring{\zt}(t)$ and the residual
\begin{align} \label{e:defz}
 \zt(t) = \vb(t) - \partial_k \phi(\cdot; 1) \psi_\xx(t),
\end{align}
and we derive a nonlinear damping estimate for $\mathring{\zt}(t)$ using uniformly local Sobolev norms.

\subsection{The unmodulated perturbation}

The unmodulated perturbation $\vt(t)$ satisfies the semilinear equation~\eqref{e:umodpert}, whose nonlinearity $\widetilde{\mathcal{N}} \colon C_{\mathrm{ub}}^1(\R) \to C_{\mathrm{ub}}^1(\R)$ is readily seen to be continuously Fr\'echet differentiable. On the other hand, regarding $\El_0$ as an operator on $C_{\mathrm{ub}}^1(\R)$ with dense domain $C_{\mathrm{ub}}^3(\R) \times C_{\mathrm{ub}}^2(\R)$, Proposition~\ref{p:semigroupgen} yields that $\El_0$ generates a $C_0$-semigroup on $C_{\mathrm{ub}}^1(\R)$. Hence, local existence and uniqueness of a classical solution to~\eqref{e:umodpert} follows by standard results, e.g.~\cite[Theorem~6.1.5]{pazy}, from semigroup theory. 

\begin{proposition} \label{p:local_unmod} Assume~\ref{assH1}. Let $\vb_0 \in C_{\mathrm{ub}}^3(\R) \times C_{\mathrm{ub}}^2(\R)$. Then, there exists a maximal time $T_{\max} \in (0,\infty]$ such that~\eqref{e:umodpert} admits a unique classical solution
\begin{align*}\vt \in C\big([0,T_{\max}),C_{\mathrm{ub}}^3(\R) \times C_{\mathrm{ub}}^2(\R)\big) \cap C^1\big([0,T_{\max}),C_{\mathrm{ub}}^1(\R)\big), \end{align*}
with initial condition $\vt(0) = \vb_0$. Moreover, if $T_{\max} < \infty$, then we have
\begin{align*} \limsup_{t \uparrow T_{\max}} \left\|\vt(t)\right\|_{C_{\mathrm{ub}}^1} = \infty. \end{align*}
\end{proposition}

\subsection{The inverse-modulated perturbation}

Using that $\ub(t)$ and $\phi_0$ solve~\eqref{FHN_co}, one finds that the inverse-modulated perturbation $\vb(t)$, given by~\eqref{e:defv}, obeys the quasilinear equation
\begin{align}
\left(\partial_t - \El_0\right)\left[\vb + \phi_0' \psi\right] = \Non(\vb,\psi,\partial_t \psi) + \left(\partial_t - \El_0\right)\left[\psi_\xx \vb\right] \label{e:modpertbeq}
\end{align}
with nonlinearity
\begin{align*}
\Non(\vb,\psi,\psi_t) &= \mathcal Q(\vb,\psi) + \partial_\xx \mathcal R(\vb,\psi,\psi_t),
\end{align*}
where
\begin{align*}
\mathcal Q(\vb,\psi) &= \left(F(\phi_0+\vb) - F(\phi_0) - F'(\phi_0) \vb\right)\left(1-\psi_\xx\right)
\end{align*}
is quadratic in $\vb$ and
\begin{align*}
\mathcal R(\vb,\psi,\psi_t) &= (c_0 \psi_\xx - \psi_t)\vb + D\left(\frac{(\vb_\xx + \phi_0' \psi_\xx)\psi_\xx}{1-\psi_\xx} + (\vb \psi_\xx)_\xx\right).
\end{align*}
contains all linear terms in $\vb$. We refer to~\cite[Appendix~E]{BjoernAvery} for a detailed derivation of~\eqref{e:modpertbeq}. 

It is relatively straightforward to verify the relevant nonlinear bound.

\begin{lemma} \label{lem:nlboundsmod}
Assume~\ref{assH1}. Then, we have 
\begin{align*}
\|\mathcal N(\vb,\psi,\psi_t)\|_{L^\infty} &\lesssim \|\vb\|_{L^\infty}^2 + \|(\psi_\xx,\psi_t)\|_{C_{\mathrm{ub}}^2 \times C_{\mathrm{ub}}^1} \left(\|\vb\|_{C_{\mathrm{ub}}^2 \times C_{\mathrm{ub}}^1} + \|\psi_\xx\|_{L^\infty}\right)
\end{align*}
for $\vb = (u,v) \in C_{\mathrm{ub}}^2(\R) \times C_{\mathrm{ub}}^1(\R)$ and $(\psi,\psi_t) \in C_{\mathrm{ub}}^3(\R) \times C_{\mathrm{ub}}^1(\R)$ satisfying $\|u\|_{L^\infty}, \|\psi_\xx\|_{L^\infty} \leq \frac{1}{2}$. 
\end{lemma}

Inspired by earlier works~\cite{JONZ,JONZNL}, we implicitly define the phase modulation by the integral equation
\begin{align}
\psi(t) = S_p(t)\vb_0 + \int_0^t S_p(t-s) \mathcal{N}(\vb(s),\psi(s),\partial_t \psi(s)) \,\de s. \label{e:intpsi}
\end{align}
Recalling from Theorem~\ref{t:linear} that $S_p(0) = 0$, we find that $\psi(t)$ vanishes at $t = 0$. Thus, integrating~\eqref{e:modpertbeq} yields the Duhamel formulation
\begin{align}
\vb(t)+\phi_0'\psi(t) = \re^{\El_0 t} \vb_0 + \int_0^t \re^{\El_0(t-s)}\mathcal{N}(\vb(s),\psi(s),\partial_t \psi(s)) \,\de s + \psi_\xx(t)\vb(t). \label{e:intv}
\end{align}
Writing the left-hand side of~\eqref{e:intv} as $\vb(t) + \phi_0'\psi(t) = \zt(t) + \left(\phi_0' + \partial_k \phi(\cdot;1) \partial_\xx \right) \psi(t)$, where $\zt(t)$ is given by~\eqref{e:defz}, and recalling the semigroup decomposition~\eqref{e:decomp1}, we observe that by defining the phase modulation by~\eqref{e:intpsi}, the term $\left(\phi_0' + \partial_k \phi(\cdot;1) \partial_\xx \right) \psi(t)$ compensates for the critical, slowest decaying, contributions on the right-hand side of~\eqref{e:intv}. Indeed, we arrive at the Duhamel formulation
\begin{align}
\begin{split}
\zt(t) &= \left(S_r(t) + S_e(t)\right)\vb_0+\int_0^t\left(S_r(t-s) + S_e(t-s)\right)\mathcal{N}(\vb(s),\psi(s),\partial_t \psi(s)) \,\de s + \psi_\xx(t)\vb(t),
\end{split}\label{e:intz}
\end{align}
for the residual $\zt(t)$, where $S_r(t) + S_e(t)$ exhibits stronger decay than $\re^{\El_0 t}$, cf.~Theorem~\ref{t:linear}.

Local existence of the phase modulation $\psi(t)$ can be obtained by applying a standard contraction mapping argument to the integral equation~\eqref{e:intpsi}, where one employs Proposition~\ref{p:local_unmod} and expresses the inverse-modulated perturbation as
\begin{align} \label{e:defv2}
\vb(\xx,t) = \vt(\xx-\psi(\xx,t),t) + \phi_0(\xx-\psi(\xx,t)) - \phi_0(\xx),
\end{align}
to obtain a fixed point problem in $\psi(t)$ and its derivatives. This leads to the following result, whose proof is identical to~\cite[Proposition~4.4]{BjoernMod}. 

\begin{proposition} \label{p:psi}
Assume~\ref{assH1}. Let $\vb_0 \in C_{\mathrm{ub}}^3(\R) \times C_{\mathrm{ub}}^2(\R)$. Fix $j,l, m \in \mathbb{N}_0$. For $\vt$ and $T_{\max}$ as in Proposition~\ref{p:local_unmod}, there exists a maximal time $\tau_{\max} \in (0,T_{\max}]$ such that equation~\eqref{e:intpsi}, with $\vb$ given by~\eqref{e:defv2}, possesses a solution
\begin{align*}\psi \in C\big([0,\tau_{\max}),C_{\mathrm{ub}}^{2 + m}(\R)\big) \cap C^{1+j}\big([0,\tau_{\max}),C_{\mathrm{ub}}^{l}(\R)\big), \end{align*}
satisfying $\psi(t)=0$ for all $t \in [0,\tau_{\max})$ with $t \leq 1$. Moreover, we have $\|(\psi(t),\partial_t \psi(t))\|_{C_{\mathrm{ub}}^2 \times C_{\mathrm{ub}}} < \frac12$ for all $t \in [0,\tau_{\max})$. Finally, if $\tau_{\max} < T_{\max}$, then
\begin{align*} \limsup_{t \uparrow \tau_{\max}} \left\|\left(\psi(t),\partial_t \psi(t)\right)\right\|_{C_{\mathrm{ub}}^2 \times C_{\mathrm{ub}}} = \frac12.\end{align*}
\end{proposition}

The existence and regularity of the inverse-modulated perturbation $\vb(t)$ and the residual $\zt(t)$ now follow immediately from~\eqref{e:defv2} and~\eqref{e:defz}, respectively, upon applying Propositions~\ref{prop:family},~\ref{p:local_unmod} and~\ref{p:psi} and using the uniform continuity of functions in $C_{\mathrm{ub}}(\R)$.

\begin{corollary} \label{C:local_v}
Assume~\ref{assH1} and~\ref{assD3}. Let $\vb_0 \in C_{\mathrm{ub}}^3(\R) \times C_{\mathrm{ub}}^2(\R)$. For $\vt$ as in Proposition~\ref{p:local_unmod} and $\psi$ and $\tau_{\max}$ as in Proposition~\ref{p:psi}, the inverse-modulated perturbation $\vb$, defined by~\eqref{e:defv}, and the residual $\zt$, defined by~\eqref{e:defz}, obey
$$\vb,\zt \in C\big([0,\tau_{\max}),C_{\mathrm{ub}}^3(\R) \times C_{\mathrm{ub}}^2(\R) \big).$$
Moreover, their Duhamel formulations~\eqref{e:intv} and~\eqref{e:intz} hold for $t \in [0,\tau_{\max})$.
\end{corollary}

\subsection{Derivation of the perturbed viscous Hamilton-Jacobi equation} \label{sec:derihamjac}

The estimates in Theorem~\ref{t:linear}, in combination with~\eqref{e:intpsi} and~\eqref{e:intz}, show that, at least on the linear level, the derivative $\smash{\partial_\xx^j \partial_t^l \psi(t)}$ of the phase modulation decays at rate $\smash{t^{-(j+l)/2}}$ for $j,l \in \mathbb{N}_0$, whereas the residual $\zt(t)$ and 
\begin{align*} 
\tilde{\psi}(t) = \partial_t \psi(t) + c_g \psi_\xx(t),
\end{align*}
decay at rate $t^{-1}$. Therefore, after substituting 
\begin{align} \label{e:reexpress}
\vb(t) = \zt(t) + \partial_k \phi(\cdot;1) \psi_\xx(t), \qquad \partial_t \psi(t) = \tilde{\psi}(t) - c_g \psi_\xx(t),
\end{align}
in the nonlinearity $\Non(\vb,\psi,\psi_t)$ one finds that the nonlinear terms exhibiting the slowest decay are of Burgers'-type, i.e.~of the form $\mathbf{f} \psi_\xx(t)^2$ with coefficient $\mathbf{f} \in L^2_{\mathrm{per}}(0,T)$. 

The decay rates of the principal part $S_p(t)$ of the semigroup $\re^{\El_0 t}$ are not strong enough to control these most critical nonlinear terms through iterative estimates on the equation~\eqref{e:intpsi} for the phase modulation. As outlined in~\S\ref{sec:strategy}, we address this issue by proceeding as in~\cite{BjoernMod}. That is, we show that $\psi(t)$ obeys a perturbed viscous Hamilton-Jacobi equation and subsequently apply the Cole-Hopf transform to this equation to eliminate the critical $\psi_\xx^2$-contributions. 

To derive a viscous Hamilton-Jacobi equation for $\psi(t)$, we first isolate the $\psi_\xx^2$-contributions in the nonlinearity $\Non(\vb,\psi,\psi_t)$ of~\eqref{e:intpsi}. We do so by reexpressing $\vb(t)$ and $\partial_t \psi(t)$ through~\eqref{e:reexpress} wherever necessary. Thus, recalling $c_0 + c_g = \omega'(1)$ from Proposition~\ref{prop:speccons}, we arrive at
\begin{align} \label{e:nonldecomp}
\Non(\vb(s),\psi(s),\partial_t \psi(s)) =  \mathbf{f}_p \psi_{\xx}(s)^2 + \mathcal{N}_p(\zt(s),\vb(s),\psi(s),\tilde{\psi}(s)),
\end{align}
with $T$-periodic coefficient
\begin{align*}
\mathbf{f}_p &= \frac{1}{2} F''(\phi_0)\left(\partial_k \phi(\cdot;1),\partial_k \phi(\cdot;1)\right) + \omega'(1) \partial_{\xx k} \phi(\cdot;1) + D\left(\phi_0'' + 2 \partial_{\xx\xx k} \phi(\cdot;1)\right)
\end{align*}
and residual nonlinearity
\begin{align*}
\mathcal N_p(\zt,\vb,\psi,\tilde{\psi}) &= \mathcal Q_p(\zt,\vb,\psi) + \partial_\xx \mathcal R_p(\zt,\vb,\psi,\tilde{\psi}),
\end{align*}
where we denote
\begin{align*}
\mathcal Q_p(\zt,\vb,\psi) &= \left(F(\phi_0+\vb) - F(\phi_0) - F'(\phi_0) \vb\right)\psi_\xx + F(\phi_0+\vb) - F(\phi_0) - F'(\phi_0) \vb \\
&\qquad - \, \frac{1}{2} F''(\phi_0)(\vb,\vb) + \frac{1}{2}F''(\phi_0)(\zt,\vb) + \frac12 \psi_{\xx} F''(\phi_0)(\zt,\partial_k \phi(\cdot;1)) \\
&\qquad + \, 2 \psi_\xx \psi_{\xx\xx} \left(\omega'(1)\partial_k \phi(\cdot;1)  +  D \left(\phi_0' + 2\partial_{\xx k} \phi(\cdot;1)\right)\right),\\
\mathcal R_p(\zt,\vb,\psi,\tilde{\psi}) &= -\tilde{\psi}\vb + \omega'(1) \psi_{\xx}\zt + D\left(\frac{\left(\vb_\xx + \phi_0'\psi_\xx\right)\psi_\xx^2}{1-\psi_\xx} + 2 \zt_\xx \psi_\xx + \vb \psi_{\xx\xx} + 2 \partial_k \phi(\cdot;1) \psi_\xx \psi_{\xx\xx}\right).
\end{align*}

We establish an $L^\infty$-estimate on the residual nonlinearity. 

\begin{lemma} \label{lem:nlboundsmod3}
Assume~\ref{assH1} and~\ref{assD3}. Then, we have
\begin{align*}
\begin{split}
\|\mathcal N_p(\zt,\vb,\psi,\tilde{\psi})\|_{L^\infty} &\lesssim \left(\|\vb\|_{L^\infty} + \|\psi_\xx\|_{C_{\mathrm{ub}}^1}\right)\left(\|\vb\|_{L^\infty}^2 + \|\zt\|_{C_{\mathrm{ub}}^2}\right) + \left(\|\tilde{\psi}\|_{C_{\mathrm{ub}}^1} + \|\psi_{\xx\xx}\|_{C_{\mathrm{ub}}^1}\right) \|\vb\|_{C_{\mathrm{ub}}^1} \\ 
&\qquad + \, \left(\|\psi_{\xx\xx}\|_{C_{\mathrm{ub}}^1} + \|\vb\|_{C_{\mathrm{ub}}^2} \|\psi_\xx\|_{L^\infty} + \|\psi_\xx\|_{L^\infty}^2 \right)\|\psi_\xx\|_{C_{\mathrm{ub}}^1}
\end{split}
\end{align*}
for $\zt,\vb \in C_{\mathrm{ub}}^2(\R)$ and $(\psi,\tilde{\psi}) \in C_{\mathrm{ub}}^3(\R) \times C_{\mathrm{ub}}^1(\R)$ satisfying $\|\vb\|_{L^\infty}, \|\psi_\xx\|_{L^\infty} \leq \frac{1}{2}$.
\end{lemma}

Next, we substitute the decompositions~\eqref{e:decomp2} of the propagtor $S_p(t)$ and~\eqref{e:nonldecomp} of the nonlinearity $\mathcal{N}(\vb(s),\psi(s),\partial_t \psi(s))$ into~\eqref{e:intpsi} and use~\eqref{e:shid} to reexpress $\smash{\re^{\left(d\partial_\xx^2 - c_g\partial_\xx\right) t} \left(\widetilde{\Phi}_0^*\mathbf{f}_p \psi_\xx^2\right)}$. All in all, we arrive at
\begin{align}
\begin{split}
\psi(t) &= r(t) + \re^{\left(d\partial_\xx^2 - c_g\partial_\xx\right) t} \left(\widetilde{\Phi}_0^*\vb_0\right) + \int_0^t \re^{\left(d\partial_\xx^2 - c_g\partial_\xx\right) (t-s)} \left(\nu \psi_\xx(s)^2 -  A_h(\mathbf{f}_p) \partial_\xx \left(\psi_\xx(s)^2\right)\right) \,\de s\\
&\qquad + \, \int_0^t \re^{\left(d\partial_\xx^2 - c_g\partial_\xx\right) t} \left(\widetilde{\Phi}_0^* \mathcal N_p(\zt(s),\vb(s),\psi(s),\tilde{\psi}(s))\right)\de s,
\end{split}
\label{e:intpsi2}
\end{align}
where we denote
\begin{align*}
\nu = \langle \widetilde{\Phi}_0, \mathbf{f}_p\rangle_{L^2(0,T)}
\end{align*}
and
\begin{align}
\begin{split}
r(t) &= \widetilde{S}_r(t)\vb_0 +  \int_0^t \widetilde{S}_r(t-s)\left(\mathbf{f}_p \psi_\xx(s)^2\right) \,\de s + \partial_\xx \int_0^t \re^{\left(d\partial_\xx^2 - c_g\partial_\xx\right) (t-s)}\left(A_h(\mathbf{f}_p) \psi_\xx(s)^2\right) \,\de s\\
&\qquad + \, \int_0^t \widetilde{S}_r(t-s) \mathcal N_p(\zt(s),\vb(s),\psi(s),\tilde{\psi}(s)) \,\de s.
\end{split}\label{e:intr}
\end{align}
Since $\psi(0)$ vanishes identically by Proposition~\ref{p:psi}, setting $t = 0$ in~\eqref{e:intpsi2} yields
\begin{align} \label{e:initr}
r(0) = - \widetilde{\Phi}_0^* \vb_0.
\end{align}
Moreover, following the computations in~\cite[Section~4.2]{DSSS}, one finds that the coefficient $\nu$ in~\eqref{e:intpsi2} equals $\smash{-\frac12 \omega''(1)}$. Thus, with the aid of Proposition~\ref{prop:speccons}, we arrive at the expression~\eqref{e:nuexpress} for $\nu$. Since $\smash{\widetilde{S}_r(t)}$ and $\smash{\partial_\zeta \re^{\left(d\partial_\xx^2 - c_g\partial_\xx\right) t}}$ decay at rate $\smash{t^{-\frac12}}$ as operators on $C_{\mathrm{ub}}(\R)$, we find that $r(t)$ captures, at least on the linear level, the decaying contributions in~\eqref{e:intpsi2}, cf.~Theorem~\ref{t:linear}. 

Finally, applying the convective heat operator $\partial_t - d\partial_\xx^2 + c_g\partial_\xx$ to~\eqref{e:intpsi2}, we arrive at the perturbed viscous Hamilton-Jacobi equation
\begin{align} \label{e:hamjacper}
\left(\partial_t - d\partial_\xx^2 + c_g\partial_\xx\right)\left(\psi - r\right) = \nu \psi_\xx^2 + G(\zt,\vb,\psi,\tilde{\psi})
\end{align}
with nonlinear residual
$$G(\zt,\vb,\psi,\tilde{\psi}) = \widetilde{\Phi}_0^* \mathcal N_p(\zt,\vb,\psi,\tilde{\psi}) - A_h(\mathbf{f}_p) \partial_\xx \left(\psi_\xx^2\right).$$
Indeed, modulo the higher-order terms $r$ and $G(\vb,\zt,\psi,\tilde{\psi})$ equation~\eqref{e:hamjacper} coincides with the Hamilton-Jacobi equation~\eqref{e:HamJac}. Regarding~\eqref{e:hamjacper} as an inhomogeneous parabolic equation, regularity properties of $\psi(t) - r(t)$, and thus of $r(t)$, can be readily deduced from standard analytic semigroup theory.

\begin{corollary} \label{C:local_r}
Assume~\ref{assH1} and~\ref{assD3}. Let $\vb_0 \in C_{\mathrm{ub}}^3(\R) \times C_{\mathrm{ub}}^2(\R)$. For $\psi$ and $\tau_{\max}$ as in Proposition~\ref{p:psi} and for $\vb$ and $\zt$ as in Corollary~\ref{C:local_v}, the residual $r$, given by~\eqref{e:intr}, obeys
$$r \in C\big([0,\tau_{\max}),C_{\mathrm{ub}}^2(\R)\big) \cap C^1\big([0,\tau_{\max}),C_{\mathrm{ub}}(\R)\big).$$
\end{corollary}
\begin{proof}
Moving $r(t)$ to the left-hand side, we can regard~\eqref{e:intpsi2} as the mild formulation of the inhomogeneous problem~\eqref{e:hamjacper} for $\psi(t) - r(t)$ with inhomogeneity $t \mapsto \nu\psi_\xx(t)^2 + G(\zt(t),\vb(t),\psi(t),\tilde{\psi}(t))$, which lies in $C\big([0,\tau_{\max}),C_{\mathrm{ub}}^1(\R)\big)$ by Proposition~\ref{p:psi} and Corollary~\ref{C:local_v}. It is well-known that $d\partial_\xx^2 - c_g\partial_\xx$ is a sectorial operator on $C_{\mathrm{ub}}(\R)$ with domain $C_{\mathrm{ub}}^2(\R)$, cf.~\cite[Corollary~3.1.9]{LUN}. Therefore, since the initial condition $\psi(0) - r(0) = \smash{\widetilde{\Phi}_0^* \vb_0}$ lies in the domain $C_{\mathrm{ub}}^2(\R)$ and $C_{\mathrm{ub}}^1(\R)$ is an intermediate space between $C_{\mathrm{ub}}(\R)$ and the domain $C_{\mathrm{ub}}^2(\R)$, it follows from~\cite[Propositions~2.1.1 and~2.1.4 and Theorem~4.3.8]{LUN} that $t \mapsto \psi(t) - r(t)$ lies in $C\big([0,\tau_{\max}),C_{\mathrm{ub}}^2(\R)\big) \cap C^1\big([0,\tau_{\max}),C_{\mathrm{ub}}(\R)\big)$. Invoking Proposition~\ref{p:psi} then yields the result.
\end{proof}

\subsection{Application of the Cole-Hopf transform}

We apply the Cole-Hopf transform to remove the critical nonlinear term $\nu \psi_\xx^2$ in~\eqref{e:hamjacper}. That is, we introduce the new variable
\begin{align} \label{e:defy} y(t) = \re^{\frac{\nu}{d} \left(\psi(t) - r(t)\right)} - 1,\end{align}
which satisfies
\begin{align}\label{e:regy} y \in C\big([0,\tau_{\max}),C_{\mathrm{ub}}^2(\R)\big) \cap C^1\big([0,\tau_{\max}),C_{\mathrm{ub}}(\R)\big)\end{align}
by Proposition~\ref{p:psi} and Corollary~\ref{C:local_r}. It is readily seen that $y(t)$ is a solution of the convective heat equation
\begin{align} \left(\partial_t - d\partial_\xx^2 + c_g\partial_\xx\right)y = 2\nu r_\xx y_\xx + \frac{\nu}{d}\left(\nu r_\xx^2 + G(\zt,\vb,\psi,\tilde{\psi})\right)\left(y+1\right) \label{e:colehopf} \end{align}
with initial condition
\begin{align} \label{e:yinit}
y(0) = \re^{\frac{\nu}{d}\widetilde{\Phi}_0^* (\psi(0) - r(0))} - 1 = \re^{\frac{\nu}{d}\widetilde{\Phi}_0^* \vb_0} - 1,
\end{align}
cf.~Proposition~\ref{p:psi} and~\eqref{e:initr}. 

Recalling that $\psi(t)$ vanishes identically for $t \in [0,1]$ by Proposition~\ref{p:psi}, the Cole-Hopf variable $y(t)$ can be expressed in terms of the residual $r(t)$ through
\begin{align} \label{e:ysmall}
y(t) = \re^{-\frac{\nu}{d} r(t)} - 1
\end{align}
for $t \in [0,\tau_{\max})$ with $t \leq 1$. On the other hand, the Duhamel formulation of~\eqref{e:colehopf} reads
\begin{align} 
\begin{split}
y(t) &= \re^{\left(d\partial_\xx^2 - c_g\partial_\xx\right) (t-1)}\left(\re^{\frac{\nu}{d}\widetilde{\Phi}_0^* \vb_0} - 1\right)\\
&\qquad + \, \int_1^t \re^{\left(d\partial_\xx^2 - c_g\partial_\xx\right) (t-s)} \Non_c(r(s),y(s),\zt(s),\vb(s),\psi(s),\tilde{\psi}(s)) \,\de s
\end{split}
\label{e:inty} \end{align}
for $t \in [0,\tau_{\max})$ with $t \geq 1$, where the nonlinearity is given by
\begin{align*}
\Non_c(r,y,\zt,\vb,\psi,\tilde{\psi}) = 2\nu r_\xx y_\xx + \frac{\nu}{d}\left(\nu r_\xx^2 + G(\zt,\vb,\psi,\tilde{\psi})\right)\left(y+1\right).
\end{align*}
We use~\eqref{e:ysmall} for short-time control on $y(t)$ (rather than its Duhamel formulation) in the upcoming nonlinear argument. The reason is that we use a temporal weight $\sqrt{s}\sqrt{1+s}$ on $r_\xx(s)$, so that the obtained bound on $r_\xx(s)^2$ is nonintegrable and blows up as $1/s$ as $s \downarrow 0$. We refer to the proof of Theorem~\ref{t:mainresult} and Remark~\ref{rem:motv2} for further details.

With the aid of Lemma~\ref{lem:nlboundsmod3}, we obtain the following nonlinear estimate.

\begin{lemma}\label{lem:nlboundsmod4}
Assume~\ref{assH1} and~\ref{assD3}. It holds
\begin{align*}
\begin{split}
\|\Non_c(r,y,\vb,\zt,\psi,\tilde{\psi})\|_{L^\infty} &\lesssim \left(\|r_\xx\|_{L^\infty} + \|y_\xx\|_{L^\infty}\right)\|r_\xx\|_{L^\infty} + \left(\|\vb\|_{L^\infty} + \|\psi_\xx\|_{C_{\mathrm{ub}}^1}\right)\left(\|\vb\|_{L^\infty}^2 + \|\zt\|_{C_{\mathrm{ub}}^2}\right)  \\ 
&\qquad + \,  \left(\|\tilde{\psi}\|_{C_{\mathrm{ub}}^1} + \|\psi_{\xx\xx}\|_{C_{\mathrm{ub}}^1}\right) \|\vb\|_{C_{\mathrm{ub}}^1}\\
&\qquad + \, \left(\|\psi_{\xx\xx}\|_{C_{\mathrm{ub}}^1} + \|\vb\|_{C_{\mathrm{ub}}^2} \|\psi_\xx\|_{L^\infty} + \|\psi_\xx\|_{L^\infty}^2 \right)\|\psi_\xx\|_{C_{\mathrm{ub}}^1}
\end{split}
\end{align*}
for each $r,y \in C_{\mathrm{ub}}^1(\R)$, $\zt,\vb \in C_{\mathrm{ub}}^2(\R)$ and $(\psi,\tilde{\psi}) \in C_{\mathrm{ub}}^3(\R) \times C_{\mathrm{ub}}^1(\R)$ with $\|y\|_{L^\infty}, \|\vb\|_{L^\infty}, \|\psi_\xx\|_{L^\infty} \leq \frac{1}{2}$.
\end{lemma}

\subsection{Forward-modulated damping}

The modified forward-modulated perturbation $\mathring{\zt}(t)$ is given by~\eqref{e:defvforw}, where the $T$-periodic continuation $\phi(\cdot;k)$ of the wave train $\phi_0$ with respect to the wavenumber $k$ is defined for $k$ in the neighborhood $[1-r_0,1+r_0]$ by Proposition~\ref{prop:family}. Combining the latter with Propositions~\ref{p:local_unmod} and~\ref{p:psi}, we find that the forward-modulated perturbation $\mathring{\zt}(t)$ is well-defined as long as $t \in [0,\tau_{\max})$ is such that $\|\psi_\xx(t)\|_{L^\infty} < r_0$. Its regularity then follows readily from Propositions~\ref{prop:family},~\ref{p:local_unmod} and~\ref{p:psi}.

\begin{corollary} \label{c:local_forward}
Assume~\ref{assH1} and~\ref{assD3}. Let $\vb_0 \in C_{\mathrm{ub}}^3(\R) \times C_{\mathrm{ub}}^2(\R)$. For $r_0 > 0$ as in Proposition~\ref{prop:family}, $\vt$ as in Proposition~\ref{p:local_unmod}, and $\psi$ and $\tau_{\max}$ as in Proposition~\ref{p:psi}, we have
\begin{align*}
\widetilde{\tau}_{\max} = \sup\left\{t \in [0,\tau_{\max}) : \|\psi_\xx(s)\|_{L^\infty} < r_0 \text{ for all } s \in [0,t]\right\} > 0
\end{align*}
and the modified forward-modulated perturbation $\mathring{\zt}(t)$, given by~\eqref{e:defvforw}, is well-defined for $t \in [0,\widetilde{\tau}_{\max})$ and satisfies
$$\mathring{\zt} \in C\big([0,\widetilde{\tau}_{\max}),C_{\mathrm{ub}}^3(\R) \times C_{\mathrm{ub}}^2(\R)\big) \cap C^1\big([0,\widetilde{\tau}_{\max}),C_{\mathrm{ub}}^1(\R)\big).$$
\end{corollary}

Using that the wave train $\ub_k(x,t) = \phi(kx - \omega(k) t; k)$ is a solution to~\eqref{FHN} and $\ub(t)$ solves~\eqref{FHN_co}, one obtains the equation 
\begin{align} \label{e:Pert1}
\partial_t \mathring{\zt} = D \mathring{\zt}_{\xx\xx} + c_0 \mathring{\zt}_\xx + F'(0) \mathring{\zt} + \mathring{\mathcal{Q}}(\mathring{\zt},\psi) + \mathring{\mathcal{R}}(\psi,\tilde{\psi},\partial_t \psi)
\end{align}
for the modified forward-modulated perturbation, where 
\begin{align*}
\mathring{\mathcal{Q}}(\mathring{\zt},\psi) &= F(\mathring{\zt} + \phi(\beta(\psi))) - F(\phi(\beta(\psi))) - F'(0)\mathring{\zt}\\
&= \begin{pmatrix}
  \left(\phi_1(\beta(\psi)) \left(2 + 2\mu - 3 \mathring{z}_1\right) - 3 \phi_1(\beta(\psi))^2 + (1 + \mu - \mathring{z}_1) \mathring{z}_1\right) \mathring{z}_1 \\ 0 
\end{pmatrix}
\end{align*}
is the nonlinearity in $\mathring{\zt} = (\mathring{z}_1,\mathring{z}_2)$, 
\begin{align*}
\mathring{\mathcal{R}}(\psi,\tilde{\psi},\psi_t) 
&= D\Big[\phi_{yy}(\beta(\psi))\left(\left(1+\psi_\xx(1+\psi_\xx) + \psi\psi_{\xx\xx}\right)^2 - (1+\psi_\xx)^2\right) + 
\phi_{kk}(\beta(\psi))\psi_{\xx\xx}^2 \\ 
&\qquad \quad + 2 \phi_{yk}(\beta(\psi)) \psi_{\xx\xx}\left(1+\psi_\xx(1+\psi_\xx) + \psi\psi_{\xx\xx}\right) + \phi_{y}(\beta(\psi))\left(\psi_{\xx\xx}(1+3\psi_\xx) + \psi\psi_{\xx\xx\xx}\right) \\
&\qquad \quad  + \phi_{k}(\beta(\psi))\psi_{\xx\xx\xx} 
\Big] + \phi_{k}(\beta(\psi))\left(c_0 \psi_{\xx\xx} - \psi_{\xx t}\right)\\
&\qquad \quad + \phi_{y}(\beta(\psi))\left(c_0 + \omega'(1) \psi_\xx - \omega(1+\psi_\xx) - \tilde{\psi} + c_0 \left(\psi_\xx^2 + \psi \psi_{\xx\xx}\right)- \psi_t \psi_\xx - \psi \psi_{\xx t}\right)
\end{align*}
is the $\mathring{\zt}$-independent residual and we used
\begin{align*}
\beta(\psi)(\xx,t) = \big(\xx + \psi(\xx,t)(1+\psi_\xx(\xx,t));1+\psi_\xx(\xx,t)\big)
\end{align*}
to abbreviate the argument of the profile function $\phi(y;k) = (\phi_1(y;k),\phi_2(y;k))$ and its derivatives. We refer to Appendix~\ref{app:derivationforward} for further details on the derivation of~\eqref{e:Pert1}. 

We proceed with deriving a nonlinear damping estimate for the modified forward-modulated perturbation $\mathring{\zt}(t)$, which will be employed in the nonlinear stability argument to control regularity. A nonlinear damping estimate in $H^3(\R) \times H^2(\R)$ for the ``classical'' forward-modulated perturbation $\mathring{\vb}(t)$, given by~\eqref{e:defforwregular}, was established in~\cite[Proposition~8.6]{BjoernAvery}. Here, we extend the method in~\cite{BjoernAvery} to nonlocalized perturbations by relying on the embedding of the uniformly local Sobolev space $H^{1}_{\mathrm{ul}}(\R)$ in $C_{\mathrm{ub}}(\R)$, see~\cite[Lemma~8.3.11]{SU17book}. 

The equation~\eqref{e:Pert1} for $\mathring{\zt}(t)$ has a similar structure as the one for $\mathring{\vb}(t)$ derived in~\cite{BjoernAvery}. That is, the second derivative $\partial_{\xx\xx} \mathring{z}_1$ yields damping in the first component of~\eqref{e:Pert1} and the term $-\eps \gamma \mathring{z}_2$ yields damping in the second component. Since~\eqref{e:Pert1} is semilinear, all other linear and nonlinear terms can be controlled by these damping terms. 

All in all, we arrive at the following result.

\begin{proposition} \label{p:damping}
Assume~\ref{assH1} and~\ref{assD3}. Fix $R > 0$. Let $\vb_0 \in C_{\mathrm{ub}}^3(\R) \times C_{\mathrm{ub}}^2(\R)$. Let $\psi(t)$ be as in Proposition~\ref{p:psi} and let $\mathring{\zt}(t)$ and $\widetilde{\tau}_{\max}$ be as in Corollary~\ref{c:local_forward}. There exist $\vb_0$- and $t$-independent constants $C, \alpha > 0$ such that the nonlinear damping estimate
\begin{align} \label{e:dampingineq}
\begin{split}
\big\|\mathring{\zt}(t)\big\|_{C_{\mathrm{ub}}^2 \times C_{\mathrm{ub}}^1} &\leq C\Bigg(\Big\|\mathring{\zt}(t)\big\|_{L^\infty} + \big\|\mathring{\zt}(t)\big\|_{L^\infty}^{\frac15} \bigg[\re^{-\alpha t}  \|\vb_0\|_{C_{\mathrm{ub}}^3 \times C_{\mathrm{ub}}^2}^2 + \int_0^t \re^{-\alpha(t-s)} \Big(\big\|\mathring{z}_1(s)\big\|_{L^\infty}^2\\ 
&\qquad\qquad + \, \|\psi_{\xx\xx}(s)\|_{C_{\mathrm{ub}}^4}^2 + \|\partial_s \psi_{\xx}(s)\|_{C_{\mathrm{ub}}^3}^2 + \big\|\tilde{\psi}(s)\big\|_{C_{\mathrm{ub}}^3}^2\\ 
&\qquad\qquad  + \, \|\psi_{\xx}(s)\|_{L^\infty}^2\left(\|\psi_{\xx}(s)\|_{L^\infty}^2 + \|\partial_s \psi(s)\|_{L^\infty}^2\right)\Big) \,\de s
\bigg]^{\frac{2}{5}}\Bigg)
\end{split}
\end{align}
holds for all $t \in [0,\widetilde{\tau}_{\max})$ with
\begin{align}\label{e:upbound}
\sup_{0 \leq s \leq t} \left(\big\|\mathring{z}_1(s)\big\|_{C_{\mathrm{ub}}^1} + \|\psi(s)\|_{C_{\mathrm{ub}}^3}\right) \leq R.
\end{align}     
\end{proposition}
\begin{proof}
Fix a constant $R > 0$ and set
\begin{align*} \vartheta = \frac12 \min\left\{1,\frac{\varepsilon \gamma}{2|c_0|+1}\right\}.\end{align*} 
We start by relating the $(C_{\mathrm{ub}}^2 \times C_{\mathrm{ub}}^1)$-norm of $\mathring{\zt}(t)$ to a uniformly local Sobolev norm. First, we define the window function $\varrho \colon \R \to \R$ given by
$$\varrho(\xx) = \frac{2}{2+\xx^2},$$
which is positive, smooth and $L^1$-integrable, and satisfies 
\begin{align} \label{e:rhoineq}
|\varrho'(\xx)| \leq \varrho(\xx) \leq 1
\end{align}
for all $\xx \in \R$. Next, we apply the Gagliaro-Nirenberg interpolation inequality, while noting that $\varrho \in W^{k+1,1}(\R) \cap W^{k+1,\infty}(\R)$ and $\varrho(0) = 1$, to infer
\begin{align*}
\big\|\partial_\xx^k z\big\|_{L^\infty} &= \sup_{y \in \R} \big\|\varrho(\vartheta(\cdot + y)) \partial_\xx^k z\big\|_{L^\infty} \lesssim  \|z\|_{C_{\mathrm{ub}}^{k-1}} + \sup_{y \in \R} \big\|\partial_\xx^k\big(\varrho(\vartheta(\cdot + y)) z\big)\big\|_{L^\infty}\\
 &\lesssim  \|z\|_{C_{\mathrm{ub}}^{k-1}} + \sup_{y \in \R} \big\|\partial_\xx^{k+1}\big(\varrho(\vartheta(\cdot + y)) z\big)\big\|_{L^2}^{\frac{4}{5}} \big\|\varrho(\vartheta(\cdot + y)) z\big\|_{L^{\frac{1}{2-k}}}^{\frac{1}{5}}\\
  &\lesssim  \|z\|_{C_{\mathrm{ub}}^{k-1}} + \|z\|_{L^\infty}^{\frac{1}{5}} \sup_{y \in \R} \big\|\partial_\xx^{k+1}\big(\varrho(\vartheta(\cdot + y)) z\big)\big\|_{L^2}^{\frac{4}{5}} \\
 &\lesssim \|z\|_{C_{\mathrm{ub}}^{k-1}} + \|z\|_{C_{\mathrm{ub}}^k}^{\frac{4}{5}}\|z\|_{L^\infty}^{\frac{1}{5}} + \|z\|_{L^\infty}^{\frac{1}{5}} \sup_{y \in \R} \big\|\varrho(\vartheta(\cdot + y)) \partial_\xx^{k+1} z\big\|_{L^2}^{\frac{4}{5}}
\end{align*}
for $z \in C_{\mathrm{ub}}^{k+1}(\R)$ and $k = 1,2$. Hence, interpolating between $C_{\mathrm{ub}}^k(\R)$ and $C_{\mathrm{ub}}(\R)$, applying Young's inequality and rearranging terms, we arrive at
\begin{align*}
\|z\|_{C_{\mathrm{ub}}^k} \lesssim \|z\|_{L^\infty}  + \|z\|_{L^\infty}^{\frac{1}{5}} \sup_{y \in \R} \big\|\varrho(\vartheta(\cdot + y)) \partial_\xx^{k+1} z\big\|_{L^2}^{\frac{4}{5}}
\end{align*}
for $z \in C_{\mathrm{ub}}^{k+1}(\R)$ and $k = 1,2$. Combining the latter with~\eqref{e:rhoineq} and recalling Corollary~\ref{c:local_forward}, yields
\begin{align} \label{e:upbounddamping}
\big\|\mathring{\zt}(t)\big\|_{C_{\mathrm{ub}}^2 \times C_{\mathrm{ub}}^1} \lesssim \big\|\mathring{\zt}(t)\big\|_{L^\infty} + \big\|\mathring{\zt}(t)\big\|_{L^\infty}^{\frac15} \sup_{y \in \R} E_y(t)^{\frac25}
\end{align}
for $t \in [0,\widetilde{\tau}_{\max})$, where we denote
$$E_y(t) = \int_\R \varrho(\vartheta(\xx+y)) \left(\upsilon\left|\partial_\xx^3 \mathring{z}_1(\xx,t)\right|^2 + \left|\partial_\xx^2 \mathring{z}_2(\xx,t)\right|^2\right) \,\de \xx, \qquad \upsilon := \frac{\varepsilon\gamma}{4} > 0$$
for $y \in \R$. The estimate~\eqref{e:upbounddamping} provides the desired relationship between the $(C_{\mathrm{ub}}^2 \times C_{\mathrm{ub}}^1)$-norm of $\mathring{\zt}(t)$ and the family of energies $E_y(t)$, which are associated with the norm on the uniformly local Sobolev space $H^3_{\mathrm{ul}}(\R) \times H^2_{\mathrm{ul}}(\R)$ with dilation parameter $\vartheta$, cf.~\cite[Section~8.3.1]{SU17book}.

Our next step is to derive an inequality for the energies $E_y(t)$. In order to be able to differentiate $E_y(t)$ with respect to $t$, we restrict ourselves for the moment to initial conditions $\vb_0 \in C_{\mathrm{ub}}^5(\R) \times C_{\mathrm{ub}}^4(\R)$. With these two additional degrees of regularity one derives, analogously as in Proposition~\ref{p:local_unmod}, that $\vt \in C\big([0,T_{\max}),C^5_{\mathrm{ub}}(\R) \times C_{\mathrm{ub}}^4(\R)\big) \cap C^1\big([0,T_{\max}),C_{\mathrm{ub}}^3(\R)\big)$. Combining the latter with Propositions~\ref{prop:family} and~\ref{p:psi} yields $\mathring{\zt} \in C\big([0,\smash{\widetilde{\tau}}_{\max}),C^5_{\mathrm{ub}}(\R) \times C_{\mathrm{ub}}^4(\R)\big) \cap C^1\big([0,\smash{\widetilde{\tau}}_{\max}),C_{\mathrm{ub}}^3(\R)\big)$. 

Let $y \in \R$ and let $t \in [0,\widetilde{\tau}_{\max})$ be such that~\eqref{e:upbound} holds. Using~\eqref{e:Pert1} and
$$F'(0) = \begin{pmatrix}
    -\mu & -1 \\
    \varepsilon & -\varepsilon\gamma
\end{pmatrix},$$ 
while noting that the second component of $\mathring{\mathcal{Q}}(\mathring{\zt},\psi)$ vanishes, we compute
\begin{align} \label{e:damptemp}
\frac{1}{2} \partial_s E_y(s) = I + II + III + IV,
\end{align}
where
\begin{align*}
I &= \upsilon \int_\R \varrho(\vartheta(\xx+y)) \left\langle \partial_{\xx}^3 \mathring{z}_1(\xx,s), \partial_{\xx}^5 \mathring{z}_1(\xx,s) + c_0 \partial_{\xx}^4 \mathring{z}_1(\xx,s) - \partial_{\xx}^3 \mathring{z}_2(\xx,s) - \mu\partial_{\xx}^3 \mathring{z}_1(\xx,s) \right \rangle \,\de \xx,\\ 
II &= \frac{c_0}{2} \int_\R \varrho(\vartheta(\xx+y)) \partial_\xx \left| \partial_{\xx}^2 \mathring{z}_2(\xx,s)\right|^2 \,\de \xx\\ 
&\qquad + \, \varepsilon \int_\R \varrho(\vartheta(\xx+y))\left\langle \partial_\xx^2 \mathring{z}_2(\xx,s), \partial_{\xx}^2 \mathring{z}_1(\xx,s) - \gamma \partial_\xx^2\mathring{z}_2(\xx,s) \right \rangle \,\de \xx,
\end{align*}
are the contributions from the linear terms, and
\begin{align*}
III &= \upsilon \int_\R \varrho(\vartheta(\xx+y)) \left\langle \begin{pmatrix} \partial_{\xx}^3 \mathring{z}_1(\xx,s) \\ 0 \end{pmatrix}, \partial_\xx^3 \left(\mathring{\mathcal{Q}}(\mathring{\zt}(\xx,s),\psi(\xx,s)) + \mathring{\mathcal{R}}(\psi(\xx,s),\tilde{\psi}(\xx,s),\partial_s \psi(\xx,s)) \right) \right\rangle \,\de \xx,\\
IV &= \int_\R \varrho(\vartheta(\xx+y)) \left\langle \begin{pmatrix} 0 \\ \partial_{\xx}^2 \mathring{z}_2(\xx,s) \end{pmatrix}, \partial_{\xx}^2 \mathring{\mathcal{R}}(\psi(\xx,s),\tilde{\psi}(\xx,s),\partial_s \psi(\xx,s)) \right\rangle \,\de \xx
\end{align*}
are the nonlinear contributions for $s \in [0,t]$. Integrating by parts, we rewrite
\begin{align*}
I &= -\upsilon \int_\R \varrho(\vartheta(\xx+y)) \left(\left|\partial_{\xx}^4 \mathring{z}_1(\xx,s)\right|^2 + \mu \left|\partial_{\xx}^3 \mathring{z}_1(\xx,s)\right|^2 - \left\langle \partial_{\xx}^4 \mathring{z}_1(\xx,s), \partial_{\xx}^2 \mathring{z}_2(\xx,s) \right \rangle \right) \,\de \xx\\ 
&\qquad - \, \upsilon \vartheta \int_\R \varrho'(\vartheta(\xx+y)) \left\langle \partial_{\xx}^3 \mathring{z}_1(\xx,s), \partial_{\xx}^4 \mathring{z}_1(\xx,s) - \partial_{\xx}^2 \mathring{z}_2(\xx,s) \right \rangle \,\de \xx\\ 
&\qquad + \, c_0\upsilon \int_\R \varrho(\vartheta(\xx+y)) \left\langle \partial_{\xx}^3 \mathring{z}_1(\xx,s), \partial_{\xx}^4 \mathring{z}_1(\xx,s)\right\rangle \,\de \xx
\end{align*}
and
\begin{align*}
II &= -\varepsilon \gamma \int_\R \varrho(\vartheta(\xx+y)) \left|\partial_{\xx}^2 \mathring{z}_2(\xx,s)\right|^2 \,\de \xx - \frac{c_0 \vartheta}{2} \int_\R \varrho'(\vartheta(\xx+y)) \left| \partial_{\xx}^2 \mathring{z}_2(\xx,s)\right|^2 \,\de \xx\\ 
&\qquad +\, \varepsilon \int_\R \varrho(\vartheta(\xx+y))\left\langle \partial_{\xx}^2 \mathring{z}_2(\xx,s), \partial_{\xx}^2 \mathring{z}_1(\xx,s) \right \rangle \,\de \xx
\end{align*}
for $s \in [0,t]$. Applying Young's inequality to the latter, while using~\eqref{e:rhoineq} and $4|c_0|\vartheta \leq \varepsilon\gamma$, yields a $t$- and $\vb_0$-independent constant $C_1 > 0$ such that
\begin{align} \label{e:dampest1} \begin{split} 
I &\leq -\frac{\upsilon}{4} \int_\R \varrho(\vartheta(\xx+y)) \left|\partial_{\xx}^4 \mathring{z}_1(\xx,s)\right|^2 \,\de \xx + \upsilon \int_\R \varrho(\vartheta(\xx+y)) \left|\partial_{\xx}^2 \mathring{z}_2(\xx,s)\right|^2 \,\de \xx\\ 
&\qquad + \, C_1 \int_\R \varrho(\vartheta(\xx+y)) \left|\partial_{\xx}^3 \mathring{z}_1(\xx,s)\right|^2 \,\de \xx
\end{split}\end{align} 
and
\begin{align} \label{e:dampest2} \begin{split} 
II &\leq -\frac{3\varepsilon \gamma}{4} \int_\R \varrho(\vartheta(\xx+y)) \left|\partial_{\xx}^2 \mathring{z}_2(\xx,s)\right|^2 \,\de \xx + C_1 \int_\R \varrho(\vartheta(\xx+y)) \left|\partial_{\xx}^2 \mathring{z}_1(\xx,s)\right|^2 \,\de \xx
\end{split}\end{align} 
for $s \in [0,t]$. Similarly, employing Young's inequality, while using that~\eqref{e:upbound} holds and $\rho$ is $L^1$-integrable, we establish a $t$- and $\vb_0$-independent constant $C_2 > 0$ such that
\begin{align} \label{e:dampest3} \begin{split} 
III &\leq C_2 \left(\int_\R \varrho(\vartheta(\xx+y))\left( \left|\partial_{\xx}^3 \mathring{z}_1(\xx,s)\right|^2 + \left|\partial_{\xx}^2 \mathring{z}_1(\xx,s)\right|^2 + \left|\partial_{\xx} \mathring{z}_1(\xx,s)\right|^2\right) \,\de \xx \right. \\
&\left. \phantom{\int_\R} \qquad \qquad + \|\mathring{z}_1(s)\|_{L^\infty}^2  + \|\psi_{\xx\xx}(s)\|_{C_{\mathrm{ub}}^4}^2 + \|\partial_s \psi_{\xx}(s)\|_{C_{\mathrm{ub}}^3}^2 + \big\|\tilde{\psi}(s)\big\|_{C_{\mathrm{ub}}^3}^2\right.\\
&\left. \phantom{\int_\R} \qquad \qquad +  \|\psi_{\xx}(s)\|_{L^\infty}^2\left(\|\psi_{\xx}(s)\|_{L^\infty}^2 + \|\partial_s \psi(s)\|_{L^\infty}^2\right)\right)
\end{split}\end{align} 
and
\begin{align} \label{e:dampest4} \begin{split} 
IV &\leq \frac{\varepsilon \gamma}{4} \int_\R \varrho(\vartheta(\xx+y)) \left|\partial_{\xx}^2 \mathring{z}_2(\xx,s)\right|^2 \de \xx + C_2 \left(\|\psi_{\xx\xx}(s)\|_{C_{\mathrm{ub}}^3}^2 + \|\partial_s \psi_{\xx}(s)\|_{C_{\mathrm{ub}}^2}^2 + \big\|\tilde{\psi}(s)\big\|_{C_{\mathrm{ub}}^2}^2\right. \\
&\left. \phantom{\int_\R} \qquad \qquad +  \|\psi_{\xx}(s)\|_{L^\infty}^2\left(\|\psi_{\xx}(s)\|_{L^\infty}^2 + \|\partial_s \psi(s)\|_{L^\infty}^2\right)\right)
\end{split}\end{align} 
for $s \in [0,t]$. Applying the estimates~\eqref{e:dampest1},~\eqref{e:dampest2},~\eqref{e:dampest3} and~\eqref{e:dampest4} to~\eqref{e:damptemp} and using that $\upsilon = \varepsilon\gamma/4$, we obtain a $t$- and $\vb_0$-independent constant $C_3 > 0$ such that
\begin{align} \label{e:dampest5}
\begin{split}
\frac{1}{2} \partial_s E_y(s) &\leq -\frac{\varepsilon \gamma}{4} E_y(s) -\frac{\upsilon}{4} \int_\R \varrho(\vartheta(\xx+y)) \left|\partial_{\xx}^4 \mathring{z}_1(\xx,s)\right|^2 \,\de \xx + C_3\left(\|\mathring{z}_1(s)\|_{L^\infty}^2 + \|\psi_{\xx\xx}(s)\|_{C_{\mathrm{ub}}^4}^2 \right.\\
&\left. \phantom{\int_\R} \qquad +  \|\partial_s \psi_{\xx}(s)\|_{C_{\mathrm{ub}}^3}^2 + \big\|\tilde{\psi}(s)\big\|_{C_{\mathrm{ub}}^3}^2 +  \|\psi_{\xx}(s)\|_{L^\infty}^2\left(\|\psi_{\xx}(s)\|_{L^\infty}^2 + \|\partial_s \psi(s)\|_{L^\infty}^2\right)\right.\\
&\left.\phantom{\int_\R} \qquad 
+ \int_\R \varrho(\vartheta(\xx+y)) \left(\left|\partial_{\xx}^3 \mathring{z}_1(\xx,s)\right|^2 + \left|\partial_{\xx}^2 \mathring{z}_1(\xx,s)\right|^2 + \left|\partial_{\xx} \mathring{z}_1(\xx,s)\right|^2\right) \,\de \xx \right) 
\end{split}
\end{align}
for $s \in [0,t]$. 

We control the term on the last line of~\eqref{e:dampest5} by deriving an interpolation inequality. To this end, we take $k \in \mathbb{N}$, $\eta \in (0,\frac14)$ and $a_1,\ldots,a_k > 0$. Integration by parts, Young's inequality, and the estimate~\eqref{e:rhoineq} yield
\begin{align*}
&\sum_{j = 1}^k a_j \int_\R \varrho(\vartheta(\xx+y)) \left|\partial_\xx^j z(\xx)\right|^2 \,\de \xx\\ 
&\qquad = -\sum_{j = 1}^k a_j \int_\R \left(\varrho(\vartheta(\xx+y)) \left\langle \partial_\xx^{j+1} z(\xx), \partial_\xx^{j-1} z(\xx)\right\rangle \,\de \xx + \vartheta \varrho'(\vartheta(\xx+y)) \left\langle \partial_\xx^j z(\xx), \partial_\xx^{j-1} z(\xx)\right\rangle \right) \,\de \xx\\
&\qquad \leq \sum_{j = 1}^k \frac{a_j}{2} \int_\R \varrho(\vartheta(\xx+y))\left(\eta \left|\partial_\xx^{j+1} z(\xx)\right|^2 + \vartheta \left|\partial_\xx^{j} z(\xx)\right|^2 + \left(\frac{1}{\eta} + \vartheta\right) \left|\partial_\xx^{j-1} z(\xx)\right|^2\right) \,\de \xx
\end{align*}
for $z \in C_{\mathrm{ub}}^{k+1}(\R)$. Setting $a_0 = 0 = a_{k+1}$, using $\vartheta \leq \frac12$ and rearranging terms in the latter, we arrive at the interpolation inequality
\begin{align} \label{e:interpolationid}
\begin{split}
&\sum_{j = 1}^k \left(\frac34 a_j - \frac{\eta}{2} a_{j-1} - \frac12 a_{j+1} \left(\frac{1}{\eta} + \frac12\right)\right) \int_\R \varrho(\vartheta(\xx+y)) \left|\partial_\xx^j z(\xx)\right|^2 \,\de \xx\\
&\qquad \leq \frac{\eta}{2} a_k \int_\R \varrho(\vartheta(\xx+y)) \left|\partial_\xx^{k+1} z(\xx)\right|^2 \,\de \xx + \frac12 a_1 \left(\frac{1}{\eta} + \frac12\right) \int_\R \varrho(\vartheta(\xx+y)) \left|z(\xx)\right|^2 \,\de \xx
\end{split}
\end{align}
for $z \in C_{\mathrm{ub}}^{k+1}(\R)$. Next, we fix $k = 3$ and solve the linear system
\begin{align*}
\frac34 a_j - \frac{\eta}{2} a_{j-1} - \frac12 a_{j+1} \left(\frac{1}{\eta} + \frac12\right) = 1, \qquad j = 1,2,3,
\end{align*}
yielding the solution
\begin{align*}
a_1 = \frac{4 \left(4- 2 \eta^3 + 9 \eta^2+10 \eta\right)}{3 \eta ^2 (1-4 \eta)}, \qquad a_2 = \frac{4
   \left(2 + 2 \eta^2 + 4 \eta\right)}{\eta  (1-4 \eta
   )},\qquad a_3 = \frac{4 \left(5+ 4 \eta ^2+4 \eta\right)}{3
   (1-4 \eta)}.
\end{align*}
where we have $a_1,a_2,a_3 > 0$ since $\eta < \frac14$. Thus, taking these values for $a_1,a_2,a_3$ in~\eqref{e:interpolationid}, we find
\begin{align*}
\sum_{j = 1}^3 \int_\R \varrho(\vartheta(\xx+y)) \left|\partial_\xx^j z(\xx)\right|^2 \,\de \xx &\leq \frac{2\eta \left(5+ 4 \eta ^2+4 \eta\right)}{3(1-4 \eta)} \int_\R \varrho(\vartheta(\xx+y)) \left|\partial_\xx^4 z(\xx)\right|^2 \,\de \xx\\ 
&\qquad +\, \frac{(\eta +2) \left(4-2 \eta^3+9 \eta^2+10 \eta\right)}{3 \eta ^3 (1-4 \eta)} \int_\R \varrho(\vartheta(\xx+y)) \left|z(\xx)\right|^2 \de \xx
\end{align*}
for $z \in C_{\mathrm{ub}}^4(\R)$. So, taking $\eta \in (0,\frac14)$ so small that
\begin{align*}
\frac{2\eta \left(5+ 4 \eta^2+4 \eta\right)}{3(1-4 \eta)} \leq \frac{\upsilon}{4C_3},
\end{align*}
we establish a constant $C_4 > 0$ such that
\begin{align} \label{e:interpolationUL}
\sum_{j = 1}^3 \int_\R \varrho(\vartheta(\xx+y)) \left|\partial_\xx^j z(\xx)\right|^2 \,\de \xx &\leq \frac{\upsilon}{4C_3} \int_\R \varrho(\vartheta(\xx+y)) \left|\partial_\xx^4 z(\xx)\right|^2 \,\de \xx + C_4 \|z\|_{L^\infty}^2
\end{align}
for $z \in C_{\mathrm{ub}}^4(\R)$. 

We apply the interpolation identity~\eqref{e:interpolationUL} to~\eqref{e:dampest5} and deduce
\begin{align*}
\partial_s E_y(s) &\leq -\frac{\varepsilon \gamma}{2} E_y(s) + C_5 \left(\big\|\mathring{z}_1(s)\big\|_{L^\infty}^2 + \|\psi_{\xx\xx}(s)\|_{C_{\mathrm{ub}}^4}^2 + \|\partial_s \psi_{\xx}(s)\|_{C_{\mathrm{ub}}^3}^2 + \big\|\tilde{\psi}(s)\big\|_{C_{\mathrm{ub}}^3}^2\right.\\ 
&\left.\qquad \qquad \qquad \phantom{\big\|\tilde{\psi}(s)\big\|_{C_{\mathrm{ub}}^3}^2} + \|\psi_{\xx}(s)\|_{L^\infty}^2\left(\|\psi_{\xx}(s)\|_{L^\infty}^2 + \|\partial_s \psi(s)\|_{L^\infty}^2\right)\right)
\end{align*}
for $s \in [0,t]$, where $C_5 > 0$ is a $t$- and $\vb_0$-independent constant. Multiplying the latter inequality with $\re^{\frac{\varepsilon\gamma}{2} s}$ and integrating, we acquire
\begin{align*}
E_y(t) &\leq \re^{-\frac{\varepsilon \gamma}{2} t} E_y(0) + C_5 \int_0^t \re^{-\frac{\varepsilon \gamma}{2} (t-s)} \left(\big\|\mathring{z}_1(s)\big\|_{L^\infty}^2 + \|\psi_{\xx\xx}(s)\|_{C_{\mathrm{ub}}^4}^2 + \|\partial_s \psi_{\xx}(s)\|_{C_{\mathrm{ub}}^3}^2 + \big\|\tilde{\psi}(s)\big\|_{C_{\mathrm{ub}}^3}^2\right.\\ 
&\left.\qquad \qquad \qquad \phantom{\big\|\tilde{\psi}(s)\big\|_{C_{\mathrm{ub}}^3}^2} + \|\psi_{\xx}(s)\|_{L^\infty}^2\left(\|\psi_{\xx}(s)\|_{L^\infty}^2 + \|\partial_s \psi(s)\|_{L^\infty}^2\right)\right) \,\de s.\end{align*}
Lastly, using that there exists a $\vb_0$-independent constant $C_6 > 0$ such that $\smash{E_y(0) \leq C_6 \|\vb_0\|_{C_{\mathrm{ub}}^3 \times C_{\mathrm{ub}}^2}^2}$ and plugging the latter estimate into~\eqref{e:upbounddamping}, we arrive at~\eqref{e:dampingineq}.

In order to extend our result to the general case $\vb_0 \in C_{\mathrm{ub}}^3(\R) \times C_{\mathrm{ub}}^2(\R)$ we argue as in the proof of~\cite[Proposition~8.6]{BjoernAvery}. That is, we approximate the initial condition $\vb_0$ in $C_{\mathrm{ub}}^3(\R) \times C_{\mathrm{ub}}^2(\R)$ by a sequence $\left(\vb_{0,n}\right)_{n \in \mathbb{N}}$ in $C_{\mathrm{ub}}^5(\R) \times C_{\mathrm{ub}}^4(\R)$. By continuity of solutions with respect to initial data and the fact that~\eqref{e:dampingineq} only depends on the $(C_{\mathrm{ub}}^3\times C_{\mathrm{ub}}^2)$-norm of $\mathring{\zt}(t)$, the desired result follows by approximation. We refer to~\cite{BjoernAvery} for further details.
\end{proof}

\begin{remark} \label{rem:gagl}
In addition to the fact that we extend the proof of the nonlinear damping estimate in~\cite[Proposition~8.6]{BjoernAvery} to nonlocalized perturbations by employing an energy associated with uniformly local Sobolev norms, our analysis deviates from the one in~\cite{BjoernAvery} in another important way: rather than using the bound $\|\partial_\xx^k w\|_{L^\infty} \leq \|\partial_\xx^k w\|_{H^1}$, we employ the Gagliardo-Nirenberg interpolation inequality $$\big\|\partial_\xx^k w\big\|_{L^\infty} \leq \big\|\partial_\xx^{k+1} w\big\|_{L^2}^{\frac45} \|w\|^{\frac15}_{L^{\frac{1}{2-k}}},$$ for $w \in H^{k+1}(\R) \cap L^1(\R)$ and $k = 1,2$. This leads to the additional factor $\|\mathring{\zt}(t)\|_{L^\infty}^{1/5}$ in the nonlinear damping estimate~\eqref{e:dampingineq}, enabling us to only require that the $L^\infty$-norm of the initial perturbation $\vb_0$ is small (and its $(C_{\mathrm{ub}}^3 \times C_{\mathrm{ub}}^2)$-norm is bounded) in our nonlinear stability result, Theorem~\ref{t:mainresult}. We expect that a similar approach can be adopted to relax the smallness condition on initial data in~\cite{BjoernAvery}.
\end{remark}

It has been argued in~\cite[Corollary 5.3]{ZUM22} that, as long as $\|\psi_\xx(t)\|_{L^\infty}$ stays sufficiently small, the Sobolev norms of the forward- and inverse-modulated perturbation $\mathring{\vb}(t)$ and $\vb(t)$ are equivalent modulo Sobolev norms of $\psi_\xx(t)$ and its derivatives. We extend this result by proving norm equivalence of the modified forward-modulated perturbation $\mathring{\zt}(t)$ and the residual $\zt(t)$ (up to controllable errors in $\psi_\xx(t)$ and its derivatives).

\begin{lemma} \label{lem:equivalence}
Let $\psi(t)$ be as in Proposition~\ref{p:psi}, let $\zt(t)$ be as in Corollary~\ref{C:local_v} and let $\mathring{\zt}(t)$ and $\widetilde{\tau}_{\max}$ be as in Corollary~\ref{c:local_forward}. Then, there exists a constant $C > 0$ such that
\begin{align}
\begin{split}
\|\zt(t)\|_{C_{\mathrm{ub}}^2 \times C_{\mathrm{ub}}^1} &\leq C\left(\big\|\mathring{\zt}(t)\big\|_{C_{\mathrm{ub}}^2 \times C_{\mathrm{ub}}^1} + \|\psi_{\xx\xx}(t)\|_{C_{\mathrm{ub}}^1} + \|\psi_{\xx}(t)\|
^2_{L^\infty}\right), 
\end{split}
\label{e:bdibf}
\end{align}
and
\begin{align}
\begin{split}
\big\|\mathring{\zt}(t)\big\|_{L^\infty} \leq C\left(\|\zt(t)\|_{L^\infty} +  \|\psi_{\xx\xx}(t)\|_{L^\infty} + \|\psi_{\xx}(t)\|
^2_{L^\infty}\right)
\end{split}
\label{e:bdibf2}
\end{align}
for any $t \in [0,\widetilde{\tau}_{\max})$.
\end{lemma}
\begin{proof}
Inserting $\vb(\xx,t) = \ub(\xx - \psi(\xx,t),t) - \phi_0(\xx)$ into~\eqref{e:defz} and using~\eqref{e:defvforw} to reexpress $\ub(\xx-\psi(\xx,t),t)$, we arrive at
\begin{align} \label{e:exprzt1}
\begin{split}
\zt(\xx,t) &= \mathring{\zt}(a(\xx,t),t)  - \phi_0(\xx) - \phi_k(\xx;1) \psi_\xx(\xx,t) + \phi\left(b(\xx,t);c(\xx,t)\right)
\end{split}
\end{align}
for $\xx \in \R$ and $t \in [0,\widetilde{\tau}_{\max})$, where we abbreviate 
\begin{align*}
a(\xx,t) = \xx-\psi(\xx,t), \qquad b(\xx,t) = \xx + \psi(\xx-\psi(\xx,t),t)\left(1+\psi_\xx(\xx-\psi(\xx,t),t)\right) - \psi(\xx,t)
\end{align*}
and
\begin{align*}
    c(\xx,t) = 1+\psi_\xx(\xx-\psi(\xx,t),t).
\end{align*}
Differentiating the latter with respect to $\xx$ yields
\begin{align} \label{e:exprzt2}
\begin{split}
\zt_\xx(\xx,t) &= \mathring{\zt}_\xx(a(\xx,t),t)a_{\xx}(\xx,t)   - \phi_0'(\xx) - \phi_{k\xx} (\xx;1) \psi_\xx(\xx,t) - \phi_k(\xx;1) \psi_{\xx\xx}(\xx,t)\\ 
&\qquad + \, \phi_\xx\left(b(\xx,t);c(\xx,t)\right) b_\xx(\xx,t) + \phi_k\left(b(\xx,t);c(\xx,t)\right)c_{\xx}(\xx,t)
\end{split}
\end{align}
and
\begin{align} \label{e:exprzt3}
\begin{split}
\zt_{\xx\xx}(\xx,t) &= \mathring{\zt}_{\xx\xx}(a(\xx,t),t)a_{\xx}(\xx,t)^2 + \mathring{\zt}_{\xx}(a(\xx,t),t)a_{\xx\xx}(\xx,t) - \phi_0''(\xx) - \phi_{k\xx\xx} (\xx;1) \psi_\xx(\xx,t)\\ 
&\qquad - \, 2\phi_{k\xx}(\xx;1) \psi_{\xx\xx}(\xx,t) - \phi_k(\xx;1) \psi_{\xx\xx\xx}(\xx,t) + \phi_{\xx\xx}\left(b(\xx,t);c(\xx,t)\right) b_\xx(\xx,t)^2\\ 
&\qquad + \, 2 \phi_{k \xx}\left(b(\xx,t);c(\xx,t)\right)b_{\xx}(\xx,t)c_{\xx}(\xx,t) + \phi_{kk}\left(b(\xx,t);c(\xx,t)\right)c_{\xx}(\xx,t)^2 \\
&\qquad + \, \phi_\xx\left(b(\xx,t);c(\xx,t)\right) b_{\xx\xx}(\xx,t) + \phi_k\left(b(\xx,t);c(\xx,t)\right)c_{\xx\xx}(\xx,t)
\end{split}
\end{align}
for $\xx \in \R$ and $t \in [0,\widetilde{\tau}_{\max})$. 

Next, we use Taylor's theorem to bound
\begin{align} \label{e:taylor1}
\begin{split}
\left|b(\xx,t) - \xx\right| &\leq \left|\psi(\xx-\psi(\xx,t),t) - \psi(\xx,t) + \psi_\xx(\xx,t)\psi(\xx,t)\right|\\ 
&\qquad + \, \left|\psi_\xx(\xx-\psi(\xx,t),t)\psi(\xx-\psi(\xx,t),t) - \psi_\xx(\xx,t)\psi(\xx,t)\right|\\
&\lesssim \|\psi(t)\|_{L^\infty}\left(\|\psi_{\xx\xx}(t)\|_{L^\infty}\|\psi(t)\|_{L^\infty} + \|\psi_{\xx}(t)\|_{L^\infty}^2\right),\\
\left|c(\xx,t) - 1 - \psi_\xx(\xx,t)\right| &\leq \|\psi(t)\|_{L^\infty}\|\psi_{\xx\xx}(t)\|_{L^\infty},
\end{split}
\end{align}
and
\begin{align} \label{e:taylor2}
\begin{split}
\left|b_\xx(\xx,t) - 1\right| &\leq \left|\psi_\xx(\xx-\psi(\xx,t),t)(1-\psi_\xx(\xx,t)) - \psi_\xx(\xx,t)\right|\\
&\qquad + \, \left|\psi_{\xx\xx}(\xx-\psi(\xx,t),t)\psi(\xx-\psi(\xx,t),t) + \psi_\xx(\xx-\psi(\xx,t),t)^2\right|\left|1-\psi_\xx(\xx,t)\right|\\
&\lesssim \left(\|\psi(t)\|_{L^\infty} \|\psi_{\xx\xx}(t)\|_{L^\infty} + \|\psi_{\xx}(t)\|_{L^\infty}^2\right)\left(1 + \|\psi_{\xx}(t)\|_{L^\infty}\right)
\end{split}
\end{align}
for $\xx \in \R$ and $t \in [0,\widetilde{\tau}_{\max})$. Recall from Proposition~\ref{prop:family} that $\phi \colon \R \times [1-r_0,1+r_0] \to \R^2$ is smooth. So, applying Taylor's theorem and estimate~\eqref{e:taylor1}, while recalling from Proposition~\ref{p:psi} and Corollary~\ref{c:local_forward} that $\|\psi(t)\|_{C_{\mathrm{ub}}^2} < \frac12$ and $\|\psi_\xx(t)\|_{L^\infty} < r_0$, we infer the bounds
\begin{align} \label{e:taylor3}
\begin{split}
\left|(\partial_\xx^j \phi)\left(b(\xx,t);c(\xx,t)\right) - (\partial_\xx^j \phi)\left(\xx;c(\xx,t)\right)\right| &\leq |b(\xx,t) - \xx| \sup_{|k-1| \leq r_0} \left\|\partial_\xx^{j+1} \phi(\cdot;k)\right\|_{L^\infty}\\ &\lesssim \|\psi_{\xx\xx}(t)\|_{L^\infty} + \|\psi_{\xx}(t)\|_{L^\infty}^2,\\
\left|(\partial_\xx^j \phi)\left(\xx;c(\xx,t)\right) - (\partial_\xx^j \phi)\left(\xx;1 + \psi_{\xx}(\xx,t)\right)\right| &\leq |c(\xx,t) - 1 - \psi_\xx(\xx,t)|\\ 
&\qquad \cdot \sup_{|k-1| \leq r_0} \left\|\partial_\xx^j \phi_k(\cdot;k)\right\|_{L^\infty} \lesssim \|\psi_{\xx\xx}(t)\|_{L^\infty}
\end{split}
\end{align}
and
\begin{align} \label{e:taylor4}
\begin{split}
\left|(\partial_\xx^j \phi)\left(\xx;1 + \psi_\xx(\xx,t)\right) - \partial_\xx^j \phi_0(\xx) - (\partial_\xx^j \phi_k)(\xx;1)\psi_\xx(\xx,t)\right| &\lesssim |\psi_\xx(\xx,t)|^2 \left\|\partial_\xx^j \phi_{kk}(\cdot;1)\right\|_{L^\infty}\\ & \lesssim \|\psi_{\xx}(t)\|_{L^\infty}^2
\end{split}
\end{align}
for $\xx \in \R$, $t \in [0,\widetilde{\tau}_{\max})$ and $j = 0,1,2$. Using again $\|\psi(t)\|_{C_{\mathrm{ub}}^2} < \frac12$, we obtain
\begin{align} \label{e:taylor5}
\begin{split}
\|a(\cdot,t)\|_{C_{\mathrm{ub}}^2} \lesssim 1, \qquad \|c_\xx(\cdot,t)\|_{C_{\mathrm{ub}}^1}, \|b_{\xx\xx}(\cdot,t)\|_{L^\infty} \lesssim \|\psi_{\xx\xx}(t)\|_{C_{\mathrm{ub}}^1}
\end{split}
\end{align}
for $\xx \in \R$ and $t \in [0,\widetilde{\tau}_{\max})$. 

Finally, applying the bounds~\eqref{e:taylor2},~\eqref{e:taylor3},~\eqref{e:taylor4} and~\eqref{e:taylor5} to~\eqref{e:exprzt1},~\eqref{e:exprzt2} and~\eqref{e:exprzt3}, while recalling that $\phi$ is smooth, one readily infers~\eqref{e:bdibf}. Similarly, applying~\eqref{e:taylor3} and~\eqref{e:taylor4} to~\eqref{e:exprzt1}, we establish
\begin{align} \label{e:taylor6}
\begin{split}
\big\|\mathring{\zt}(a(\cdot,t),t)\big\|_{L^\infty} \lesssim \|\zt(t)\|_{L^\infty} + \|\psi_{\xx\xx}(t)\|_{L^\infty} + \|\psi_{\xx}(t)\|^2_{L^\infty}
\end{split}
\end{align}
for $t \in [0,\widetilde{\tau}_{\max})$. Since we have $\|\psi_\xx(t)\|_{L^\infty} < \frac12$, it holds $a_\xx(\xx,t) \geq \frac12$ for all $\xx \in \R$ and the function $a(\cdot,t) \colon \R \to \R$ is bijective for each $t \in [0,\widetilde{\tau}_{\max})$. Consequently, we have $\big\|\mathring{\zt}(a(\cdot,t),t)\big\|_{L^\infty} = \big\|\mathring{\zt}(\cdot,t)\big\|_{L^\infty}$ for each $t \in [0,\widetilde{\tau}_{\max})$, which yields~\eqref{e:bdibf2} upon invoking~\eqref{e:taylor6}.
\end{proof}

\section{Nonlinear stability argument} \label{sec:nonlstab}

We prove our nonlinear stability result, Theorem~\ref{t:mainresult}, by applying the linear bounds, obtained in Theorem~\ref{t:linear}, and the nonlinear bounds, established in Lemmas~\ref{lem:nlboundsmod},~\ref{lem:nlboundsmod3} and~\ref{lem:nlboundsmod4}, to iteratively estimate the phase modulation $\psi(t)$, the residuals $\zt(t)$ and $r(t)$, and the Cole-Hopf variable $y(t)$ through their respective Duhamel formulations~\eqref{e:intpsi},~\eqref{e:intz},~\eqref{e:intr} and~\eqref{e:inty}. We control regularity in the scheme via the nonlinear damping estimate in Proposition~\ref{p:damping}.

\begin{proof}[Proof of Theorem~\ref{t:mainresult}] 
Take $\vb_0 \in C_{\mathrm{ub}}^3(\R) \times C_{\mathrm{ub}}^2(\R)$ with $\|\vb_0\|_{C_{\mathrm{ub}}^3 \times C_{\mathrm{ub}}^2} < K$. Propositions~\ref{p:local_unmod} and~\ref{p:psi}, Corollaries~\ref{C:local_v} and~\ref{C:local_r}, and identity~\eqref{e:regy} yield that the template function $\eta \colon [0,\tau_{\max}) \to \R$ given by
\begin{align*}
\eta(t) = \eta_1(t) + \eta_2(t)^5,
\end{align*}
with
\begin{align*}
\eta_1(t) &= \sup_{0\leq s\leq t} \left[\|\psi(s)\|_{L^\infty} + \|y(s)\|_{L^\infty} + \sqrt{s}\,\|y_\xx(s)\|_{L^\infty} + \sqrt{1+s}\left(\frac{\|r(s)\|_{L^\infty} + \sqrt{s}\,\|r_\xx(s)\|_{L^\infty}}{\log(2+s)} \right.\right. \\
&\left.\left.\phantom{\frac{\sqrt{s}}{\log(2+s}} + \|\psi_\xx(s)\|_{L^\infty}\right) + \frac{1+s}{\log(2+s)}\left(\|\zt(s)\|_{L^\infty} + \left\|\psi_{\xx\xx}(s)\right\|_{C_{\mathrm{ub}}^4} + \big\|\tilde{\psi}(s)\big\|_{C_{\mathrm{ub}}^4}\right)\right],
\end{align*}
and
\begin{align*}
\eta_2(t) = \sup_{0\leq s\leq t} \big\|\vt(t)\big\|_{C_{\mathrm{ub}}^1}
\end{align*}
is well-defined, positive, monotonically increasing and continuous, where we recall $\tilde{\psi}(t) = \partial_t \psi(t) + c_g\psi_\xx(t)$. In addition, if $\tau_{\max} < \infty$, then we have
\begin{align}
 \lim_{t \uparrow \tau_{\max}} \eta(t) \geq \frac12. \label{e:blowupeta}
\end{align}
We refer to Remarks~\ref{rem:motv1} and~\ref{rem:motv2} for motivation for the choice of template function.

\paragraph*{Approach.} Let $r_0 > 0$ be the constant from Proposition~\ref{prop:family}. As usual in nonlinear iteration arguments, our goal is to prove a nonlinear inequality for the template function $\eta(t)$. Specifically, we show that there exists a constant $C > 1$ such that for all $t \in [0,\tau_{\max})$ with $\eta(t) \leq \frac12 \min\{1,r_0\}$ we have the key inequality
\begin{align}
\eta(t) \leq C\left(E_0 + \eta(t)^{\frac65}\right), \label{e:etaest}
\end{align}
where we denote $E_0 := \|\vb_0\|_{L^\infty}$. We note that by interpolation there exists an $E_0$-independent constant $C_0 > 0$ such that it holds $\smash{\|\vb_0\|_{C_{\mathrm{ub}}^1} \leq C_0 \sqrt{E_0}}$ as long as $E_0 \leq 1$. So, recalling that $\psi(0)$ vanishes identically by Proposition~\ref{p:psi} and using~\eqref{e:initr} and~\eqref{e:yinit}, we find an $E_0$-independent constant $C_* > 0$ such that $\eta(0) \leq C_* E_0$ as long as $E_0 \leq 1$. Subsequently, we set
$$M_0 = 2\max\{C,C_*\} > 2, \qquad \epsilon_0 = \min\left\{\frac{1}{M_0^6},\frac{\min\{1,r_0\}}{2M_0}\right\} < 1.$$
Assuming that~\eqref{e:etaest} holds, we claim that, provided $E_0 \in (0,\epsilon_0)$, we have $\eta(t) \leq M_0E_0$ for all $t \in [0,\tau_{\max})$. To prove the claim we, argue by contradiction and assume that there exists a $t \in [0,\tau_{\max})$ with $\eta(t) > M_0E_0$. Since $\eta$ is continuous and $\eta(0) \leq C_*E_0 < M_0E_0$, there must exist $t_0 \in (0,\tau_{\max})$ with $\eta(t_0) = M_0E_0 \leq \frac12 \min\{1,r_0\}$. Thus, applying~\eqref{e:etaest} and using $E_0 < \epsilon_0$, we arrive at
\begin{align*}
\eta(t_0) \leq CE_0\left(1 + M_0^{\frac65} E_0^{\frac15} \right) < 2CE_0 \leq M_0E_0,
\end{align*}
which contradicts $\eta(t_0) = M_0E_0$. Therefore, it must hold $\eta(t) \leq M_0E_0$ for all $t \in [0,\tau_{\max})$. Since $M_0 > 2$, we have $M_0E_0 < \frac12$, which implies $\tau_{\max} = \infty$ by~\eqref{e:blowupeta}, i.e.,~$\smash{\ub(t) = \vt(t) + \phi_0}$ is a global solution to~\eqref{FHN_co} satisfying~\eqref{e:regu} by Proposition~\ref{p:local_unmod}.

Our next step is thus to establish the key inequality~\eqref{e:etaest}. The estimates~\eqref{e:mtest10}-\eqref{e:mtest2} and~\eqref{e:mtest3} then follow readily by employing applying Lemma~\ref{lem:equivalence} and using that $\eta(t) \leq M_0E_0$ holds for all $t \geq 0$. 

\paragraph*{Bounds on \texorpdfstring{$\vb(t)$}{w(t)} and \texorpdfstring{$\partial_t \psi(t)$}{psi_t(t)}.} 
Let $t \in [0,\tau_{\max})$ with $\eta(t) \leq \frac{1}{2}\min\{1,r_0\}$. We bound  $\vb(s) = \zt(s) + \partial_k \phi(\cdot;1) \psi_\xx(s)$ and $\partial_t \psi(s) = \tilde{\psi}(s) - c_g\psi_\xx(s)$ as
\begin{align}
\begin{split}
\|\vb(s)\|_{L^\infty} &\lesssim \|\zt(s)\|_{L^\infty} + \|\psi_\xx(s)\|_{L^\infty} \lesssim \frac{\eta_1(t)}{\sqrt{1+s}}, \\ 
\|\partial_t \psi(s)\|_{L^\infty} &\lesssim \|\tilde{\psi}(s)\|_{L^\infty} + \|\psi_\xx(s)\|_{L^\infty} \lesssim \frac{\eta_1(t)}{\sqrt{1+s}}
\end{split}
\label{e:vbound} \end{align}
for $s \in [0,t]$.

\paragraph*{Application of nonlinear damping estimate.} Take $t \in [0,\tau_{\max})$ such that $\eta(t) \leq \frac12 \min\{1,r_0\}$. Then, we have $t < \widetilde{\tau}_{\max}$ by Corollary~\ref{c:local_forward}. Moreover, using identity~\eqref{e:defvforw}, $\eta(t) \leq \frac12 \min\{1,r_0\}$ and the fact that $\phi \colon [1-r_0,1+r_0] \times \R \to \R^2$ is smooth by Proposition~\ref{prop:family}, we find a $t$- and $E_0$-independent constant $R_0 > 0$ such that $\smash{\|\mathring{\zt}(\tau)\|_{C_{\mathrm{ub}}^1}} \leq R_0$ for $\tau \in [0,t]$. On the other hand, Lemma~\ref{lem:equivalence} implies
\begin{align*}
\big\|\mathring{\zt}(\tau)\big\|_{L^\infty} \lesssim \eta_1(t) \frac{\log(2+\tau)}{1+\tau}
\end{align*}
for $\tau \in [0,t]$, where we use that $\eta_1(t) \leq \frac12$. Hence, employing the nonlinear damping estimate in Proposition~\ref{p:damping}, while using $\eta_1(t) \leq \frac12$ and $\smash{\|\mathring{\zt}(\tau)\|_{C_{\mathrm{ub}}^1} \leq R_0}$ for $\tau \in [0,t]$, we arrive at
\begin{align} \label{e:duha1}
\begin{split}
\big\|\mathring{\zt}(s)\big\|_{C_{\mathrm{ub}}^2 \times C_{\mathrm{ub}}^1} 
&\lesssim \eta_1(t) \frac{\log(2+s)}{1+s} + \left(\eta_1(t) \frac{\log(2+s)}{1+s}\right)^{\frac15}\left(\re^{-\alpha s} + \int_0^s \frac{\log(2+\tau)^2}{\re^{\alpha(s-\tau)}(1+\tau)^2} \,\de \tau\right)^{\frac25}\\
&\lesssim \eta_1(t)^{\frac15} \frac{\log(2+s)}{1+s}
\end{split}
\end{align}
for $s \in [0,t]$. We combine the latter with Lemma~\ref{lem:equivalence} and use $\eta_1(t) \leq \frac12$ to obtain
\begin{align} \label{e:duha2}
\big\|\zt(s)\big\|_{C_{\mathrm{ub}}^2 \times C_{\mathrm{ub}}^1} \lesssim \eta_1(t)^{\frac15} \frac{\log(2+s)}{1+s}
\end{align}
for $s \in [0,t]$. Therefore, recalling $\vb(s) = \zt(s) + \partial_k \phi(\cdot;1) \psi_\xx(s)$ and using $\eta_1(t) \leq \frac12$, the latter estimate yields
\begin{align}
\begin{split}
\|\vb(s)\|_{C_{\mathrm{ub}}^2 \times C_{\mathrm{ub}}^1} &\lesssim \|\zt(s)\|_{C_{\mathrm{ub}}^2 \times C_{\mathrm{ub}}^1} + \|\psi_\xx(s)\|_{C_{\mathrm{ub}}^2} \lesssim \frac{\eta_1(t)^{\frac15}}{\sqrt{1+s}}
\end{split}
\label{e:vbound2} \end{align}
for $s \in [0,t]$.

\paragraph*{Bounds on \texorpdfstring{$\zt(t)$}{z(t)}, \texorpdfstring{$\psi_{\xx\xx}(t)$}{psi_xx(t)} and \texorpdfstring{$\tilde{\psi}(t)$}{psi(t)}.} Let $t \in [0,\tau_{\max})$ be such that $\eta(t) \leq \frac12 \min\{1,r_0\}$. We invoke the nonlinear bound in Lemma~\ref{lem:nlboundsmod}, employ the estimates~\eqref{e:vbound} and~\eqref{e:vbound2}, and use $\eta_1(t) \leq \frac12$ to obtain
\begin{align} \label{e:nlest100}
\|\mathcal N(\vb(s),\psi(s),\partial_t \psi(s))\|_{L^\infty} &\lesssim \frac{\eta_1(t)^{\frac65}}{1+s}
\end{align}
for $s \in [0,t]$. 

Subsequently, we apply the linear estimates in Theorem~\ref{t:linear} and the nonlinear estimate~\eqref{e:nlest100} to the Duhamel formulas~\eqref{e:intpsi} and~\eqref{e:intz} and establish
\begin{align} \label{e:nlest1}
\begin{split}
\|\zt(t)\|_{L^\infty} &\lesssim \left(\frac{1}{1+t} + \re^{-\alpha t}\right)E_0 + \int_0^t \left(\frac{1}{1+t-s} + \re^{-\alpha(t-s)}\right) \frac{\eta_1(t)^{\frac65}}{1+s} \,\de s + \frac{\eta_1(t)^2}{1+t}\\ 
&\lesssim \left(E_0 + \eta_1(t)^{\frac65}\right)\frac{\log(2+t)}{1+t}
\end{split}
\end{align}
and
\begin{align} \label{e:nlest2}
\begin{split}
\left\|(\partial_t + c_g\partial_\xx)^j \partial_\xx^l \psi(t)\right\|_{L^\infty} \lesssim \frac{E_0}{1+t} + \int_0^t \frac{\eta_1(t)^{\frac65}}{(1+t-s)(1+s)} \,\de s \lesssim \left(E_0 + \eta_1(t)^{\frac65}\right)\frac{\log(2+t)}{1+t},
\end{split}
\end{align}
for all $t \in [0,\tau_{\max})$ with $\eta(t) \leq \frac{1}{2}\min\{1,r_0\}$ and $j,l \in \mathbb{N}_0$ with $2 \leq l + 2j \leq 6$, where we used $S_p(0) = 0$ when taking the temporal derivative of~\eqref{e:intpsi}. 

\paragraph*{Bounds on \texorpdfstring{$r(t)$}{r(t)} and \texorpdfstring{$r_\xx(t)$}{r_x(t)}.} Let $t \in [0,\tau_{\max})$ with $\eta(t) \leq \frac{1}{2}\min\{1,r_0\}$. We employ the nonlinear bound in Lemma~\ref{lem:nlboundsmod3} and estimates~\eqref{e:vbound},~\eqref{e:duha2} and~\eqref{e:vbound2} to establish
\begin{align}
\label{e:nlest71}
\begin{split}
&\|\mathcal N_p(\zt(s),\vb(s),\psi(s),\tilde{\psi})\|_{L^\infty}
\lesssim \frac{\log(2+s)}{(1+s)^{\frac{3}{2}}}\eta_1(t)^{\frac65},
\end{split}
\end{align}
for $s \in [0,t]$, where we used $\eta_1(t) \leq \frac12$.

We recall the well-known $L^\infty$-estimates on the convective heat semigroup:
\begin{align} \label{e:linheat}
\left\|\partial_\xx^m \re^{(d\partial_\xx^2 - c_g \partial_\xx) \tau} z\right\|_{L^\infty} \lesssim \tau^{-\frac{m}{2}} \|z\|_{L^\infty}, \qquad \left\|\partial_\xx \re^{(d\partial_\xx^2 - c_g \partial_\xx) \tau} w\right\|_{L^\infty} \lesssim \frac{\|w\|_{C^1_{\mathrm{ub}}}}{\sqrt{1+\tau}}
\end{align}
for $m = 0,1$, $\tau > 0$, $z \in C_{\mathrm{ub}}(\R)$ and $w \in C_{\mathrm{ub}}^1(\R)$, cf.~\cite[Proposition~3.6]{BjoernMod}. So, using that $\partial_\xx$ commutes with $\smash{\re^{\left(d\partial_\xx^2 - c_g\partial_\xx\right) (t-s)}}$, we estimate
\begin{align} \label{e:nlest6}
\begin{split}
&\left\|\partial_\xx^2 \int_0^t \re^{\left(d\partial_\xx^2 - c_g\partial_\xx\right) (t-s)}\left(A_h(\mathbf{f}_p) \psi_\xx(s)^2\right) \,\de s\right\|_{L^\infty}\\
&\qquad \lesssim \int_{0}^{\max\{0,t-1\}} \frac{\eta_1(t)^2}{(t-s)(1+s)} \,\de s + \int_{\max\{0,t-1\}}^t \frac{\eta_1(t)^2}{\sqrt{t-s}(1+s)} \,\de s \lesssim \frac{\eta_1(t)^2\log(2+t)}{1+t},
\end{split}
\end{align}
for all $t \in [0,\tau_{\max})$. Thus, applying the linear estimates in~\eqref{e:linheat} and in Theorem~\ref{t:linear} and the nonlinear estimates~\eqref{e:nlest71} to~\eqref{e:intr}, we obtain the bounds
\begin{align} \label{e:nlest7}
\begin{split}
\|r(t)\|_{L^\infty} \lesssim \frac{E_0}{\sqrt{1+t}} + \int_0^t \frac{\eta_1(t)^{\frac65}}{\sqrt{t-s}(1+s)} \,\de s \lesssim \left(E_0 + \eta_1(t)^{\frac65}\right)\frac{\log(2+t)}{\sqrt{1+t}}
\end{split}
\end{align}
and, using~\eqref{e:nlest6},
\begin{align} \label{e:nlest78}
\begin{split}
\|r_\xx(t)\|_{L^\infty} &\lesssim \frac{E_0}{\sqrt{t}\sqrt{1+t}} + \int_0^t \frac{\eta_1(t)^{\frac65}}{\sqrt{t-s}\sqrt{1+t-s}(1+s)} \,\de s + \frac{\eta_1(t)^2\log(2+t)}{1+t}\\
&\lesssim \left(E_0 + \eta_1(t)^{\frac65}\right)\frac{\log(2+t)}{\sqrt{t}\sqrt{1+t}}
\end{split}
\end{align}
for all $t \in [0,\tau_{\max})$ with $\eta(t) \leq \frac{1}{2}\min\{1,r_0\}$.

\paragraph*{Bounds on \texorpdfstring{$y(t)$}{y(t)} and \texorpdfstring{$y_\xx(t)$}{y_x(t)}.} Applying the estimates~\eqref{e:nlest7} and~\eqref{e:nlest78} to~\eqref{e:ysmall}, we derive the short-time bound
\begin{align} \label{e:yshort}
t^{\frac{m}{2}} \|\partial_\xx^m y(t)\|_\infty \lesssim t^{\frac{m}{2}}\|\partial_\xx^m r(t)\|_\infty \lesssim E_0 + \eta_1(t)^{\frac65},
\end{align}
for $m = 0,1$ and all $t \in [0,\tau_{\max})$ with $t \leq 1$ and $\eta(t) \leq \frac{1}{2}\min\{1,r_0\}$. 

Next, take $t \in [0,\tau_{\max})$ with $t \geq 1$ and $\eta(t) \leq \frac{1}{2}\min\{1,r_0\}$. Using the nonlinear bound in Lemma~\ref{lem:nlboundsmod4} and the estimates~\eqref{e:vbound},~\eqref{e:duha2} and~\eqref{e:vbound2}, we infer
\begin{align} \label{e:nlest1000}
\|\Non_c(r(s),y(s),\zt(s),\vb(s),\psi(s),\tilde{\psi}(s))\|_{L^\infty} \lesssim \frac{\eta_1(t)^{\frac65} \log(2+s)}{\left(1+s\right)^{\frac{3}{2}}}
\end{align}
for $s \in [1,t]$, where we use $\eta_1(t) \leq \frac{1}{2}$. 

We apply the linear estimates~\eqref{e:linheat} and the nonlinear bound~\eqref{e:nlest1000} to the Duhamel formula~\eqref{e:inty} and use~\eqref{e:yshort} to establish
\begin{align*}
\begin{split}
\left\|\partial_\xx^m y(t)\right\|_{L^\infty} \leq \frac{\|y(1)\|_{C_{\mathrm{ub}}^m}}{(1+t)^{\frac{m}{2}}} + \int_1^t \frac{\eta_1(t)^{\frac65} \log(2+s)}{(t-s)^{\frac{m}{2}} (1+s)^{\frac{3}{2}}} \,\de s \lesssim \frac{E_0 + \eta_1(t)^{\frac65}}{(1+t)^{\frac{m}{2}}},
\end{split}
\end{align*}
for $m = 0,1$ and all $t \in [0,\tau_{\max})$ with $t \geq 1$ and $\eta(t) \leq \frac{1}{2} \min\{1,r_0\}$. Combining the latter with the short-time bound~\eqref{e:yshort}, we arrive at
\begin{align} \label{e:nlest9}
\begin{split}
t^{\frac{m}{2}}\|\partial_\xx^m y(t)\|_{L^\infty} \lesssim E_0 + \eta_1(t)^{\frac65},
\end{split}
\end{align}
for $m = 0,1$ and all $t \in [0,\tau_{\max})$ with $\eta(t) \leq \frac{1}{2}\min\{1,r_0\}$.

\paragraph*{Bounds on \texorpdfstring{$\psi(t)$}{psi(t)} and \texorpdfstring{$\psi_\xx(t)$}{psi_x(t)}.} We start by considering the case $\nu \neq 0$. Through~\eqref{e:defy} we can express $\psi(t)$ in terms of the residual $r(t)$ and the Cole-Hopf variable $y(t)$ as
\begin{align*} \psi(t) = r(t) + \frac{d}{\nu} \log(y(t) + 1),\end{align*}
with derivative
\begin{align*} \psi_\xx(t) = r_\xx(t) + \frac{d y_\xx(t)}{\nu (1+y(t))},\end{align*}
for $t \in (0,\tau_{\max})$. We emphasize that, as long as $\eta_1(t) \leq \frac{1}{2}$ and $\nu \neq 0$, the above expressions are well-defined. So, using $\|\partial_\xx^m \psi(t)\|_{L^\infty} \lesssim \|\partial_\xx^m r(t)\|_{L^\infty} + \|\partial_\xx^m y(t)\|_{L^\infty}$, employing the estimates~\eqref{e:nlest7},~\eqref{e:nlest78} and~\eqref{e:nlest9} and recalling the fact that $\psi(s)$ vanishes identically for $s \in [0,\tau_{\max})$ with $s \leq 1$ by Proposition~\ref{p:psi}, we establish
\begin{align} \label{e:nlest11} \|\partial_\xx^m \psi(t)\|_{L^\infty} \lesssim \frac{E_0 + \eta_1(t)^{\frac65}}{(1+t)^{\frac{m}{2}}},\end{align}
for $m = 0,1$ and $t \in [0,\tau_{\max})$ with $\eta(t) \leq \frac12\min\{1,r_0\}$. 

Next, we consider the case $\nu = 0$. Recalling that $\psi(s)$ vanishes for $s \in [0,1]$ by Proposition~\ref{p:psi}, we apply the linear estimates in~\eqref{e:linheat} and the nonlinear bounds~\eqref{e:nlest71},~\eqref{e:nlest7} and~\eqref{e:nlest78} to~\eqref{e:intpsi2}, and deduce
\begin{align*} 
\begin{split}
\|\psi(t)\|_{L^\infty} \lesssim \left(E_0 + \eta_1(t)^{\frac65}\right)\frac{\log(2+t)}{\sqrt{1+t}} + E_0 + \int_0^t \eta_1(t)^{\frac65} \frac{\log(2+s)}{(1+s)^{\frac32}} \,\de s \lesssim E_0 + \eta_1(t)^{\frac65}
\end{split}
\end{align*}
and
\begin{align*} 
\begin{split}
\|\psi_\xx(t)\|_{L^\infty} \lesssim \left(E_0 + \eta_1(t)^{\frac65}\right)\frac{\log(2+t)}{\sqrt{t} \sqrt{1+t}} + \frac{E_0}{\sqrt{t}} + \int_0^t \frac{\eta_1(t)^{\frac65} \log(2+s)}{\sqrt{t-s}(1+s)^{\frac32}} \,\de s \lesssim \frac{E_0 + \eta_1(t)^{\frac65}}{\sqrt{1+t}}
\end{split}
\end{align*}
for $t \in [0,\tau_{\max})$ with $\eta(t) \leq \frac12\min\{1,r_0\}$. That is,~\eqref{e:nlest11} also holds for $\nu = 0$.

\paragraph*{Bounds on \texorpdfstring{$\vt(t)$}{w(t)} and \texorpdfstring{$\mathring{\vb}(t)$}{w(t)}.}
Using~\eqref{e:defvforw}, applying the mean value theorem and recalling that $\phi$ is smooth, we bound the forward-modulated perturbation $\mathring{\vb}(t)$, defined by~\eqref{e:defforwregular}, as
\begin{align} \label{e:nltest91}
\begin{split}
\big\|\mathring{\vb}(t)\big\|_{L^\infty} &\lesssim \big\|\mathring{\zt}(t)\big\|_{L^\infty} + \sup_{\xx \in \R} \big\|\phi(a(\xx,t);a_\xx(\xx,t)) - \phi_0(a(\xx,t))\big\| \\
&\qquad + \, \sup_{\xx \in \R} \big\|\phi(a(\xx,t) + \psi(\xx,t)\psi_\xx(\xx,t);a_\xx(\xx,t)) - \phi(a(\xx,t);a_\xx(\xx,t))\big\| \\
&\lesssim \big\|\mathring{\zt}(t)\big\|_{L^\infty} + \|\psi_\xx(t)\|_{L^\infty} \sup_{|k-1| \leq r_0} \|\phi_k(\cdot;k)\|_{L^\infty}\\ 
&\qquad + \, \|\psi(t)\|_{L^\infty} \|\psi_\xx(t)\|_{L^\infty} \sup_{|k-1| \leq r_0} \|\phi_\xx(\cdot;k)\|_{L^\infty} \lesssim \frac{\eta_1(t)}{\sqrt{1+t}}
\end{split}
\end{align}
for all $t \in [0,\tau_{\max})$ with $\eta(t) \leq \frac{1}{2}\min\{1,r_0\}$, where we abbreviate $a(\xx,t) = \xx + \psi(\xx,t)$. Similarly, we establish
\begin{align*} 
\begin{split}
\big\|\partial_\xx^j \vt(t)\big\|_{L^\infty} 
&\lesssim \big\|\partial_\xx^j \mathring{\zt}(t)\big\|_{L^\infty} + \|\psi(t)\|_{L^\infty}\left(1 + \|\psi_\xx(t)\|_{L^\infty} \right)\sup_{|k-1| \leq r_0} \|\partial_\xx^j \phi_{\xx}(\cdot;k)\|_{L^\infty}\\ 
&\qquad + \, \|\psi_\xx(t)\|_{L^\infty} \sup_{|k-1| \leq r_0} \|\partial_\xx^j \phi_{k}(\cdot;k)\|_{L^\infty} + \|\partial_\xx^j \psi_{\xx}(t)\|_{L^\infty}
\end{split}
\end{align*}
for $j = 0,1$ and $t \in [0,\tau_{\max})$ with $\eta(t) \leq \frac{1}{2}\min\{1,r_0\}$. Hence, combining the latter with~\eqref{e:duha1} yields
\begin{align}
    \label{e:nlest99}
\begin{split}
\big\|\vt(t)\big\|_{L^\infty} \leq \eta_1(t), \qquad \big\|\vt(t)\big\|_{C_{\mathrm{ub}}^1} \leq \eta_1(t)^{\frac15}  
\end{split}
\end{align}
for $t \in [0,\tau_{\max})$ with $\eta(t) \leq \frac{1}{2}\min\{1,r_0\}$. 

\paragraph*{Proof of key inequality and estimates~\eqref{e:mtest10}-\eqref{e:mtest2}.} Take $t \in [0,\tau_{\max})$ such that $\eta(t) \leq \frac{1}{2}\min\{1,r_0\}$. By estimate~\eqref{e:nlest99} there exists a $t$- and $E_0$-independent constant $C_2 > 0$ such that
\begin{align} \label{e:eta1}
\eta_2(t) \leq C_2 \eta_1(t)^{\frac15}.
\end{align}
On the other hand, employing the estimates~\eqref{e:nlest1},~\eqref{e:nlest2},~\eqref{e:nlest7},~\eqref{e:nlest78},~\eqref{e:nlest9} and~\eqref{e:nlest11}, we establish a $t$- and $E_0$-independent constant $C_1 > 0$ such that
\begin{align} \label{e:eta2}
\eta_1(t) \leq C_1\left(E_0 + \eta_1(t)^{\frac65}\right).
\end{align}
Hence, combining~\eqref{e:eta1} and~\eqref{e:eta2} we acquire
\begin{align*}
\eta(t) = \eta_1(t) + \eta_2(t)^5 \leq \left(1+C_2^5\right)\eta_1(t) \leq C_1\left(1+C_2^5\right)\left(E_0 + \eta_1(t)^{\frac65}\right) \leq C_1\left(1+C_2^5\right) \left(E_0 + \eta(t)^{\frac65}\right).
\end{align*}
We conclude that there exists a $t$- and $E_0$-independent constant such that the key inequality~\eqref{e:etaest} holds for all $t \in [0,\tau_{\max})$ with $\eta(t) \leq \frac{1}{2}\min\{1,r_0\}$. As argued above, this implies, provided $E_0 \in (0,\epsilon_0)$, that $\tau_{\max} = \infty$ and we have $\eta(t) \leq M_0 E_0$ for all $t \geq 0$. The estimates~\eqref{e:mtest10},~\eqref{e:mtest11} and~\eqref{e:mtest12} now follow directly by combining $\eta_1(t) \leq M_0 E_0$ with~\eqref{e:nltest91} and~\eqref{e:nlest99}, respectively. In addition, $\eta_1(t) \leq M_0E_0$ and~\eqref{e:vbound} yield the estimate~\eqref{e:mtest2}. 

\paragraph*{Approximation by the viscous Hamilton-Jacobi equation.} All that remains is to establish the approximation~\eqref{e:mtest3}. We proceed as in~\cite{BjoernMod} and distinguish between the cases $\nu = 0$ and $\nu \neq 0$. We start with the case $\nu = 0$. Then,~\eqref{e:HamJac} is a linear convective heat equation. We consider the classical solution $\breve \psi \in  C\big([0,\infty),C_{\mathrm{ub}}^2(\R)\big) \cap C^1\big([0,\infty),C_{\mathrm{ub}}(\R)\big)$ of~\eqref{e:HamJac} with initial condition $\smash{\breve \psi(0)} = \smash{\widetilde{\Phi}_0^* \vb_0} \in \smash{C_{\mathrm{ub}}^2(\R)}$ given by $\smash{\breve \psi(t)} = \smash{\re^{(d\partial_\xx^2 - c_g\partial_\xx)t} \widetilde{\Phi}_0^* \vb_0}$. Recalling that $\psi(t)$ vanishes identically for $t \in [0,1]$ by Proposition~\ref{p:psi}, we obtain by~\eqref{e:linheat} a $t$- and $E_0$-independent constant $M_1 \geq 1$ such that
\begin{align} \label{e:mtest200}
\begin{split}
t^{\frac{m}{2}}\left\|\partial_\xx^m \left(\psi(t) - \breve{\psi}(t)\right)\right\|_{L^\infty} = t^{\frac{m}{2}}\left\|\partial_\xx^m \breve{\psi}(t)\right\|_{L^\infty} \leq \frac{M_1E_0}{\sqrt{1+t}}, 
\end{split}
\end{align}
holds for $t \in [0,1]$ and $m = 0,1$. For $t \geq 1$, we apply the linear estimates in~\eqref{e:linheat} and the nonlinear bounds~\eqref{e:nlest71} and $\eta_1(t) \leq M_0 E_0$ to~\eqref{e:intpsi2} to establish $t$- and $E_0$-independent constants $M_2,M_3 \geq 1$ such that
\begin{align} \label{e:mtest21}
\begin{split}
\left\|\partial_\xx^m\!\left(\psi(t) - \breve{\psi}(t)\right)\right\|_{L^\infty} &\leq M_2\left(\left\|\partial_\xx^m r(t)\right\|_{L^\infty} +\int_0^t \eta_1(t)^{\frac65} \frac{\log(2+s)}{(t-s)^{\frac{m}{2}}(1+s)^{\frac32}} \,\de s\right)\\ 
&\leq M_3\frac{\eta_1(t)}{(1+t)^{\frac{m}{2}}} \left(\eta_1(t)^{\frac15} + \frac{\log(2+t)}{\sqrt{1+t}}\right), 
\end{split}
\end{align}
holds for $m = 0,1$. Estimate~\eqref{e:mtest3} now follows by combining~\eqref{e:mtest200} and~\eqref{e:mtest21} and using $\eta_1(t) \leq M_0E_0$. 

Next, we take $\nu \neq 0$. We consider the solution $\breve{\psi} \in C\big([0,\infty),C_{\mathrm{ub}}^2(\R)\big) \cap C^1\big([0,\infty),C_{\mathrm{ub}}(\R)\big)$ of~\eqref{e:HamJac} with initial condition $\smash{\breve \psi(0)} = \smash{\widetilde{\Phi}_0^* \vb_0}$ given by
\begin{align*} \breve \psi(t) = \frac{d}{\nu}\log\left(1 + \breve{y}(t)\right) \quad \text{with} \quad \breve{y}(t) = \re^{(d\partial_\xx^2 - c_g\partial_\xx)t}\left(\re^{\frac{\nu}{d}\widetilde{\Phi}_0^* \vb_0} - 1\right),\end{align*}
which arises through the Cole-Hopf transform and is well-defined as long as $E_0 = \|\vb_0\|_{L^\infty}$ is sufficiently small. Employing Taylor's theorem, Theorem~\ref{t:linear}, identities~\eqref{e:intr} and~\eqref{e:ysmall}, and estimates~\eqref{e:nlest71},~\eqref{e:linheat} and $\eta_1(1) \leq M_0E_0$, while using that $0 = S_p(1)\vb_0 = \smash{\re^{d\partial_\xx^2 - c_g\partial_\xx} \widetilde{\Phi}_0^* \vb_0 + \widetilde{S}_r(1)\vb_0}$ holds by Theorem~\ref{t:linear}, we establish an $E_0$-independent constant $M_4 > 0$ such that
\begin{align} \label{e:mtest23}
\begin{split}
\|y(1) - \breve{y}(1)\|_{L^\infty} &\leq \left\|y(1) + \tfrac{\nu}{d} r(1)\right\|_{L^\infty} + \left\|\breve{y}(1) - \tfrac{\nu}{d} \re^{d\partial_\xx^2 - c_g\partial_\xx} \widetilde{\Phi}_0^* \vb_0 \right\|_{L^\infty}\\ 
&\qquad + \, \tfrac{|\nu|}{d} \left\|r(1) - \widetilde{S}_r(1)\vb_0\right\|_{L^\infty} \leq M_4E_0^{\frac65}.
\end{split}
\end{align}
Noting that $\breve{y}(t) = \smash{\re^{(d \partial_\xx^2 - c_g \partial_\xx) (t-1)} \breve{y}(1)}$, applying the mean value theorem to~\eqref{e:defy}, employing the estimates~\eqref{e:linheat} and~\eqref{e:nlest1000} to~\eqref{e:inty}, and using~\eqref{e:mtest23} and $\eta_1(t) \leq M_0E_0$, we establish
\begin{align*} 
\begin{split}
\left\|\psi(t) - \breve{\psi}(t)\right\|_{L^\infty} &\lesssim \|r(t)\|_{L^\infty} + \left\|y(t) - \breve{y}(t)\right\|_{L^\infty} \lesssim \left\|r(t)\right\|_{L^\infty} + E_0^{\frac65} + \eta_1(t)^{\frac65},\\
\left\|\psi_\xx(t) - \breve{\psi}_\xx(t)\right\|_{L^\infty} &\lesssim \|r_\xx(t)\|_{L^\infty} + \left\|y_\xx(t) - \breve{y}_\xx(t)\right\|_{L^\infty} + \left\|y(t) - \breve{y}(t)\right\|_{L^\infty}\left\|y_\xx(t)\right\|_{L^\infty}
\\ &\lesssim \left\|r_\xx(t)\right\|_{L^\infty} + \frac{E_0^{\frac65} + \eta_1(t)^{\frac65}}{\sqrt{1+t}}
\end{split}
\end{align*}
for $t \geq 1$. So, using that $\eta_1(t) \leq M_0E_0$, affords a $t$- and $E_0$-independent constant $M_5 > 0$ such that 
\begin{align*} 
\begin{split}
\left\|\partial_\xx^m\!\left(\psi(t) - \breve{\psi}(t)\right)\right\|_{L^\infty} &\leq M_5\frac{E_0}{(1+t)^{\frac{m}{2}}} \left(E_0^{\frac15} + \frac{\log(2+t)}{\sqrt{1+t}}\right), 
\end{split}
\end{align*}
holds for all $t \geq 1$. On the other hand, we establish~\eqref{e:mtest200} for $t \in [0,1]$ analogously to the case $\nu = 0$. Thus, we obtain~\eqref{e:mtest3} for $\nu \neq 0$.
\end{proof}

\begin{remark} \label{rem:motv1}
Due to the use of forward-modulated damping in the proof of Theorem~\ref{t:mainresult}, it is, in contrast to~\cite{BjoernMod}, not necessary to control derivatives of $\zt(t)$ or $\vt(t)$ through iterative estimates on their Duhamel formulas. That is, we find that the template function $\eta_1(t)$ in the proof of Theorem~\ref{t:mainresult} coincides with the one from~\cite[Theorem~1.3]{BjoernMod}, upon omitting all derivatives of $\zt(t)$ and $\vt(t)$. Nevertheless, in order to apply the nonlinear damping estimate in Proposition~\ref{p:damping}, the condition~\eqref{e:upbound} needs to be fulfilled, which requires control on the first derivative of the (forward-modulated) perturbation. For that reason, we introduce the second template function $\eta_2(t)$ yielding a priori control on the $C_{\mathrm{ub}}^1$-norm of $\vt(t)$ and, thus, via~\eqref{e:defvforw} of $\mathring{\zt}(t)$. We can then a posteriori bound $\eta_2(t)^5$ with aid of the nonlinear damping estimate in terms of $\eta_1(t)$. Since $\eta_1(t)$ obeys the nonlinear key inequality~\eqref{e:eta2}, the same then follows for the full template function $\eta(t) = \eta_1(t) + \eta_2(t)^5$.
\end{remark}

\begin{remark} \label{rem:motv2}
The choice of temporal weights in the template function $\eta(t)$ in the proof of Theorem~\ref{t:mainresult} coincides with the one from the proof of~\cite[Theorem~1.3]{BjoernMod} and reflects, as explained in~\cite[Remark~5.1]{BjoernMod}, the linear decay rates of $\zt(t)$, $\psi(t)$, $y(t)$, $\vt(t)$ and $r(t)$, cf.~Theorem~\ref{t:linear} and~\eqref{e:linheat}, up to a logarithmic correction.
\end{remark}

\section{Discussion and outlook} \label{sec:disc}

We discuss the wider applicability of our method to establish nonlinear stability of wave trains against fully nonlocalized perturbations. 

\subsection{Applicability to general semilinear dissipative problems}

Our analysis does not rely on the specific structure of the FHN system. As a matter of fact, our approach only requires that the wave train is diffusively spectrally stable, it has nonzero group velocity, the perturbation equation obeys a nonlinear damping estimate and the linearization of the system about the wave train generates a $C_0$-semigroup on $C_{\mathrm{ub}}(\R)$, whose high-frequency component is exponentially damped. As long as these criteria are satisfied, we expect our method to work for general semilinear dissipative problems.

It was already observed in~\cite{BjoernAvery} that the same linear terms in the FHN system~\eqref{FHN}, i.e.~the term $u_{xx}$ in the first component and the term $-\varepsilon \gamma v$ in the second component, are key to obtaining a nonlinear damping estimate, as well as high-frequency resolvent bounds leading to exponentially damped behavior of the high-frequency part of the semigroup. It has been pointed out in the context of the St.~Venant equations in~\cite{STVenant2} that high-frequency resolvent bounds are equivalent to linear damping estimates, which then yield a nonlinear damping estimate as long as solutions stay small. Therefore, we expect that we can replace the requirements that the high-frequency component of the semigroup is exponentially damped and a nonlinear damping estimate can be derived by the condition that  the linearization obeys high-frequency resolvent bounds.

In addition, we expect that it is possible to drop the requirement that the wave train has nonzero group velocity. In the case of zero group velocity the diffusive mode at the origin is \emph{branched}, cf.~\cite[Section~2.1]{BjoernAvery}, i.e., the linear dispersion relation $\lambda_c(\xi)$ has a double root at $\xi = 0$. The fact that the linear dispersion relation $\lambda_c(\xi)$ is no longer locally invertible about $\xi = 0$ poses a technical hurdle in relating the inverse Laplace representation of the low-frequency part of the semigroup to its Floquet-Bloch representation. We anticipate that this challenge can be addressed by unfolding the double root at $0$ by working with the spectral parameter $\sigma = \sqrt{\lambda}$ with branch cut along the negative real axis.

\subsection{Open problems}

There are however several prominent examples of semilinear dissipative systems, where nonlinear stability of wave trains against localized perturbations has been established, but where one (or more) of the above requirements are not satisfied, thereby obstructing a straightforward application of our method to extend to fully nonlocalized perturbations. Here, we highlight two of these examples.

The first is the Lugiato-Lefever equation, a damped and forced nonlinear Schr\"odinger equation arising in nonlinear optics, whose diffusively spectrally stable periodic waves are nonlinear stable against localized perturbations~\cite{LLEperiod}. Here, the principal part of the linearization about the wave is the Schr\"odinger operator $\ri \partial_x^2$, which does not generate a $C_0$-(semi)group on $C_{\mathrm{ub}}(\R)$, cf.~\cite[Lemma~2.1]{BPSS}. Thus, an extension of our method to this setting necessitates reconsidering the choice of space. Natural candidates are the \emph{modulation spaces} $M_{\infty,1}^k(\R)$ on which the Schr\"odinger operator generates a $C_0$-group, cf.~\cite[Proposition~3.8]{KUNST}. These spaces consist of nonlocalized functions as can be seen from the embeddings $C_{\mathrm{ub}}^{k+2}(\R) \hookrightarrow M_{\infty,1}^k(\R) \hookrightarrow C_{\mathrm{b}}^k(\R)$ for $k \in \mathbb{N}_0$, cf.~\cite[Theorem~5.7 and Lemma~5.9]{KLAUS}. An application of our method would then require to establish high-frequency damping in modulation spaces, which could be challenging. We refer to~\cite{GROE} for further background on modulation spaces.

A second example are the St.~Venant equations, which describe shallow water flow down an inclined ramp and admit viscous roll waves. Nonlinear stability of these periodic traveling waves against localized perturbations has been established in~\cite{STVenant1,STVenant2}. The St.~Venant system is only viscous in one component and therefore, similar to the current analysis for the FHN system, incomplete parabolicity must be addressed. Moreover, due to the presence of an additional conservation law the spectrum of the linearization about the wave train possesses an additional curve touching the imaginary axis at $0$, thereby violating the spectral stability assumption~\ref{assD3}. Thus, the leading-order dynamics of perturbations is no longer governed by the scalar viscous Hamilton-Jacobi equation~\eqref{e:HamJac}, but instead by an associated Whitham system describing the interactions between critical modes, cf.~\cite{johnson_whitham}. It is an open question of how to handle the most critical nonlinear terms that cannot be controlled through iterative $L^\infty$-estimates on the Duhamel formula as the Cole-Hopf transform is no longer available. However, motivated by the results in~\cite{HDRS22} on the dynamics of roll waves in the Ginzburg-Landau equation coupled to a conservation law against $C_{\mathrm{ub}}$-perturbations, we do expect that our method yields control of perturbations on exponentially long time scales in the setting of the St.~Venant equations and more general semilinear dissipative systems admitting conservation laws.

\appendix

\section{The Laplace transform and its complex inversion formula} \label{sec:laplace}

This section is devoted to background material on the vector-valued Laplace transforms. In particular, we prove that the complex inversion formula holds for the Laplace transform of convolutions of semigroups. For an extensive introduction into the topic, we refer to the book~\cite{arendt} of Arendt, Batty, Hieber and Neubrander. 

Let $X,Y$ be complex Banach spaces. We denote by $B(X)$ the space of bounded operators mapping from $X$ into $X$. The \emph{growth bound} $\omega_0(G)$ of a map $G \colon [0,\infty) \to Y$ is given by
\begin{align*}
\omega_0(G) = \inf\left\{\omega \in \R : \sup_{t \geq 0} \re^{-\omega t} \|G(t)\| < \infty \right\}.
\end{align*}
If $\omega_0(G) < \infty$, then we say that $G$ is \emph{exponentially bounded}. 

For a continuous and exponentially bounded function $F \colon [0,\infty) \to X$, the \emph{Laplace transform} $\mathfrak{L}(F) \colon \{\lambda \in \C : \Re(\lambda) > \omega_0(F)\} \to X$ is given by
\begin{align*}
\mathfrak{L}(F)(\lambda) = \int_0^\infty \re^{-\lambda s} F(s) \,\de s.
\end{align*}

Strong continuity of an operator-valued map $T \colon [0,\infty) \to B(X)$ entails that for each $x \in X$ the \emph{orbit map} $T_x \colon [0,\infty) \to X$ given by $T_x(t) = T(t) x$ is continuous. For a strongly continuous and exponentially bounded $T \colon [0,\infty) \to B(X)$, the Laplace transform $\mathfrak{L}(T) \colon \{\lambda \in \C : \Re(\lambda) > \omega_0(T)\} \to B(X)$, given by
\begin{align*}
\mathfrak{L}(T)(\lambda) = \int_0^\infty \re^{-\lambda s} T(s) \,\de s,
\end{align*}
is also well-defined by~\cite[Proposition~1.4.5]{arendt}. For a $C_0$-semigroup $T \colon [0,\infty) \to B(X)$ with infinitesimal generator $A \colon D(A) \subset X \to X$, it is well-known, by~\cite[Proposition~I.5.5 \& Theorem~II.1.10]{nagel}, that $T$ is exponentially bounded and its Laplace transform is given by the resolvent $\mathfrak{L}(T)(\lambda) = (\lambda - A)^{-1}$ for $\Re(\lambda) > \omega_0(T)$. 

Let $S, T \colon [0,\infty) \to B(X)$ be strongly continuous and exponentially bounded. The \emph{convolution} $S \ast T \colon [0,\infty) \to B(X)$ of $S$ and $T$ is given by
\begin{align*}
    (S \ast T)(t) = \int_0^t S(s) T(t-s) \,\de s.
\end{align*}
The convolution theorem, cf.~\cite[Theorem~C.17]{nagel}, now states that $S \ast T$ is also strongly continuous and exponentially bounded with $\omega_0(S \ast T) \leq \max\{\omega_0(S),\omega_0(T)\}$ and its Laplace transform obeys
\begin{align} \label{productformula}
    \mathfrak{L}(S \ast T)(\lambda) = \mathfrak{L}(S)(\lambda) \mathfrak{L}(T)(\lambda),
\end{align}
for $\lambda \in \C$ with $\Re(\lambda) > \max\{\omega_0(S),\omega_0(T)\}$. 

The complex inversion formula of the Laplace transform holds for $C_0$-semigroups. That is, if $T$ is a $C_0$-semigroup with infinitesimal operator $A$, then we have
\begin{align*} 
    T(t) x = \lim_{R \to \infty} \frac{1}{2\pi \ri} \int_{\omega - \ri R}^{\omega + \ri R} \re^{\lambda t} \mathfrak{L}(T)(\lambda) x \,\de \lambda = \lim_{R \to \infty} \frac{1}{2\pi \ri} \int_{\omega - \ri R}^{\omega + \ri R} \re^{\lambda t} (\lambda - A)^{-1} x \, \de \lambda
\end{align*}
for all $t > 0$, $\omega > \omega_0(T)$ and $x \in D(A)$, cf.~\cite[Proposition 3.12.1]{arendt}. 

In Section~\ref{sec_lin}, we decompose the $C_0$-semigroup generated by the linearization $\El_0$ by deforming and partitioning the integration contour of the complex line integral in the inversion formula, alongside decomposing the resolvent operator. It has been shown in~\cite{BjoernAvery} that for high frequencies the resolvent can be expanded as a Neumann series, whose leading-order terms can be identified as products of resolvents of simpler, well-understood operators. The formula~\eqref{productformula} reveals that such products can be recognized as the Laplace transform of a convolution of $C_0$-semigroups generated by those simpler operators. Indeed, if $T$ and $S$ are $C_0$-semigroups with infinitesimal operators $A \colon D(A) \subset X \to X$ and $B \colon D(B) \subset X \to X$, respectively, then~\eqref{productformula} and~\cite[Theorem~II.1.10]{nagel} yield
\begin{align*}
    \mathfrak{L}(S \ast T)(\lambda) = (\lambda - B)^{-1}(\lambda - A)^{-1},
\end{align*}
for $\lambda \in \C$ with $\Re(\lambda) > \max\{\omega_0(S),\omega_0(T)\}$. Thus, to bound the contour integrals arising in the decomposition of the inverse Laplace transform of the $C_0$-semigroup $\re^{\El_0 t}$, we wish to show that the inversion formula of the Laplace transform also holds for \emph{convolutions} of $C_0$-semigroups. As far as we are aware, such a result is not readily stated in the current literature. Therefore, we provide a proof in the upcoming. Our proof relies on the observation that the inversion formula holds for $F$ as long as it is Lipschitz continuous and $F(0) = 0$. 

\begin{proposition}
\label{know_result}
Let $X$ be a complex Banach space. Let $F \colon [0,\infty) \rightarrow X$ be Lipschitz continuous. Assume $F(0) = 0$. Then, the complex inversion formula
\begin{align*}
F(t) = \lim_{R\rightarrow \infty} \frac{1}{2\pi \ri} \int_{\omega - \ri R}^{\omega+\ri R} \re^{\lambda t} \mathfrak{L}(F)(\lambda) \, \de \lambda
\end{align*}
holds for $t > 0$ and $\omega > 0$.
\end{proposition}
\begin{proof}
Since $F$ is Lipschitz continuous, it grows at most linearly and is therefore exponentially bounded with growth bound $\omega_0(F) \leq 0$. Let $t > 0$ and $\omega > 0$. By~\cite[Theorem 2.3.4]{arendt}, we have
\begin{align*}
    F(t) = \lim_{R\rightarrow \infty} \frac{1}{2\pi \ri} \int_{\omega - \ri R}^{\omega+\ri R} \re^{\lambda t} \frac{r(\lambda)}{\lambda} \, \de \lambda,
\end{align*}
where the analytic function $r \colon \{\lambda \in \C : \Re(\lambda) > 0\} \to X$ given by
\begin{align*}
r(\lambda) = \int_0^\infty \re^{-\lambda s} \,\de F(s)
\end{align*}
is the Laplace-Stieltjes transform of $F$, cf.~\cite[Theorem~1.10.6]{arendt}. We integrate by parts, cf.~\cite[Formula~(1.20)]{arendt}, and arrive at
\begin{align*}
\frac{r(\lambda)}{\lambda} &= \lim_{t \to \infty} \int_0^t \frac{\re^{-\lambda s}}{\lambda} \de F(s) = \lim_{t \to \infty} \frac{1}{\lambda}\left(\re^{-\lambda t} F(t) - F(0) - \int_0^t F(s) \,\de \left(\re^{-\lambda s}\right) \right) \\
&= \int_0^\infty F(s)\re^{-\lambda s} \,\de s = \mathfrak{L}(F)(\lambda)
\end{align*}
for $\lambda \in \C$ with $\Re(\lambda) > 0$, which proves the claim.
\end{proof}

The fact that the complex inversion formula of the Laplace transform holds for convolutions of $C_0$-semigroups is now a direct consequence of Proposition~\ref{know_result}.

\begin{corollary} \label{convolution_semigroup}
Let $X$ be a complex Banach space. Let $T,S \colon [0,\infty) \rightarrow \mathcal{L}(X)$ be $C_0$-semigroups with infinitesimal generators $A \colon D(A) \subset X \to X$ and $B \colon D(B) \colon X \to X$, respectively. Then, we have
\begin{align*}
(S \ast T)(t)x =  \frac{1}{ 2\pi \ri}\lim_{R \rightarrow\infty} \int^{\omega + \ri R}_{\omega - \ri R} \re^{\lambda t} (\lambda- B)^{-1}(\lambda- A)^{-1}x \,\de \lambda 
\end{align*}
for $t > 0$, $x \in D(A)$ and $\omega > \max\{\omega_0(S),\omega_0(T)\}$.
\end{corollary}
\begin{proof}
Let $t > 0$, $x \in D(A)$ and $\omega > \max\{\omega_0(S),\omega_0(T)\}$. Take $\max\{\omega_0(S),\omega_0(T)\} < \alpha < \omega$. The rescaled semigroups $\tilde{T}(s) = e^{-\alpha s} T(s)$ and $\tilde{S}(s) = e^{-\alpha s} S(s)$ are generated by $A-\alpha$ and $B-\alpha$, respectively. Moreover, $\tilde{S}$ and $\tilde{T}$ have negative growth bounds $\omega_0(\tilde{S}) = \omega_0(S) - \alpha$ and $\omega_0(\tilde{T}) = \omega_0(T) - \alpha$ and so has their convolution $\tilde{S} \ast \tilde{T}$. 

Since we have $x \in D(A)$, the map $F \colon [0,\infty) \to X$ given by $F(s) = (\tilde{S} \ast \tilde{T})(s) x$ is differentiable with 
\begin{align*}
F'(s) = (\tilde{S}\ast \tilde{T})(s)(Ax- \alpha x) + \tilde{S}(s)x.
\end{align*}
Thanks to the fact that $\tilde{S} \ast \tilde{T}$ and $\tilde{S}$ have negative growth bound, there exists a constant $M > 0$ such that $\|F'(s)\| \leq M(\|Ax\| + \|x\|)$ for all $s \geq 0$. Hence, using the mean value theorem, cf.~\cite[Proposition 1.2.3]{arendt}, we infer $\|F(s)- F(r)\| \leq M (\|Ax\| + \|x\|) |s-r|$, showing that $F$ is Lipschitz continuous. Since we have in addition $F(0) = 0$, an application of Proposition~\ref{know_result} yields
\begin{align*}
(\tilde{S} \ast \tilde{T})(t) x = F(t) = \lim_{R\rightarrow \infty} \frac{1}{2\pi \ri} \int_{\tilde\omega - \ri R}^{\tilde\omega+\ri R} \re^{\lambda t} \mathfrak{L}(F)(\lambda) \, \de \lambda,
\end{align*}
where we denote $\tilde\omega = \omega - \alpha > 0$. On the other hand, with the aid of~\cite[Theorems~II.1.10 and~C.17]{nagel}, we compute
\begin{align*}
    \mathfrak{L}(F)(\lambda) =  \int_0^\infty \re^{-\lambda s} (\tilde{S} \ast \tilde{T})(s) x \,\de s = \mathfrak{L}(\tilde{S} \ast \tilde{T})(\lambda) x = (\lambda + \alpha - B)^{-1}(\lambda + \alpha - A)^{-1} x
\end{align*}
for $\lambda \in \C$ with $\Re(\lambda) > 0$. Therefore, pulling out the exponential factors and scaling back, we arrive at
\begin{align*}
(S \ast T)(t)x &= \re^{\alpha t} (\tilde{S} \ast \tilde{T})(t) x = 
\lim_{R\rightarrow \infty} \frac{1}{2\pi \ri} \int_{\omega - \alpha - \ri R}^{\omega - \alpha +\ri R} \re^{(\lambda + \alpha) t} (\lambda + \alpha - B)^{-1}(\lambda + \alpha - A)^{-1} x \, \de \lambda\\ 
&= \lim_{R\rightarrow \infty} \frac{1}{2\pi \ri} \int_{\omega - \ri R}^{\omega+\ri R} \re^{\lambda t} (\lambda - B)^{-1}(\lambda  - A)^{-1} x \, \de \lambda,
\end{align*}
which finishes the proof.
\end{proof}

\section{Derivation of equation for the modified forward-modulated perturbation} \label{app:derivationforward}

Assume~\ref{assH1} and~\ref{assD3}. Let $t \in [0,\widetilde{\tau}_{\max})$. Recalling Proposition~\ref{prop:family} and noting that $\|\psi_\xx(t)\|_{L^\infty} < r_0$, we substitute $k = 1+\psi_\xx(\xx,t)$ and $y = \xx + \psi(\xx;t)(1+\psi_\xx(\xx;t))$ in the equation 
\begin{align*}
k^2 D \phi_{yy}(y;k) + \omega(k)\phi_y(y;k) + F(\phi(y;k)) = 0
\end{align*}
for the profile function $\phi(y;k)$ and arrive at
\begin{align} \label{e:profeq}
(1+\psi_\xx(\xx;t))^2 D \phi_{yy}(\beta(\xx,t)) + \omega(1+\psi_\xx(\xx,t))\phi_y(\beta(\xx,t)) + F(\phi(\beta(\xx,t))) = 0
\end{align}
for $\xx \in \R$, where we abbreviate $\beta(\xx,t) = \big(\xx + \psi(\xx;t)(1+\psi_\xx(\xx;t));1+\psi_\xx(\xx,t)\big)$. Using Corollary~\ref{c:local_forward} and the fact that $\ub(\xx,t)$ solves~\eqref{FHN_co}, we compute the temporal derivative
\begin{align}
\label{e:step1}
\begin{split}
\mathring{\zt}_t = D \ub_{\xx\xx} + c_0 \ub_{\xx} + F(\ub) - (\phi_y \circ \beta)\left(\psi_t (1+\psi_\xx) + \psi \psi_{\xx t}\right) - (\phi_k \circ \beta)\psi_{\xx t}. 
\end{split}
\end{align}
In an effort to reexpress the $\ub$-contributions in~\eqref{e:step1} in terms of $\mathring{\zt}$, we determining the spatial derivatives of $\ub(\xx,t) = \mathring{\zt}(\xx,t) + \phi(\beta(\xx,t))$ yielding
\begin{align*}
\ub_{\xx} &= \mathring{\zt}_\xx + (\phi_y \circ \beta)\left(1 + \psi_\xx (1+\psi_\xx) + \psi \psi_{\xx \xx}\right) + (\phi_k \circ \beta)\psi_{\xx \xx},\\
\ub_{\xx\xx} &= \mathring{\zt}_{\xx\xx} + (\phi_{yy} \circ \beta)\left(1 + \psi_\xx (1+\psi_\xx) + \psi \psi_{\xx \xx}\right)^2 + (\phi_y \circ \beta)\left(\psi_{\xx\xx} (1+3\psi_\xx) + \psi \psi_{\xx \xx \xx}\right)\\ &\qquad + \, (\phi_{kk} \circ \beta)\psi_{\xx \xx}^2 + (\phi_{k} \circ \beta)\psi_{\xx \xx\xx} + 2(\phi_{yk} \circ \beta)\left(1 + \psi_\xx (1+\psi_\xx) + \psi \psi_{\xx \xx}\right)\psi_{\xx\xx}.
\end{align*}
Thus, inserting $\ub(\xx,t) = \mathring{\zt}(\xx,t) + \phi(\beta(\xx,t))$ into~\eqref{e:step1} and employing~\eqref{e:profeq}, we arrive at the equation~\eqref{e:Pert1} for the modified forward-modulated perturbation.

\bibliographystyle{abbrv}
\bibliography{refs}
\end{document}